\documentclass{amsart}
\usepackage{youngtab}
\usepackage{graphicx,tikz, subfig}
\usepackage{mathrsfs,hyperref, amssymb}
\usepackage{cancel, fullpage}
\usepackage[cp1251]{inputenc}
\usepackage[all]{xy}
\usetikzlibrary{decorations.markings}

\newtheorem{theorem}{Theorem}[section]

\newtheorem{conjecture}[theorem]{Conjecture}

\newtheorem{lemma}[theorem]{Lemma}
\newtheorem{proposition}[theorem]{Proposition}
\newtheorem{corollary}[theorem]{Corollary}

\newtheorem*{maintheorem}{Theorem}

\theoremstyle{definition}
\newtheorem{definition}[theorem]{Definition}
\newtheorem{example}[theorem]{Example}

\newtheorem{remark}[theorem]{Remark}

\newcommand{\ten}{10}
\newcommand{\eleven}{11}
\newcommand{\twelve}{12}
\newcommand{\thirteen}{13}
\newcommand{\fourteen}{14}
\newcommand{\fifteen}{15}
\newcommand{\sixteen}{16}
\newcommand{\seventeen}{17}
\newcommand{\eighteen}{18}

\author{Heather M. Russell}
\address{Department of Mathematics and Computer Science, University of Richmond, 212 Jepson Hall, 221 Richmond Way, University of Richmond, VA 23173 U.S.A.}
\email{hrussell@richmond.edu}

\author{Julianna Tymoczko}
\address{Department of Mathematics and Statistics, Smith College, Northampton, MA 01063 U.S.A.}
\email{jtymoczko@smith.edu}

\thanks{HMR received support from an AWM Mentoring Travel Grant. JT received support from NSF-DMS-1362855 and NSF-DMS-1800773. The authors thank Mee Seong Im, Brendon Rhoades, and Jieru Zhu for helpful conversations on this work and Miles Clikeman for generating the data in Figure \ref{fig:counterdata}.}

\title{The transition matrix between the Specht and  $\mathfrak{sl}_3$ web bases is unitriangular with respect to shadow containment}

\begin{document}

\maketitle

\begin{abstract}
Webs are planar graphs with boundary that describe morphisms in a diagrammatic representation category for $\mathfrak{sl}_k$.  They are studied extensively by knot theorists because braiding maps provide a categorical way to express link diagrams in terms of webs, producing quantum invariants like the well-known Jones polynomial. One important question in representation theory is to identify the relationships between different bases; coefficients in the change-of-basis matrix often describe combinatorial, algebraic, or geometric quantities (like, e.g., Kazhdan-Lusztig polynomials).  By ``flattening" the braiding maps, webs can also be viewed as the basis elements of a symmetric-group representation.  

In this paper, we define two new combinatorial structures for webs: {\em band diagrams} and their one-dimensional projections, {\em shadows}, that measure depths of regions inside the web. As an application, we resolve an open conjecture that the change-of-basis between the so-called {\em Specht basis} and web basis of this symmetric-group representation is unitriangular for $\mathfrak{sl}_3$-webs (See \cite{MR3920353} and \cite{MR3998730}.) We do this using band diagrams and shadows to construct a new partial order on webs that is a refinement of the usual partial order. In fact, we prove that for $\mathfrak{sl}_2$-webs, our new partial order coincides with the tableau partial order on webs studied by the authors and others \cite{ImZhu, khovanov2004crossingless, MR3998730, MR3920353}. We also prove that though the new partial order for $\mathfrak{sl}_3$-webs is a refinement of the previously-studied tableau order, the two partial orders do not agree for $\mathfrak{sl}_3$.
\end{abstract}

\section{Introduction}

Webs are planar graphs that lie at the interface of knot theory, representation theory, and combinatorics.  From the knot-theoretic perspective, $\mathfrak{sl}_k$ webs are labeled plane graphs that act as morphisms in diagrammatic representation categories for the Lie algebras $\mathfrak{sl}_k$ and their associated quantum universal  enveloping algebras \cite{MR1403861}. Each tangle diagram can be interpreted as a linear combination of webs that is unique up to a writhe factor. When the tangle is a knot or link, web relations produce an integer (in the classical setting) or polynomial (in the quantum setting) \cite{MR1091619}. The polynomials coming from this process are known as quantum link invariants; the celebrated Jones polynomial is a quantum link invariant that arises in the $\mathfrak{sl}_2$ setting \cite{MR766964}. 

From the representation-theoretic point of view, webs encode maps between representations, and the graph-theoretic structure of webs corresponds to algebraic properties and operations on the maps. For instance, webs are often drawn in a rectangular region with boundary vertices at the top and bottom that indicate source and target spaces, each of which is a tensor product of fundamental representations; stacking webs corresponds to composition of maps. The internal structure corresponds to a particular intertwining of the Lie algebra action, and skein-theoretic operations on the web summarize algebraic operations on the representation \cite{MR2710589}.

We study webs whose boundary points all lie on the bottom of the rectangular region (so-called {\em bottom webs}) in the case of $\mathfrak{sl}_2$ and $\mathfrak{sl}_3$.  The space of all $\mathfrak{sl}_k$ bottom webs with fixed boundary can be identified with the $\mathfrak{sl}_k$-invariant subspace of the tensor product of fundamental representations dictated by that boundary data \cite{K}. In the $\mathfrak{sl}_2$ and $\mathfrak{sl}_3$ cases, web relations can be used in an algorithmic fashion to produce a basis of {\em reduced webs} for these invariant spaces. These reduced web bases have a rich combinatorial structure. For instance, the reduced $\mathfrak{sl}_2$ bottom webs are also known as Temperley-Lieb diagrams or crossingless matchings; they are enumerated by Catalan numbers and are ubiquitous throughout the mathematical literature \cite{jones1983index, kauffman1987state, lickorish1992calculations, RTW, MR498284}. The combinatorics of reduced $\mathfrak{sl}_3$ webs and implications in knot theory and representation theory are less well-understood and have emerged as fertile areas of ongoing work \cite{housley2015robinson, ImZhu, MR3205780, PPR, MR3119361, MR3273401, T}. For $\mathfrak{sl}_4$ and beyond, the web relations do not appear a priori to reduce the complexity of the webs, meaning there is no clear mechanism to decide if a web is reduced nor is there a closed-form characterization of a basis of reduced webs (though see work of Fontaine in this direction \cite{FON}).

One core question in combinatorial representation theory is {\bf how two bases for a representation are related}.  In symmetric function theory, the transition matrices between different bases have been studied extensively, in part because they encode important combinatorial quantities \cite{MR1354144}. Much less is understood about the web basis, so the current open questions ask fundamental questions like whether the web basis is the same or different from previously-known bases.  For instance, as bases for representations of $\mathfrak{sl}_k$, the web basis for $\mathfrak{sl}_2$ agrees with the dual canonical basis \cite{MR1446615}.  However, the two bases are different for $\mathfrak{sl}_3$ webs \cite{KK}. Tools from cluster algebras suggests that the web basis for $\mathfrak{sl}_3$ may instead align with the dual semicanonical basis \cite{MR3534844}, and very recent work develops a total order on terms of the basis that may help prove this conjecture \cite{10.1093/imrn/rnaa110}.

In the rest of this paper, we study the space of $\mathfrak{sl}_k$-invariants under the natural symmetric group action coming from permuting tensor factors. On the level of webs, this can be explicitly described using braiding on adjacent strands at the web boundary. From this diagrammatic viewpoint, the $\mathfrak{sl}_2$ and $\mathfrak{sl}_3$ web spaces equipped with this action have been proven to be irreducible symmetric group representations \cite{ImZhu, PPR, MR2801314}. More generally, one can use Schur-Weyl duality to decompose these invariant spaces as symmetric group representations \cite{MR1321638}.  

At the same time, invariant theory shows that the dimension of the space of $\mathfrak{sl}_2$ and $\mathfrak{sl}_3$ bottom webs is the number of rectangular standard Young tableaux with two and three rows, respectively \cite{MR0000255}.  There are several explicit 
bijections between webs and standard tableaux  \cite{FON, KK, MR3119361, T, MR2873098}, and standard tableaux index other classical bases for symmetric group representations.   Various authors proved that with respect to symmetric group representations, the $\mathfrak{sl}_k$ web basis coincides with the Kazhdan-Lusztig basis in the $k=2$ case \cite{frenkel1998kazhdan} but not for $k=3$ \cite{housley2015robinson} and likely not beyond.  

We focus on another classical basis called {\em the Specht} (or {\em Young's natural} or {\em the polytabloid}) basis \cite{MR1545531, MR0439548}. A naive guess is that the Specht and web bases coincide; however, this is false, in the sense that the natural bijection between webs and tableaux does not induce an $S_n$-equivariant map from the Specht module to web space \cite{MR2801314}. 
In a recent paper, we showed that in the case of $\mathfrak{sl}_2$ the change-of-basis matrix between these representations is unitriangular (with respect to a partial order inherited from the symmetric group action on tableaux), and also characterized some additional vanishing entries \cite{MR3920353}.  We conjectured that all entries were in fact nonnegative, which has since been proven by Rhoades \cite{MR3998730}.  We also conjectured that the vanishing entries we identified were the only zero terms in the transition matrix, a conjecture that was just proven by Im and Zhu \cite{ImZhu}. The third conjecture from that paper was that unitriangularity also holds for $\mathfrak{sl}_3$ webs.

In this paper, we prove the third conjecture, with respect to a different---and new---partial order on the basis webs.  The main results of this paper are as follows.  First, we construct a package of combinatorial structures associated to webs called {\em band diagrams} and {\em shadows}, including a partial order on the web basis called the shadow containment partial order.  Second, we compare the shadow containment partial order to the partial order inherited from tableaux.  One main result is the following.
\begin{maintheorem}
Let $\phi: V^{\mathcal{T}_{3n}} \rightarrow V^{\mathcal{W}_{3n}}$ be the transition matrix between the $S_{3n}$-representations with the Specht and web bases partially ordered according to the shadow containment order of Section \ref{sec:band}.  Then $\phi$ is unitriangular, namely upper-triangular with ones along the diagonal.  
\end{maintheorem}

The precise statement of the theorem can be found in Theorem \ref{thm:coefficients}.  Note one subtlety about this result: when it says these bases are ``partially ordered according to shadow containment order," it means that {\em any} total order consistent with the shadow containment partial order can be used in the result.  This is a strong condition that actually implies the vanishing of additional entries above the diagonal, as detailed in Theorem \ref{thm:coefficients}. 

En route to proving this result, we also show that a relationship in the tableau partial order implies the same relationship holds in the shadow containment partial order.  We can use that to resolve the conjecture about unitriangularity with respect to the tableau order, for {\em at least one} total order extending the tableau partial order.

\begin{maintheorem}
The shadow containment partial order is a refinement of the tableau partial order.  In particular, there exists a total order consistent with the tableau partial order with respect to which the transition matrix between the $S_{3n}$-representations with Specht basis and web basis is unitriangular.  
\end{maintheorem}

 Indeed, any total order extending the shadow containment partial order will also demonstrate unitriangularity with respect to the tableau partial order.  This is stated formally in Theorem \ref{thm:existtriang} and follows from Theorem \ref{cor:total}. In addition, we prove that the shadow containment partial order coincides with the partial order coming from tableaux on $\mathfrak{sl}_2$ webs (Theorem \ref{thm:samesl2}) but different on $\mathfrak{sl}_3$ webs (Theorem \ref{thm:notsamesl3}).  Despite this, we still believe that the transition matrix between Specht and web bases is unitriangular wtih respect to {\em any} total order extending the tableau partial order, as discussed in Conjecture \ref{conj:taborder}.

In the rest of this introduction, we sketch our methods and results in more detail.  Section \ref{sec:webtab} gives a quick summary of the background information on webs and tableaux that we need, including the bijection between the two (and intermediate objects like $M$-diagrams and boundary words used in that bijection).  

Section \ref{sec:halfplane} establishes the combinatorial generality in which we work. In particular, many of our arguments use only a subset of the conditions that characterize webs.  We define a {\em half-plane graph} to be any graph embedded in the plane with univalent vertices on the $x$-axis of the plane and trivalent vertices in the upper half-plane.  Since half-plane graphs are planar, they have dual graphs; we use distance in the dual graph to establish the notion of depth of faces of a half-plane graph.  Our definition of depth of faces generalizes the existing notion of depth for webs, which is crucial to the bijection between webs and tableaux \cite{KK, MR3119361, T}.  We use half-plane graphs to partition the upper half-plane into so-called {\em $d$-bands,} each of which consists of the union of all regions of depth $d$, with a {\em band diagram} that consists of the edges of the half-plane graph where distinct bands meet.  Figure \ref{fig:shadowex} gives an example.

\begin{figure}[ht]
\begin{tikzpicture}[scale=.5]

\draw[dashed, <-] (-1.5,0)--(4,0);
\draw[dashed, ->] (4,0)--(6.5,0);
\draw[fill=black,radius=4pt] (0,0)circle;
\draw[fill=black,radius=4pt] (1,0)circle;
\draw[fill=black,radius=4pt] (2,0)circle;
\draw[fill=black,radius=4pt] (3,0)circle;
\draw[fill=black,radius=4pt] (4,0)circle;
\draw[fill=black,radius=4pt] (5,0)circle;
\draw[fill=black,radius=4pt] (.5,1.5)circle;
\draw[fill=black,radius=4pt] (2.5,1.5)circle;
\draw[fill=black,radius=4pt] (4,1)circle;
\draw[fill=black,radius=4pt] (5,1)circle;
\draw[fill=black,radius=4pt] (4.5,2)circle;
\draw[fill=black,radius=4pt] (2.5,2.5)circle;

\draw[ black] (0,0)--(.5,1.5)--(1,0);
\draw[black] (.5,1.5)--(2.5,2.5)--(2.5,1.5)--(2,0);
\draw[black] (2.5,1.5)--(3,0);
\draw[black] (4,0)--(4,1)--(4.5,2)--(2.5,2.5);
\draw[black] (4,1)--(5,1);
\draw[black] (5,0)--(5,1)--(4.5,2);

\draw[thick, red] (0,0)--(.5,1.5)--(2.5,2.5)--(4.5,2)--(5,1)--(5,0);
\draw[thick, red] (2,0)--(2.5,1.5)--(3,0);

\node at (.5,.5) {\tiny{1}};
\node at (2.5,.5) {\tiny{2}};
\node at (2,3) {\tiny{0}};
\node at (1.5,1) {\tiny{1}};
\node at (3.5,1) {\tiny{1}};
\node at (4.5,1.35) {\tiny{1}};
\node at (4.5,.5) {\tiny{1}};

\node at (0,-.5) {\tiny{+}};
\node at (1,-.5) {\tiny{0}};
\node at (2,-.5) {\tiny{+}};
\node at (3,-.5) {\tiny{$-$}};
\node at (4,-.5) {\tiny{0}};
\node at (5,-.5) {\tiny{$-$}};

\node at (-.75,-1.25) {\small{$S_1$:}};
\draw[black] (0,-1.1)--(0,-1.4);
\draw[black] (5,-1.1)--(5,-1.4);
\draw[black] (0,-1.25)--(5,-1.25);

\node at (-.75,-2.25) {\small{$S_2$:}};
\draw[black] (2,-2.1)--(2,-2.4);
\draw[black] (3,-2.1)--(3,-2.4);
\draw[black] (2,-2.25)--(3,-2.25);
\end{tikzpicture}
\caption{A half-plane graph, band diagram, and shadow}\label{fig:shadowex}
\end{figure}
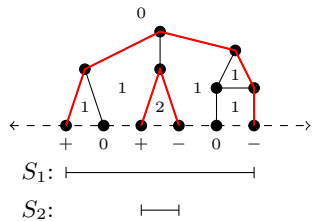

To a band diagram, we associate a {\em boundary word} in the alphabet $\{+, 0, -\}$ that encodes the consecutive changes in depth at integral boundary points of the half-plane graph. Corollary \ref{cor:anchored} proves the band diagram is a crossingless matching on boundary points labeled $+$ and $-$ possibly with additional closed circular components. This, too, generalizes a classical result from webs: Petersen, Pylyavskyy, and Rhoades proved that the boundary word of an $\mathfrak{sl}_3$ web is a balanced Yamanouchi word on the symbols $\{+,0,-\},$ where {\em balanced} means there are the same number of $+$s, $0$s, and $-$s, and {\em Yamanouchi} means every prefix of the word has at least as many $+$s as $0$s and at least as many $0$s as $-$s \cite{PPR}.  Finally, the {\em shadow} of a half-plane graph is a nested sequence of intervals obtained by projecting the arcs in the band diagram to the $x$-axis, as shown in Figure \ref{fig:shadowex}; two webs are comparable in the {\em shadow containment partial order} if their respective shadows are contained as subsets of the real line.

Sections \ref{sec:band} and Section \ref{sec:trans} contain the main results of the paper.  They build on Section \ref{sec:taborder}, which reviews a ranked partial order on tableaux $\prec_T$ coming from permutation of entries that we call the {\em tableau partial order,} as well as the partial order it induces on webs via the bijection from Section \ref{sec:webtab} \cite{MR3920353}.   Section \ref{sec:band} then defines the shadow containment partial order $\prec_S$ on webs and proves foundational properties, including: Lemma \ref{cont1} and Corollary \ref{cor:TvsS} that shadow containment partial order is a refinement of the tableau partial order for $\mathfrak{sl}_2$ and $\mathfrak{sl}_3$; Theorem \ref{thm:samesl2} that  $\prec_T$ and $\prec_S$ coincide for $\mathfrak{sl}_2$ webs; and Theorem \ref{thm:notsamesl3} that $\prec_T$ and $\prec_S$ differ for $\mathfrak{sl}_3$. Finally, Section \ref{sec:trans} defines the symmetric group action on the Specht and web bases and proves the main unitriangularity theorem stated above. 

We end with the conjecture that the transition matrix between the Specht and $\mathfrak{sl}_3$ web bases is unitriangular with respect to any total order completing the tableau partial order, and with a discussion of why the counterexamples produced elsewhere in the manuscript do not disprove this conjecture.

\section{Webs and Tableaux} \label{sec:webtab}

This section gives the essential background and definitions on webs and tableaux.

Given $k\geq 2$, an $\mathfrak{sl}_k$ web is a plane graph that acts as a morphism in a diagrammatic category encoding the representation theory of $U_q(\mathfrak{sl}_k)$. Each web encodes a unique intertwining map between representations of $U_q(\mathfrak{sl}_k)$. One common convention is to draw webs in a rectangular region with boundary vertices on the top and bottom edges of the region. The bottom and top boundary vertices of a web indicate the domain and codomain; composition of maps corresponds to stacking webs. The internal structure and labels of a web encode the map itself.  For instance, all boundary vertices are univalent, all internal vertices are trivalent, and the edges carry labels that depend on the value of $k$.  There are diagrammatic relations between (linear combinations of) webs that represent the same map. See \cite{cautis2014webs, Kim,MR1403861,MR2710589,MR2873098} for more on webs and their connection to quantum representation theory.

Precise labeling conventions and relations for webs differ slightly across the literature. We use those found in Kuperberg's work  though Kuperberg typically draws webs with boundary on a disk with a basepoint rather than along a horizontal axis \cite{MR1403861}. We specialize the quantum parameter to $q=1$ throughout our discussion.

In this section, we define $\mathfrak{sl}_2$ webs, $\mathfrak{sl}_3$ webs, and standard Young tableaux and describe combinatorial bijections between certain subsets of these objects. These bijections are key to understanding the relationship between the Specht and web bases for symmetric group representations.

\subsection{$\mathfrak{sl}_2$ and $\mathfrak{sl}_3$ Webs}

$\mathfrak{sl}_2$ webs--also known as Temperley-Lieb diagrams--are ubiquitous in mathematics and have been extensively studied \cite{jones1983index, kauffman1987state, lickorish1992calculations, RTW}. We include them in this paper especially to contrast their structure and properties with those of $\mathfrak{sl}_3$ webs. 

\begin{definition}
An {\it $\mathfrak{sl}_2$ web} on $n+m$ points is a crossingless matching in a rectangular region pairing $n$ points on the bottom and $m$ points on top possibly with circles in the interior. An {\it $\mathfrak{sl}_2$ bottom web} is an $\mathfrak{sl}_2$ web on $2n+0$ points for some $n\in \mathbb{N}$.   
\end{definition}

Let $\mathcal{W}$ be the set of all $\mathfrak{sl}_2$ webs with a fixed number of top and bottom boundary points. Denote by $V^{\mathcal{W}}$ the space of $\mathbb{C}$- linear combinations of webs in $\mathcal{W}$ modulo the relation shown in Figure \ref{fig:circlesl2} that a web with a circle component is equal to the result of removing the circle and multiplying by a factor of $-2$.  We call an $\mathfrak{sl}_2$ web {\it reduced} if it has no circle components. The reduced webs in $\mathcal{W}$ form a basis for $V^{\mathcal{W}}$. 

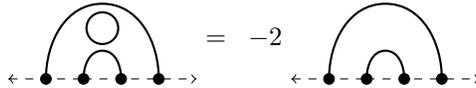
\begin{figure}[ht]
\begin{tikzpicture}[scale=.5]

\draw[dashed, <-] (-1,0)--(1,0);
\draw[dashed, ->] (1,0)--(4,0);
\draw[fill=black,radius=4pt] (0,0)circle;
\draw[fill=black,radius=4pt] (1,0)circle;
\draw[fill=black,radius=4pt] (2,0)circle;
\draw[fill=black,radius=4pt] (3,0)circle;
\draw[radius=12pt, thick] (1.5,1.35)circle;

\draw[ black, thick] (0,0) to[out=90,in=180] (1.5,2) to[out=0,in=90] (3,0);
\draw[ black, thick] (1,0) to[out=90,in=180] (1.5,.75) to[out=0,in=90] (2,0);
\end{tikzpicture}
\raisebox{15pt}{$=\;\; -2$}
\begin{tikzpicture}[scale=.5]
\draw[dashed, <-] (-1,0)--(1,0);
\draw[dashed, ->] (1,0)--(4,0);
\draw[fill=black,radius=4pt] (0,0)circle;
\draw[fill=black,radius=4pt] (1,0)circle;
\draw[fill=black,radius=4pt] (2,0)circle;
\draw[fill=black,radius=4pt] (3,0)circle;

\draw[ black, thick] (0,0) to[out=90,in=180] (1.5,2) to[out=0,in=90] (3,0);
\draw[ black, thick] (1,0) to[out=90,in=180] (1.5,.75) to[out=0,in=90] (2,0);
\end{tikzpicture}
\caption{Removing a circle from an $\mathfrak{sl}_2$ web} \label{fig:circlesl2}
\end{figure}

\begin{definition}
An {\it $\mathfrak{sl}_3$ web} is an oriented, plane graph in a rectangular region with boundary vertices on the top and bottom edges of the region such that: all boundary vertices are univalent, all internal vertices are trivalent, and every vertex is a source or sink. An {\it $\mathfrak{sl}_3$ bottom web} is an $\mathfrak{sl}_3$ web with no endpoints on the top boundary of the rectangular region. 
\end{definition}

As with $\mathfrak{sl}_2$ webs, an $\mathfrak{sl}_3$ web can have disconnected oriented circle components in the interior of the region. 
 
Given the set $\mathcal{W}$ of all $\mathfrak{sl}_3$ webs with a fixed set of boundary vertices and boundary orientations, we denote by $V^{\mathcal{W}}$ the space of $\mathbb{C}$-linear combinations of webs in $\mathcal{W}$ equipped with the relations shown in Figure \ref{fig:webrelations}. We call an $\mathfrak{sl}_3$ web {\it reduced} (or {\it non-elliptic}) if all closed faces have at least six boundary edges. (This prohibits circle components, which are interior faces with just one boundary edge.) Using the $\mathfrak{sl}_3$ web relations, every web can be expressed uniquely as a linear combination of reduced webs \cite{MR1403861}. Figure \ref{fig:webreduc} shows an example. Thus the reduced webs in $\mathcal{W}$ form a basis for $V^{\mathcal{W}}$.

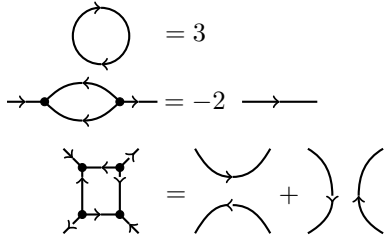
\begin{figure}[ht]
\begin{tikzpicture}[scale=.5]
\draw[thick, ->] (-.5,0) to[out=180,in=270] (-1.25,.75) to[out=90,in=180] (-.5,1.5); 
\draw[thick, ->] (-.5,1.5) to[out=0,in=90] (.25,.75) to[out=270,in=0] (-.5,0); 
\node at (1.75,.85) {$= 3$};

\draw[thick,->] (-3,-1)--(-2.5,-1);
\draw[thick] (-2,-1)--(-2.5,-1);
\draw[fill=black,radius=3pt] (-2,-1)circle;
\draw[thick] (-2,-1) to[out=45,in=180] (-1,-.5);
\draw[thick,<-] (-1,-.5) to[out=0,in=135] (0,-1);
\draw[fill=black,radius=3pt] (0,-1)circle;
\draw[thick] (-2,-1) to[out=-45,in=180] (-1,-1.5);
\draw[thick,<-] (-1,-1.5) to[out=0,in=225] (0,-1);
\draw[thick,->] (0,-1)--(.5,-1);
\draw[thick] (1,-1)--(.5,-1);
\node at (2,-1) {$= -2$};
\draw[thick,->] (3.25,-1)--(4.25,-1);
\draw[thick] (4.25,-1)--(5.25,-1);

\draw[fill=black,radius=3pt] (0,-2.75)circle;
\draw[fill=black,radius=3pt] (-1,-2.75)circle;
\draw[fill=black,radius=3pt] (0,-4)circle;
\draw[fill=black,radius=3pt] (-1,-4)circle;
\draw[thick,->](0,-2.75) -- (-.5,-2.75);
\draw[thick](-1,-2.75) -- (-.5,-2.75);
\draw[thick](0,-4) -- (-.5,-4);
\draw[thick,->](-1,-4) -- (-.5,-4);
\draw[thick,-<](0,-4) -- (0,-3);
\draw[thick](0,-3) -- (0,-2.75);
\draw[thick,->](-1,-4) -- (-1,-3);
\draw[thick](-1,-3) -- (-1,-2.75);
\draw[thick] (0,-2.75)--(.25,-2.5);
\draw[thick,-<] (.5,-2.25)--(.25,-2.5);
\draw[thick] (0,-4)--(.25,-4.25);
\draw[thick,->] (.5,-4.5)--(.25,-4.25);
\draw[thick] (-1,-2.75)--(-1.25,-2.5);
\draw[thick,->] (-1.5,-2.25)--(-1.25,-2.5);
\draw[thick] (-1,-4)--(-1.25,-4.25);
\draw[thick,-<] (-1.5,-4.5)--(-1.25,-4.25);

\draw[thick, -<] (2,-4.5) to[out=60,in=180] (3,-3.75);
\draw[thick,->] (2,-2.25) to[out=-60,in=180] (3,-3);
\draw[thick] (3,-3.75) to[out=0,in=120] (4,-4.5);
\draw[thick] (3,-3) to[out=0,in=240] (4,-2.25);

\draw[thick, -<] (5,-4.5) to[out=30,in=270] (5.65,-3.5);
\draw[thick] (5.65,-3.5) to[out=90,in=-30] (5,-2.25);
\draw[thick, ->] (7,-4.5) to[out=150,in=270] (6.35,-3.5);
\draw[thick] (6.35,-3.5) to[out=90,in=210] (7,-2.25);

\node at (1.5,-3.5) {$=$};
\node at (4.5, -3.5) {$+$};

\end{tikzpicture}
\caption{Relations on $\mathfrak{sl}_3$ webs}\label{fig:webrelations}
\end{figure}  

\begin{figure}[ht]
\begin{tikzpicture}[scale=.75]
\draw[style=dashed, <->] (.5,0)--(3.5,0);
\draw[radius=.08, fill=black](1,0)circle;
\draw[radius=.08, fill=black](2,0)circle;

\draw[radius=.08, fill=black](3,0)circle;

\draw[radius=.08, fill=black](2,.5)circle;
\draw[radius=.08, fill=black](2,1.25)circle;
\draw[radius=.08, fill=black](2,2)circle;
\begin{scope}[thick,decoration={
    markings,
    mark=at position 0.5 with {\arrow{>}}}
    ] 
   \draw[postaction={decorate}] (1,0)--(2,2);
    \draw[postaction={decorate}] (3,0)--(2,2);
   \draw[postaction={decorate}] (2,0)--(2,.5);
   \draw[postaction={decorate}] (2,1.25)--(2,2);
   \draw[postaction={decorate}] (2,1.25) to[out=250,in=90] (1.75,.9) to[out=270,in=110] (2,.5);
  \draw[postaction={decorate}] (2,1.25) to[out=290,in=90] (2.25,.9) to[out=270, in=70] (2,.5);
\end{scope}

\end{tikzpicture}
\raisebox{15pt}{$=\;\; -2$}\hspace{-.1in}
\begin{tikzpicture}[scale=.75]
\draw[style=dashed, <->] (.5,0)--(3.5,0);
\draw[radius=.08, fill=black](1,0)circle;
\draw[radius=.08, fill=black](2,0)circle;

\draw[radius=.08, fill=black](3,0)circle;

\draw[radius=.08, fill=black](2,1)circle;

\begin{scope}[thick,decoration={
    markings,
    mark=at position 0.5 with {\arrow{>}}}
    ] 
   \draw[postaction={decorate}] (1,0)--(2,1);
    \draw[postaction={decorate}] (3,0)--(2,1);
   \draw[postaction={decorate}] (2,0)--(2,1);

\end{scope}

\end{tikzpicture}
\caption{Reducing an $\mathfrak{sl}_3$ web} \label{fig:webreduc}
\end{figure}
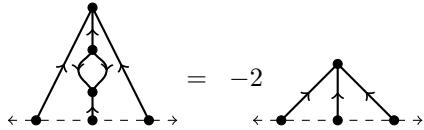

In this paper, we consider the following sets of bottom webs, to which we assign special notation.  Let $\mathcal{W}_{2n}$ be the set of $\mathfrak{sl}_2$ webs on $2n+0$ points, and let $\mathcal{W}_{3n}$ be the set of $\mathfrak{sl}_3$ webs with $0$ top boundary points and $3n$ bottom boundary points.  Moreover, we require that all (bottom) boundary points of webs in $\mathcal{W}_{3n}$ are sources. We denote their corresponding web spaces by $V^{\mathcal{W}_{2n}}$ and $V^{\mathcal{W}_{3n}}$. 

To track combinatorial statistics for these webs, we enumerate their vertices from left to right beginning with $1$. Using the notion of depth in a web, we associate to each web a boundary word \cite{MR3119361, T}.

\begin{definition}
Let $w$ be an $\mathfrak{sl}_2$ or $\mathfrak{sl}_3$ bottom web.
\begin{itemize}
\item The {\it faces of $w$} are the path-connected, open subsets of the complement of $w$ in the upper half-plane. 
\item Let $F$ be a face of $w$ and let $U$ be the unbounded face. The {\it depth} of $F$ in $w$ is the minimum number of edges of $w$ a path from $F$ to $U$ must cross.
\item The {\it boundary word for $w$} is a string of symbols--one for each boundary vertex of $w$--such that the $i^{th}$ symbol is $+$ if the depth immediately to the left of vertex $i$ is smaller than the depth to the right, $0$ if the depths on either side are equal, and $-$ if the depth to the left is larger than the depth to the right.   
\end{itemize}
\end{definition}

\begin{remark}
Note that the paths used to define depth are not permitted to pass through vertices of the web. Moreover, depth is well-defined because it is the distance between vertices in the web's dual graph (considering the web as a planar graph).
\end{remark}

The boundary word for an $\mathfrak{sl}_2$ web uses only $+$ and $-$. Figures  \ref{fig:unreduced} and \ref{fig:sl2band}  show examples of bottom webs and their boundary words. 

The next definitions are adapted from Petersen-Pylyavskky-Rhoades \cite{PPR}.

\begin{figure}[ht]
\begin{tikzpicture}[scale=.75]
\draw[style=dashed, <->] (.5,0)--(9.5,0);
\draw[radius=.08, fill=black](1,0)circle;
\draw[radius=.08, fill=black](2,0)circle;

\draw[radius=.08, fill=black](3,0)circle;
\draw[radius=.08, fill=black](4,0)circle;
\draw[radius=.08, fill=black](5,0)circle;
\draw[radius=.08, fill=black](6,0)circle;
\draw[radius=.08, fill=black](7,0)circle;
\draw[radius=.08, fill=black](8,0)circle;
\draw[radius=.08, fill=black](9,0)circle;

\draw[radius=.08, fill=black](4.5,.5)circle;
\draw[radius=.08, fill=black](4.5,1)circle;
\draw[radius=.08, fill=black](5,1.5)circle;
\draw[radius=.08, fill=black](4,1.5)circle;
\draw[radius=.08, fill=black](4.5,2)circle;
\draw[radius=.08, fill=black](4.5,2.5)circle;
\draw[radius=.08, fill=black](5.5,2)circle;
\draw[radius=.08, fill=black](5.5,2.5)circle;
\draw[radius=.08, fill=black](5,3)circle;
\draw[radius=.08, fill=black](5,3.5)circle;
\draw[radius=.08, fill=black](6.5,.5)circle;

\node at (1,-.4) {$+$};
\node at (2,-.4) {$+$};
\node at (3,-.4) {$+$};
\node at (4,-.4) {$+$};
\node at (5,-.4) {$-$};
\node at (6,-.4) {0};
\node at (7,-.4) {$-$};
\node at (8,-.4) {$-$};
\node at (9,-.4) {$-$};
\begin{scope}[thick,decoration={
    markings,
    mark=at position 0.5 with {\arrow{>}}}
    ] 
   \draw[postaction={decorate}] (3,0)--(4,1.5);
   \draw[postaction={decorate}] (4,0)--(4.5,.5);
   \draw[postaction={decorate}] (5,0)--(4.5,.5);
   \draw[postaction={decorate}] (4.5,1)--(4.5,.5);
   \draw[postaction={decorate}] (4.5,1)--(4,1.5);
   \draw[postaction={decorate}] (4.5,1)--(5,1.5);
   \draw[postaction={decorate}] (4.5,2)--(4,1.5);
   \draw[postaction={decorate}] (4.5,2)--(5,1.5);
   \draw[postaction={decorate}] (4.5,2)--(4.5,2.5);
   \draw[postaction={decorate}] (5.5,2)--(5,1.5);
   \draw[postaction={decorate}] (5.5,2)--(5.5,2.5);
   \draw[postaction={decorate}] (5.5,2)--(6.5,.5);
   \draw[postaction={decorate}] (6,0)--(6.5,.5);
   \draw[postaction={decorate}] (7,0)--(6.5,.5);
   \draw[postaction={decorate}] (5,3)--(4.5,2.5);
   \draw[postaction={decorate}] (5,3)--(5.5,2.5);
   \draw[postaction={decorate}] (5,3)--(5,3.5);
   \draw[postaction={decorate}] (2,0)--(4.5,2.5);
   \draw[postaction={decorate}] (8,0)--(5.5,2.5);
   \draw[postaction={decorate}] (1,0)--(5,3.5);
   \draw[postaction={decorate}] (9,0)--(5,3.5);
\end{scope}

\node at (3,2.25) {\tiny{$0$}};
\node at (1.85, .5) {\tiny{$1$}};
\node at (2.85, .5) {\tiny{$2$}};
\node at (3.85, .5) {\tiny{$3$}};
\node at (4.5, .25) {\tiny{$4$}};
\node at (4.5, 1.5) {\tiny{$3$}};
\node at (5.5, .5) {\tiny{$3$}};
\node at (6.5, .25) {\tiny{$3$}};
\node at (5, 2.25) {\tiny{$2$}};
\node at (7.15, .5) {\tiny{$2$}};
\node at (8.15, .5) {\tiny{$1$}};

\end{tikzpicture}
\caption{An unreduced $\mathfrak{sl}_3$ web and its boundary word.}\label{fig:unreduced}
\end{figure}
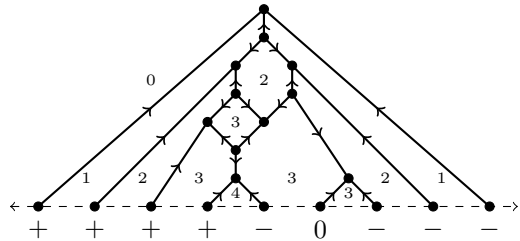 

\begin{definition}
Let $S$ be a string of symbols in the alphabet $\{s_1, \ldots, s_n\}$. Given $1\leq i,j\leq n$, we say $S$ is {\it $(s_is_j)$-balanced} if it has the same number of $s_i$ and $s_j$ symbols and {\it $(s_is_j)$-Yamanouchi} if, reading left to right, the number of $s_j$ symbols at any point does not exceed the number of $s_i$ symbols. 
\end{definition}

Note that the order of symbols matters in the definition of Yamanouchi but not balanced. In other words $(s_is_j)$-balanced is equivalent to $(s_js_i)$-
balanced, but $(s_is_j)$-Yamanouchi is not the same as $(s_js_i)$-Yamanouchi.

As we will see, boundary words are key to relating webs to tableaux. The next two results are proven in some form in various works \cite{KK,PPR,T}. They also follow from the discussion on band diagrams for half plane graphs in Section \ref{sec:halfplane} below. 

\begin{proposition}
The boundary words for webs in $\mathcal{W}_{2n}$ are in the alphabet $\{+,-\}$ and are both $(+-)$-balanced and $(+-)$-Yamanouchi. 
\end{proposition}
\begin{proposition}\label{prop:sl3webword}
The boundary words for webs in $\mathcal{W}_{3n}$ are in the alphabet $\{+,0,-\}$ and are both $(+-)$-balanced and $(+-)$-Yamanouchi.  Reduced webs in $\mathcal{W}_{3n}$ have boundary words that are also $(+0)$ and $(0-)$-balanced as well as $(+0)$ and $(0-)$-Yamanouchi.
\end{proposition}

\subsection{Tableaux and Webs} \label{sub:tabweb}
We begin with terminology and notation related to Young tableaux. An interested reader can find this information and much more in Fulton's book \cite{MR1464693}. More information on the maps between tableaux and webs described below can be found in, e.g., \cite{KK, PPR,MR3119361, T}.

\begin{definition}
Given $m\in \mathbb{N}$ and a partition $\lambda = (\lambda_1, \ldots, \lambda_t)\vdash m$, we define the following terms. 
\begin{itemize}
\item The {\it Young diagram of shape $\lambda$} is a top- and left-justified collection of $m$ boxes in which the $i^{th}$ row has $\lambda_i$ boxes for each $i$. 
\item A {\it Young tableau of shape $\lambda$} is  a filling of the boxes of the Young diagram of shape $\lambda$ with the numbers $1$ through $m$ each occurring exactly once. 
\item A {\it standard Young tableau of shape $\lambda$} is a Young tableau of shape $\lambda$  in which the numbers increase from left to right along rows and top to bottom along columns.
\end{itemize}
We denote the set of all standard Young tableaux of shape $\lambda$ by $SYT(\lambda)$.
\end{definition}

Let $\mathcal{W}_{2n}^R$ be the set of reduced webs in $\mathcal{W}_{2n}$.  Define the map $\psi: \mathcal{W}_{2n}^R \rightarrow SYT(n,n)$ so that $\psi(w)$ is the tableau obtained by writing $i$ in the top (respectively bottom) row if the $i^{th}$ symbol in the boundary word for $w$ is $+$ (respectively $-$). It is widely known that the map $\psi$ is a well-defined bijection \cite[items (n) and (ww)]{MR1676282}.

To find the web $\psi^{-1}(T)$ for a standard tableau $T\in SYT(n,n)$,  first construct a word in the alphabet $\{+,-\}$ from $T$ such that the $i^{th}$ symbol is $+$ if $i$ is in the top row of $T$ and $-$ if $i$ is in the bottom row. This word will be $(+-)$-balanced and $(+-)$- Yamanouchi by construction. There is a unique reduced $\mathfrak{sl}_2$ web on $2n+0$ boundary points with that boundary word.  Figure \ref{fig:sl2band} gives an example. 

\begin{figure}[ht]
\begin{tikzpicture}[scale=.5]

\draw[dashed, <-] (-1,0)--(1,0);
\draw[black] (0,0)--(5,0);
\draw[dashed, ->] (4,0)--(6,0);
\draw[fill=black,radius=4pt] (0,0)circle;
\draw[fill=black,radius=4pt] (1,0)circle;
\draw[fill=black,radius=4pt] (2,0)circle;
\draw[fill=black,radius=4pt] (3,0)circle;
\draw[fill=black,radius=4pt] (4,0)circle;
\draw[fill=black,radius=4pt] (5,0)circle;

\draw[ black, thick] (0,0) to[out=90,in=180] (2.5,2) to[out=0,in=90] (5,0);
\draw[ black, thick] (1,0) to[out=90,in=180] (1.5,.75) to[out=0,in=90] (2,0);
\draw[ black, thick] (3,0) to[out=90,in=180] (3.5,.75) to[out=0,in=90] (4,0);

\node at (2.5,1.25) {\tiny{1}};
\node at (2.5,2.5) {\tiny{0}};
\node at (1.5,.4) {\tiny{2}};
\node at (3.5,.4) {\tiny{2}};

\node at (0,-.5) {\tiny{+}};
\node at (1,-.5) {\tiny{+}};
\node at (2,-.5) {\tiny{$-$}};
\node at (3,-.5) {\tiny{$+$}};
\node at (4,-.5) {\tiny{$-$}};
\node at (5,-.5) {\tiny{$-$}};

\node at (7,1) {$\stackrel{\psi}{\longrightarrow}$};

\node at (10,1) {\Large{\young(124,356)}};
\end{tikzpicture}
\caption{An $\mathfrak{sl}_2$ web, boundary word, and tableau} \label{fig:sl2band}
\end{figure}
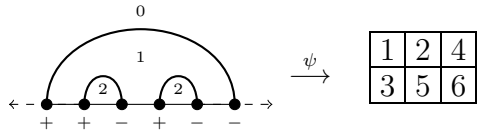

Analogously, let $\mathcal{W}_{3n}^R$ be the set of reduced webs in $\mathcal{W}_{3n}$ and define the map $\psi: \mathcal{W}_{3n}^R \rightarrow SYT(n,n,n)$ such that $\psi(w)$ is the tableau obtained by writing $i$ in the top (respectively middle, bottom) row if the $i^{th}$ symbol in the boundary word for $w$ is $+$ (respectively 0, $-$).  

Explicitly constructing $\psi^{-1}(T)$ for $T\in SYT(n,n,n)$ involves an intermediate object called an $M$-diagram \cite{T}.  As in the $\mathfrak{sl}_2$ case, begin by constructing a word in $\{+,-,0\}$ with $+$ corresponding to top row, $0$ corresponding to middle, and $-$ to bottom of $T$. Since $T$ is standard, the resulting word is $(+,0)$, $(+,-)$ and $(0,-)$- balanced as well as $(+,0)$, $(+,-)$ and $(0,-)$-Yamanouchi. 

Next, label the vertices $1, \ldots 3n$ on a horizontal axis with the word coming from $T$. The $M$-diagram for $T$ consists of pairs of crossingless matchings on these labeled points: the vertices labeled $+$ and $0$ are matched with left endpoints labeled $+$ and right endpoints labeled $0$; and those labeled $0$ and $-$ are matched with left endpoints labeled $0$ and right endpoints labeled $-$. 
Note that each matching is crossingless, but the arcs from one matching can intersect the other. As we will see later, $M$-diagrams are not only useful when describing $\psi^{-1}$ but also give a partial order on reduced $\mathfrak{sl}_3$ webs.

\begin{figure}[ht]
\scalebox{1}{\begin{tikzpicture}[baseline=0cm, scale=.6]
\node at (0,0) {\young(13,25,46)};
\node at (1.5,0) {$\longrightarrow$};
\end{tikzpicture}
 \begin{tikzpicture}[baseline=0cm, scale=.6]
\draw[style=dashed, <->] (3,-.7)--(9,-.7);
\draw[radius=.08, fill=black](3.5,-.7)circle;
\draw[radius=.08, fill=black](4.5,-.7)circle;
\draw[radius=.08, fill=black](5.5,-.7)circle;
\draw[radius=.08, fill=black](6.5,-.7)circle;
\draw[radius=.08, fill=black](7.5,-.7)circle;
\draw[radius=.08, fill=black](8.5,-.7)circle; 
\node at (3.5,-1.1) {$+$};
\node at (4.5,-1.1) {0};
\node at (5.5,-1.1) {$+$};
\node at (6.5,-1.1) {$-$};
\node at (7.5,-1.1) {0};
\node at (8.5,-1.1) {$-$};
\draw (4.5,-.7) arc (0:180: .5cm);
\draw (7.5,-.7) arc (0:180: 1cm); 
\draw (6.5,-.7) arc (0:180: 1cm);
\draw (8.5,-.7) arc (0:180: .5cm);  
\node at (9.75,0) {$\longrightarrow$};
\end{tikzpicture} 
\raisebox{-13pt}{\begin{tikzpicture}[baseline=0cm, scale=0.6]
\draw[style=dashed, <->] (2.5,0)--(8.5,0);
\draw[radius=.08, fill=black](3,0)circle;
\draw[radius=.08, fill=black](4,0)circle;
\draw[radius=.08, fill=black](5,0)circle;
\draw[radius=.08, fill=black](6,0)circle;
\draw[radius=.08, fill=black](7,0)circle;
\draw[radius=.08, fill=black](8,0)circle;
\draw[style=thick,->](4, 0) -- (4,.5);
\draw[style=thick](4,.5)--(4,1);
\draw[radius=.08, fill=black](4,1)circle;
\draw[style=thick,->](3,0)--(3.5,.5);
\draw[style=thick](3.5,.5)--(4,1);
\draw[style=thick,->](7, 0) -- (7,.5);
\draw[style=thick](7,.5)--(7,1);
\draw[radius=.08, fill=black](7,1)circle;
\draw[style=thick,->](8,0)--(7.5,.5);
\draw[style=thick](7.5,.5)--(7,1);
\draw[radius=.08, fill=black](5.5,1)circle;
\draw[style=thick,-<](5.5,1)--(5.5,1.5);
\draw[style=thick](5.5,1.5)--(5.5,2);
\draw[radius=.08, fill=black](5.5,2)circle;
\draw[style=thick,->](5,0)--(5.25,.5);
\draw[style=thick](5.25,.5)--(5.5,1);
\draw[style=thick,->](6,0)--(5.75,.5);
\draw[style=thick](5.75,.5)--(5.5,1);
\draw[style=thick,-<](4,1)--(4.75,1.5);
\draw[style=thick](4.75,1.5)--(5.5,2);
\draw[style=thick,-<](7,1)--(6.25,1.5);
\draw[style=thick](6.25,1.5)--(5.5,2);

\node at (3,-.4) {$+$};
\node at (4,-.4) {0};
\node at (5,-.4) {$+$};
\node at (6,-.4) {$-$};
\node at (7,-.4) {0};
\node at (8,-.4) {$-$};
\end{tikzpicture}}
}
\caption{Constructing an $\mathfrak{sl}_3$ web from a tableau}\label{fig:sl3ex}
\end{figure}
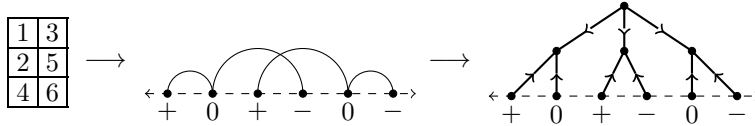 

To obtain a web from an $M$-diagram, orient all arcs of the $M$-diagram away from boundary vertices labeled $+$ and $-$, and then modify a small neighborhood of each crossing and each boundary vertex labeled $0$ as shown in Figure \ref{fig:MMod}.  Figure \ref{fig:sl3ex} has an example of a 3-row tableau and its companion diagrams.  

\begin{figure}[ht]
\raisebox{-12pt}{\begin{tikzpicture}[baseline=0cm, scale=.5]
\draw[style=dashed, <->] (-.5,0)--(3.5,0);
\node at (1.5,-.5) {$0$};
\draw[radius=.1, fill=black](1.5,0)circle;
\draw [style=thick, -<](1.5,0) arc (0:90: 1.2cm);
\draw[style=thick](.3,1.2) arc (90:110: 1.2cm);
\draw [style=thick, -<] (1.5,0) arc (180:90: 1.2cm);
\draw[style=thick](2.7,1.2) arc (90:70: 1.2cm);
\node at (4,.75) {$\rightarrow$};
\draw[style=dashed, <->] (4.25,0)--(7.25,0);
\draw[radius=.1, fill=black](5.75,0)circle;
\draw[radius=.1, fill=black](5.75,1)circle;
\draw[style=thick,->](5.75,0) -- (5.75,.5);
\draw[style=thick](5.75,.5) -- (5.75,1);
\draw[style=thick, -<] (5.75,1) -- (5.25,1.1);
\draw[style=thick](5.25,1.1) -- (4.75,1.2);
\draw[style=thick, -<] (5.75,1) -- (6.25,1.1);
\draw[style=thick](6.25,1.1) -- (6.75,1.2);
\end{tikzpicture}}
\hspace{.5in}
\raisebox{-3pt}{\begin{tikzpicture}[baseline=0cm, scale=0.4]
\draw[style=thick, ->] (-1,-1)--(-.5,-.5);
\draw[style=thick, -<] (-1,1)--(-.5,.5);
\draw[style=thick, -<] (-.5,.5)--(.5,-.5);
\draw[style=thick, ->] (-.5,-.5)--(.5,.5);
\draw[style=thick] (.5,.5)--(1,1);
\draw[style=thick] (.5,-.5)--(1,-1);
\node at (2,0) {$\rightarrow$};

\draw[style=thick, ->] (4,.5)--(4,0);
\draw[style=thick] (4,0)--(4,-.5);
\draw[style=thick, ->] (4,.5)--(3.5,.75);
\draw[style=thick] (3.5,.75)--(3,1);
\draw[style=thick, ->] (4,.5)--(4.5,.75);
\draw[style=thick] (4.5,.75)--(5,1);
\draw[style=thick, -<] (4,-.5)--(3.5,-.75);
\draw[style=thick] (3.5,-.75)--(3,-1);
\draw[style=thick, -<] (4,-.5)--(4.5,-.75);
\draw[style=thick] (4.5,-.75)--(5,-1);
\draw[ radius=.12, fill=black] (4,-.5) circle;
\draw[ radius=.12, fill=black] (4,.5) circle;
\end{tikzpicture}}
\caption{Constructing a web from an oriented $M$-diagram}\label{fig:MMod}
\end{figure}
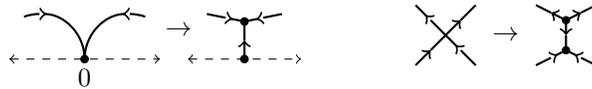 

\section{Half-plane graphs}\label{sec:halfplane}
Webs were constructed as graphs that model representation-theoretic phenomena, yet many of the graph-theoretic tools used on webs are applicable in a wider setting.  In this section, we generalize webs to a larger family of what we call {\em half-plane graphs} and give a preliminary analysis of their graph-theoretic properties. 

More specifically, we define half-plane graphs, construct a structure called a {\em band diagram} that summarizes the nesting structure of a half-plane graph, establish some general properties about band diagrams, and then specialize to the case of webs. 

\subsection{Band diagrams for half-plane graphs}
We start by defining half-plane graphs, band diagrams, and some of their fundamental properties.  Recall that a {\em plane} graph is a planar graph with a choice of embedding.  

\begin{definition}A {\it half-plane graph $G^+$ with $n$ boundary vertices} is a simple, plane graph with vertex set $V= \{1,2,\dots,n\} \cup V_{int}$ and edge set 
$E$ with the properties that
\begin{enumerate}
\item $E$ contains exactly one edge incident to each vertex $\{1,2,\dots,n\}$.
\item Vertices in $V_{int}$ are trivalent.
\item The graph can be drawn so that
\begin{itemize}
\item the vertices $\{1,2,\dots,n\}$ are the associated integers on the $x$-axis of the plane, and 
\item the vertices $V_{int}$ are all in the the upper half of the plane.
\end{itemize}
\end{enumerate}
The {\it faces of $G^+$} are the path-connected, open subsets of the complement of $G^+$ in the upper half-plane. 
\end{definition}

\begin{remark}
While our results about half-plane graphs require a choice of embedding, there is a large topological equivalence class of embeddings related via deformation.  As is usual in graph theory, we typically consider half-plane graphs up to this equivalence.  
\end{remark}

Bottom webs for both $\mathfrak{sl}_2$ and $\mathfrak{sl}_3$ are examples of half-plane graphs. In fact, a half-plane graph for which $V_{int} = \emptyset$ is a $\mathfrak{sl}_2$ web. Any bipartite half-plane graph admits a choice of edge orientations that makes it an $\mathfrak{sl}_3$ web. 
Figure \ref{fig:weirdo} shows an example of a half-plane graph that is not a web.

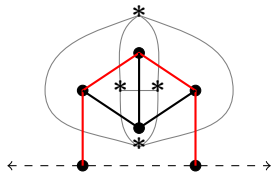
\begin{figure}[ht]
\begin{tikzpicture}[scale=.5]
\draw[gray] (2.5,.5) to[out=160,in=270] (0,2) to[out=90,in=200] (2.5,4);
\draw[gray] (2.5,.5) to[out=20,in=270] (5,2) to[out=90,in=-20] (2.5,4);
\draw[gray] (2.5,.5) to[out=150,in=270] (2,2);
\draw[gray] (2.5,.5) to[out=30,in=270] (3,2);
\draw[gray] (2.5,4) to[out=210,in=90] (2,2);
\draw[gray] (2.5,4) to[out=-30,in=90] (3,2);
\draw[gray] (2,2)--(3,2);

\draw[dashed, <-] (-1,0)--(2,0);
\draw[dashed, ->] (2,0)--(6,0);
\draw[fill=black,radius=4pt] (1,0)circle;
\draw[fill=black,radius=4pt] (4,0)circle;
\draw[fill=black,radius=4pt] (1,2)circle;
\draw[fill=black,radius=4pt] (4,2)circle;
\draw[fill=black,radius=4pt] (2.5,1)circle;
\draw[fill=black,radius=4pt] (2.5,3)circle;
\draw[thick, red] (1,0)--(1,2)--(2.5,3)--(4,2)--(4,0);
\draw[thick, black] (1,2)--(2.5,1)--(4,2);
\draw[thick,black] (2.5,1)--(2.5,3);

\node at (2.5,.5) {{\bf *}};
\node at (2.5,4) {{\bf *}};
\node at (2,2) {{\bf *}};
\node at (3,2) {{\bf *}};
\end{tikzpicture}
\caption{A half-plane graph and its dual.}\label{fig:weirdo}
\end{figure}

We constructed half-plane graphs so that they have an unbounded face at the top and the boundary line at the bottom.  Thus, we will now define the {\em depth} of faces in a half-plane graph to describe, informally, how deeply nested they are in the graph; the unbounded face has depth zero.  To formalize this definition, we use the dual graph $(G^+)^*$ of the planar graph $G^+$.

\begin{definition}
Let $G^+$ be a half-plane graph.  Given a face $F$ of $G^+$, the {\it depth of $F$} is the distance in the dual graph $(G^+)^*$
 from the vertex $F^*$ to the vertex corresponding to the unbounded face of $G^+$.\end{definition} 
 
\begin{remark}  Recall that distance between two vertices in a graph is the minimal path length between those vertices.  Depth is well-defined and finite because the dual $(G^+)^*$ of a planar graph is well-defined and connected.
\end{remark}

We can now define bands, which are the unions of faces at the same depth, as well as some other associated terms.

\begin{definition} Associated to a half-plane graph $G^+$, we define the following subsets of the upper half-plane.
\begin{itemize}
\item Given a nonnegative integer $d$, the {\it $d$-band of $G^+$} is the closure of the union of all faces of depth $d$. (We refer to the connected components of the $d$-band as the {\it components of the $d$-band}.)
\item Given a nonnegative integer $d$, the {\it $d$-arc of $G^+$} is the intersection of the $d$-band and the $(d-1)$-band. (We refer to the connected components of the $d$-arc as the {\it components of the $d$-arc}.)
\item The {\it band diagram for $G^+$} is the union of the $d$-arcs of $G^+$ over all nonnegative integers $d$ and is denoted $B(G^+)$.
\item The {\it anchored band diagram for $G^+$} is the band diagram for $G^+$ minus any closed components and is denoted $B_A(G^+)$. (If a band diagram has no closed components, in other words $B(G^+) = B_A(G^+)$, we say the band diagram is {\it anchored}.)
\end{itemize}
\end{definition}

As an example, the $0$-band of the half-plane graph shown in Figure \ref{fig:weirdo} is the closure of the unbounded face, and the $1$-band is closure of the union of all bounded faces. The $1$-arc, band diagram, and anchored band diagram are all equal in this case and consist of the red edges in Figure \ref{fig:weirdo}, namely the edges of the graph adjacent to both the unbounded face and any bounded face.

Say $i$ is a boundary vertex for a half-plane graph $G^+$. There is a unique edge incident to $i$ that separates two faces. Either these faces both have depth $d$ or the depths of these faces are $d$ and $d-1$ (in some order). We use the relationship between these faces to construct a boundary word for a half-plane graph.

\begin{definition}
Given a half-plane graph $G^+$, label its boundary vertices $+$, $0$, or $-$ based on whether depth of the face immediately to the left of the vertex is smaller, the same, or larger than the depth to the right.  The {\it boundary word} of a half-plane graph is the resulting string of symbols read from left to right. 
\end{definition}\label{signs}

Some properties of the boundary word associated to a web only use the weaker hypothesis of a half-plane graph.  The next lemma proves two such properties.

\begin{lemma}\label{cor:bandword}
The boundary word for a half-plane graph is $(+-)$-balanced and $(+-)$-Yamanouchi.
\end{lemma}
\begin{proof}
A half-plane graph with $n$ boundary vertices has $n+1$ faces that touch the $x$-axis. Consider the depths of these faces as we traverse the axis from left to right.  Since the first and last face is unbounded, depth begins and ends at zero.  By definition, depth is never negative and depth increases or decreases by one or zero at each boundary point. Hence the boundary word must be $(+-)$-balanced. 
Moreover, the number of $+$ symbols at any point in the boundary word must always be greater than or equal to the number of $-$ symbols, again because depth is nonnegative. We conclude that the boundary word is $(+-)$-Yamanouchi.
\end{proof}

The next series of results proves properties of the bands and arcs associated to half-plane graphs.  They culminate in Corollary \ref{cor:anchored}, which generalizes the statement that boundary words determine webs to say that boundary words determine the (anchored) band diagram of half-plane graphs.

We begin by showing that the $d$-arcs of a half-plane graph are isotopic to a disjoint unions of arcs and circles (hence their name).

\begin{proposition} \label{prop:manifold}
The $d$-arc of a half-plane graph is a 1-manifold with boundary on the $x$-axis.
\end{proposition}

\begin{proof}
The $d$ and $d-1$ bands intersect along edges of $G^+$ so the $d$-arc is necessarily the union of a subset of edges of $G^+$. We will show that if one edge incident to an internal vertex is part of the $d$-arc then exactly two edges at that vertex are part of the $d$-arc. It follows from this fact that each $d$-arc component is either a closed one-manifold passing through a subset of the internal vertices of $G^+$ or a path between boundary vertices possibly passing through intermediate internal vertices.

The depths of faces that share an edge in $G^+$ can differ by at most one. Therefore either \begin{itemize}
\item[(i)] the depths of the three faces around an internal vertex are all the same or 
\item[(ii)] two of the three faces around an internal vertex are the same and the third depth differs by one. There are two ways this can happen, both shown in Figure \ref{localdepths}. 
\end{itemize} 
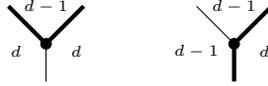
\begin{figure}[h]
\begin{tikzpicture}[scale=.5]
\draw[fill=black,radius=4pt] (0,0)circle;
\draw[ultra thick] (0,0)--(1,1);
\draw[ultra thick] (0,0)--(-1,1);
\draw (0,0)--(0,-1);
\node at (0,1) {\tiny{$d-1$}};
\node at (-.8,-.2) {\tiny{$d$}};
\node at (.8,-.2) {\tiny{$d$}};

\draw[fill=black,radius=4pt] (5,0)circle;
\draw[ultra thick] (5,0)--(6,1);
\draw (5,0)--(4,1);
\draw[ultra thick] (5,0)--(5,-1);
\node at (5,1) {\tiny{$d-1$}};
\node at (4,-.2) {\tiny{$d-1$}};
\node at (5.8,-.2) {\tiny{$d$}};
\end{tikzpicture}
\caption{The $d$-arc near an internal vertex}\label{localdepths}
\end{figure}

No edges incident to vertices of type (i) can be part of the $d$-arc. If an edge incident to a vertex of type (ii) is part of the $d$-arc, Figure \ref{localdepths} shows exactly two edges incident to that vertex are part of the $d$-arc.
\end{proof}

\begin{remark}
Trivalency is essential in Proposition \ref{prop:manifold}. To illustrate this, consider the example in Figure \ref{figure: nonexample of a half-plane graph}, which satisfies all properties of being a half-plane graph except trivalency of internal vertices. Since all edges separate faces of different depths, every edge of this graph would be part of the band diagram.

\begin{figure}[ht]
\begin{tikzpicture}[scale=.5]

\draw[dashed, <-] (-1,0)--(4,0);
\draw[dashed, ->] (4,0)--(6,0);
\draw[fill=black,radius=4pt] (1,0)circle;
\draw[fill=black,radius=4pt] (4,0)circle;
\draw[fill=black,radius=4pt] (1,2)circle;
\draw[fill=black,radius=4pt] (4,2)circle;
\draw[fill=black,radius=4pt] (2.5,1)circle;
\draw[fill=black,radius=4pt] (2.5,2)circle;
\draw[fill=black,radius=4pt] (2.5,3)circle;
\draw[thick, black] (1,0)--(1,2)--(2.5,1)--(4,2)--(4,0);
\draw[thick, black] (1,2)--(2.5,3)--(4,2);
\draw[thick,black] (1,2)--(4,2);

\node at (2.5,.5) {\tiny{$1$}};
\node at (2.5,1.5) {\tiny{$2$}};
\node at (2.5,2.5) {\tiny{$1$}};
\node at (1.5,3.25) {\tiny{$0$}};

\end{tikzpicture}
\caption{Nonexample of a half-plane graph}\label{figure: nonexample of a half-plane graph}
\end{figure}
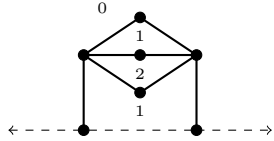
\end{remark}

Many half-plane graphs have anchored band diagrams, including reduced webs for $\mathfrak{sl}_2$ and $\mathfrak{sl}_3$ (see Subsection \ref{subsec:bandweb}).  In that case, the band diagram is a crossingless matching, as the following corollary states.

\begin{corollary}\label{cor:anchored}
The anchored band diagram of a half-plane graph is a crossingless matching on the boundary vertices labeled $+$ and $-$. 
\end{corollary}

\begin{proof}
By construction, the anchored band diagram is the union of nonintersecting, connected 1-manifolds each of which is homeomorphic to an interval and has both endpoints on the axis of the half-plane graph. This is exactly a crossingless matching.
\end{proof}

We now show that the band diagram can be used to identify the vertices labeled $+$ and $-$ in the boundary word.

\begin{proposition} \label{prop:bandwordendpoints}
The left (respectively right) endpoints of the band diagram for a half-plane graph are labeled $+$ (respectively $-$).
\end{proposition}

\begin{proof}
Let $c$ be a curve in the anchored band diagram for a half-plane graph, so $c$ is a component of the $d$-arc for some $d$.  The curve $c$ connects two points on the $x$-axis and splits the upper half-plane into two regions -- one above (defined as the region containing the unbounded face) and one below $c$. 

The $d$-arc is the intersection of the $d$-band and $(d-1)$-band, so the faces adjacent to $c$ on one side have depth $d-1$ and on the other side have depth $d$. Every path to the unbounded face from a face adjacent to and below $c$ must cross $c$ and enter a face adjacent to and above $c$. It follows that the faces of depth $d$ are below $c$.  This means the left endpoint of $c$ is labeled $+$ because depth increases while the right endpoint is labeled $-$ because depth decreases.
\end{proof}

Given a $(+-)$-balanced, $(+-)$-Yamanouchi word in the alphabet $\{+,-\}$, we observed earlier that there is a unique crossingless matching pairing these points whose left endpoints are labeled $+$ and right endpoints are labeled $-$. Together with Proposition \ref{prop:bandwordendpoints}, this gives the next corollary.

\begin{corollary}\label{cor:unique}
The anchored band diagram of a half-plane graph is determined up to isotopy by its boundary word.
\end{corollary}

The following definition gives a 1-dimensional projection of the band diagram that is useful in subsequent arguments.
 
\begin{definition}
Given a half-plane graph $G^+$, the {\it shadow} of $G^+$ is a nested collection of sets $$S(G^+): S_1(G^+)\supseteq S_2(G^+) \supseteq S_3(G^+) \supseteq \cdots $$ where $S_d(G^+)$ is the set of all points on the real axis of the upper half-plane belonging to a $k$-band for some $k\geq d$. 
\end{definition}

\begin{remark}
Several properties of the shadow of a half-plane graph $G^+$ follow immediately from its definition:
\begin{itemize}
\item If a band diagram is drawn using half circles, then $S_d(G^+)$ is the result of orthogonally projecting the $d$-arcs in the band diagram to the $x$- axis.
\item For all $d$, the shadow $S_d$ is the union of finitely many closed intervals.
\end{itemize}
\end{remark}

Figure \ref{fig:shadowex} shows a half-plane graph on $6$ boundary vertices with face depths labeled and the band diagram highlighted in red. The shadow for this graph is $S: [1,6] \supset [3,4]$.

We conclude with one final definition that will lead to a partial order on half-plane graphs.
\begin{definition}
If $G_1^+$ and $G_2^+$ are two half-plane graphs, we say the shadow of $G_2^+$ contains the shadow of $G_1^+$ if $S_d(G_1^+) \subset S_d(G_2^+)$ for all $d$. We denote this {\it $S(G_1^+)\subset S(G_2^+)$}.
\end{definition}

The next lemma follows immediately from the definitions of shadows and shadow containment. 

\begin{lemma}\label{lem:axisdepth}
Suppose $G_1^{+}$ and $G_2^{+}$ are two half-plane graphs with the same number of boundary vertices.  The shadows $S(G_1^+)\subset S(G_2^+)$ if and only if the depth at every point along the $x$-axis in $G_1^{+}$ is at most the depth at the same point on the $x$-axis in $G_2^{+}$.
\end{lemma}

\subsection{Band diagrams for webs}\label{subsec:bandweb}
Every bottom web is a half-plane graph. We can therefore associate a band diagram to each bottom web. We will see that for reduced $\mathfrak{sl}_2$ webs, the band diagram is trivial. For $\mathfrak{sl}_3$ webs, however, the band diagram has some interesting properties.

Recall that an $\mathfrak{sl}_2$ bottom web is a half-plane graph with no internal vertices. Further, each arc in an $\mathfrak{sl}_2$ web separates faces of different depth, so the band diagram for an $\mathfrak{sl}_2$ bottom web is the web itself. Trivially, then, a reduced $\mathfrak{sl}_2$ bottom web is completely determined by its band diagram.  The band diagram for an $\mathfrak{sl}_2$ web is anchored if and only if the web is reduced. 

Unlike $\mathfrak{sl}_2$ webs, webs for $\mathfrak{sl}_3$ can have edges between faces of the same depth, so their band diagrams are more complicated. Figure \ref{fig:sl3band} has an example of a reduced $\mathfrak{sl}_3$ web and its band diagram.

The next proposition shows that a reduced $\mathfrak{sl}_3$ web has an anchored band diagram, as mentioned in the discussion preceding Corollary \ref{cor:anchored}. 

\begin{proposition}\label{prop:sl3anch}
The band diagram for any reduced $\mathfrak{sl}_3$ bottom web is anchored.
\end{proposition}

\begin{proof}
The second author showed that the depth of a face of a reduced web is equal to the number of arcs above the corresponding face in the $M$-diagram \cite{T}. Therefore, we can label the local depths around vertices of reduced webs as shown in Figure \ref{fig:anchored} and highlight the  $d$-arc in red.
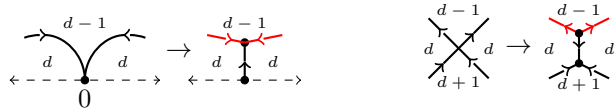
\begin{figure}[ht]
\raisebox{-12pt}{\begin{tikzpicture}[baseline=0cm, scale=.5]
\draw[style=dashed, <->] (-.5,0)--(3.5,0);
\node at (1.5,-.5) {$0$};
\draw[radius=.1, fill=black](1.5,0)circle;
\draw [style=thick, -<](1.5,0) arc (0:90: 1.2cm);
\draw[style=thick](.3,1.2) arc (90:110: 1.2cm);
\draw [style=thick, -<] (1.5,0) arc (180:90: 1.2cm);
\draw[style=thick](2.7,1.2) arc (90:70: 1.2cm);
\node at (1.5,1.5) {\tiny{$d-1$}};
\node at (.5,.5) {\tiny{$d$}};
\node at (2.5,.5){\tiny{$d$}};
\node at (4,.75) {$\rightarrow$};
\draw[style=dashed, <->] (4.25,0)--(7.25,0);
\draw[radius=.1, fill=black](5.75,0)circle;
\draw[radius=.1, fill=black](5.75,1)circle;
\draw[style=thick,->](5.75,0) -- (5.75,.5);
\draw[style=thick](5.75,.5) -- (5.75,1);
\draw[style=thick, -<, red] (5.75,1) -- (5.25,1.1);
\draw[style=thick, red](5.25,1.1) -- (4.75,1.2);
\draw[style=thick, -<, red] (5.75,1) -- (6.25,1.1);
\draw[style=thick, red](6.25,1.1) -- (6.75,1.2);
\node at (5.75,1.5) {\tiny{$d-1$}};
\node at (5,.5) {\tiny{$d$}};
\node at (6.5,.5){\tiny{$d$}};
\end{tikzpicture}}
\hspace{.5in}
\begin{tikzpicture}[baseline=0cm, scale=0.4]
\draw[style=thick, ->] (-1,-1)--(-.5,-.5);
\draw[style=thick, -<] (-1,1)--(-.5,.5);
\draw[style=thick, -<] (-.5,.5)--(.5,-.5);
\draw[style=thick, ->] (-.5,-.5)--(.5,.5);
\draw[style=thick] (.5,.5)--(1,1);
\draw[style=thick] (.5,-.5)--(1,-1);
\node at (0,1.25) {\tiny{$d-1$}};
\node at (0,-1.15) {\tiny{$d+1$}};
\node at (-1,0) {\tiny{$d$}};
\node at (1,0){\tiny{$d$}};
\node at (2,0) {$\rightarrow$};

\draw[style=thick, ->] (4,.5)--(4,0);
\draw[style=thick] (4,0)--(4,-.5);
\draw[style=thick, ->, red] (4,.5)--(3.5,.75);
\draw[style=thick, red] (3.5,.75)--(3,1);
\draw[style=thick, ->, red] (4,.5)--(4.5,.75);
\draw[style=thick, red] (4.5,.75)--(5,1);
\draw[style=thick, -<] (4,-.5)--(3.5,-.75);
\draw[style=thick] (3.5,-.75)--(3,-1);
\draw[style=thick, -<] (4,-.5)--(4.5,-.75);
\draw[style=thick] (4.5,-.75)--(5,-1);
\node at (4,1.25) {\tiny{$d-1$}};
\node at (4,-1.2) {\tiny{$d+1$}};
\node at (3,0) {\tiny{$d$}};
\node at (5,0){\tiny{$d$}};

\draw[ radius=.12, fill=black] (4,-.5) circle;
\draw[ radius=.12, fill=black] (4,.5) circle;

\end{tikzpicture}
\caption{Constructing a web from an oriented $M$-diagram}\label{fig:anchored}
\end{figure} 
Away from these local pictures, the $d$-arc for a web constructed from an $M$-diagram consists of smooth curves that have no points at which the tangent is vertical.  Any closed component of a $d$-arc must have at least two points at which the tangent to the curve is vertical. Hence it is impossible for the band diagram for a reduced $\mathfrak{sl}_3$ web to have closed components.
\end{proof}

\begin{figure}[ht]
\raisebox{-13pt}{\begin{tikzpicture}[baseline=0cm, scale=0.6]
\draw[style=dashed, <->] (2.5,0)--(8.5,0);
\draw[radius=.08, fill=black](3,0)circle;
\draw[radius=.08, fill=black](4,0)circle;
\draw[radius=.08, fill=black](5,0)circle;
\draw[radius=.08, fill=black](6,0)circle;
\draw[radius=.08, fill=black](7,0)circle;
\draw[radius=.08, fill=black](8,0)circle;
\draw[style=thick,->](4, 0) -- (4,.5);
\draw[style=thick](4,.5)--(4,1);
\draw[radius=.08, fill=black](4,1)circle;
\draw[style=thick,->](3,0)--(3.5,.5);
\draw[style=thick](3.5,.5)--(4,1);
\draw[style=thick,->](7, 0) -- (7,.5);
\draw[style=thick](7,.5)--(7,1);
\draw[radius=.08, fill=black](7,1)circle;
\draw[style=thick,->](8,0)--(7.5,.5);
\draw[style=thick](7.5,.5)--(7,1);
\draw[radius=.08, fill=black](5.5,1)circle;
\draw[style=thick,-<](5.5,1)--(5.5,1.5);
\draw[style=thick](5.5,1.5)--(5.5,2);
\draw[radius=.08, fill=black](5.5,2)circle;
\draw[style=thick,->](5,0)--(5.25,.5);
\draw[style=thick](5.25,.5)--(5.5,1);
\draw[style=thick,->](6,0)--(5.75,.5);
\draw[style=thick](5.75,.5)--(5.5,1);
\draw[style=thick,-<](4,1)--(4.75,1.5);
\draw[style=thick](4.75,1.5)--(5.5,2);
\draw[style=thick,-<](7,1)--(6.25,1.5);
\draw[style=thick](6.25,1.5)--(5.5,2);
\node at (3,1) {\tiny{$0$}};
\node at (3.65,.25) {\tiny{$1$}};
\node at (4.75,1) {\tiny{$1$}};
\node at (5.5,.25) {\tiny{$2$}};
\node at (6.25,1) {\tiny{$1$}};
\node at (7.35,.25) {\tiny{$1$}};
\node at (3,-.4) {$+$};
\node at (4,-.4) {0};
\node at (5,-.4) {$+$};
\node at (6,-.4) {$-$};
\node at (7,-.4) {0};
\node at (8,-.4) {$-$};
\end{tikzpicture}}\hspace{-.1in}
\raisebox{-13pt}{\begin{tikzpicture}[baseline=0cm, scale=0.6]
\node at (1.9,.75) {$\longleftrightarrow$};
\draw[style=dashed, <->] (2.5,0)--(8.5,0);
\draw[radius=.08, fill=black](3,0)circle;
\draw[radius=.08, fill=black](4,0)circle;
\draw[radius=.08, fill=black](5,0)circle;
\draw[radius=.08, fill=black](6,0)circle;
\draw[radius=.08, fill=black](7,0)circle;
\draw[radius=.08, fill=black](8,0)circle;
\draw (3,0) arc (180:0: 2.5cm);
\draw (5,0) arc (180:0: .5cm);
\node at (3,-.4) {$+$};
\node at (4,-.4) {0};
\node at (5,-.4) {$+$};
\node at (6,-.4) {$-$};
\node at (7,-.4) {0};
\node at (8,-.4) {$-$};
\end{tikzpicture}}
\caption{A band diagram for an $\mathfrak{sl}_3$ web}\label{fig:sl3band}
\end{figure} 

\begin{remark}\label{rem:sl3undet}
An immediate consequence of Proposition \ref{prop:sl3anch} together with Corollary \ref{cor:unique} is that  the band diagrams for reduced $\mathfrak{sl}_3$ bottom webs are bijective with their standard tableaux. This means that a reduced $\mathfrak{sl}_3$ bottom web is uniquely determined by its band diagram and therefore uniquely determined by its shadow. The algorithm to identify the tableau associated to a reduced web can be modified slightly to produce the tableau instead using the band diagram (or shadow) for a reduced web.
\end{remark}

\begin{remark}
Reduced $\mathfrak{sl}_3$ bottom webs with arbitrary boundary orientation are in bijection with certain {\em semistandard} Young tableaux of rectangular shape \cite{MR3119361}. These webs can be viewed as a kind of {\it contraction} of the reduced $\mathfrak{sl}_3$ webs with $3n$ boundary vertices that are all sources. Moreover, this contraction preserves the property of being anchored. Hence all reduced $\mathfrak{sl}_3$ bottom webs have anchored band diagrams. 
\end{remark}

\begin{remark}\label{rmk:unreduced}
The combinatorics of unreduced $\mathfrak{sl}_3$ webs is considerably more complicated than that of reduced webs. For instance, the boundary word for an unreduced web is guaranteed to be $(+,-)$-balanced and $(+-)$-Yamanouchi but is not necessarily$(+0)$- or $(0-)$-balanced nor is it necessarily $(+0)$- or $(0-)$-Yamanouchi. Figure \ref{fig:unreduced} shows an unreduced web that is neither $(+0)$-balanced nor $(0-)$-balanced (and thus not Yamanouchi in $(+0)$ or $(0-)$).

In addition, the band diagram for an unreduced $\mathfrak{sl}_3$ web need not necessarily be anchored. Figure \ref{fig:webreduc} shows an example of an unreduced web whose band diagram is not anchored as well as the unique reduced web that shares its boundary word (which in this case is the reduction of the unreduced web).

Band diagrams and boundary words uniquely identify reduced $\mathfrak{sl}_3$ bottom webs but not necessarily unreduced $\mathfrak{sl}_3$ bottom webs.  There exist distinct webs with the same boundary word and even distinct webs with the same band diagram. For instance, the webs in Figure \ref{fig:webreduc} have the same boundary word (though {\em not} the same band diagram). The webs in Figure \ref{fig:bandnotunique} have the same band diagram (and hence the same boundary word).

\begin{figure}[h]
\begin{tikzpicture}[scale=.75]
                              \draw[style=dashed, <->] (.5,0)--(6.5,0);
                              \draw[radius=.08, fill=black](1,0)circle;
                              \draw[radius=.08, fill=black](2,0)circle;
                             
                              \draw[radius=.08, fill=black](3,0)circle;
                              \draw[radius=.08, fill=black](4,0)circle;
                              \draw[radius=.08, fill=black](5,0)circle;
                             
                              \draw[radius=.08, fill=black](6,0)circle;
                              \draw[radius=.08, fill=black](2.5,.35)circle;
                              \draw[radius=.08, fill=black](2.5,.75)circle;
                              \draw[radius=.08, fill=black](5.5,.5)circle;
                              \draw[radius=.08, fill=black](4,.75)circle;
                                   \draw[radius=.08, fill=black](2.5,1.5)circle;
                              \draw[radius=.08, fill=black](4,1.5)circle;
                              \node at (1,-.4) {$+$};
\node at (2,-.4) {$+$};
\node at (3,-.4) {$0$};
\node at (4,-.4) {$-$};
\node at (5,-.4) {0};
\node at (6,-.4) {$-$};
                             
                              \begin{scope}[thick,decoration={
                                             markings,
                                             mark=at position 0.5 with {\arrow{>}}}
                              ]
                              \draw[postaction={decorate}, red] (2,0)--(2.5,.35);
                              \draw[postaction={decorate}] (3,0)--(2.5,.35);
                              \draw[postaction={decorate}, red] (2.5,.75)--(2.5,.35);
                               \draw[postaction={decorate}, red] (4,0)--(4,.75);
                                \draw[postaction={decorate}, red] (2.5,.75)--(4,.75);
                                \draw[postaction={decorate}] (2.5,.75)--(2.5,1.5);
                                 \draw[postaction={decorate}, red] (1,0)--(2.5,1.5);
                                  \draw[postaction={decorate}, red] (4,1.5)--(2.5,1.5);
                                  \draw[postaction={decorate}] (4,1.5)--(4,.75);
                                  \draw[postaction={decorate}] (5,0)--(5.5,.5);
                              \draw[postaction={decorate}, red] (6,0)--(5.5,.5);
                              \draw[postaction={decorate}, red] (4,1.5)--(5.5,.5);
                              \end{scope}
                              \end{tikzpicture}
                           \hspace{.5in}
                           \begin{tikzpicture}[scale=.75]
                              \draw[style=dashed, <->] (.5,0)--(6.5,0);
                              \draw[radius=.08, fill=black](1,0)circle;
                              \draw[radius=.08, fill=black](2,0)circle;
                             
                              \draw[radius=.08, fill=black](3,0)circle;
                              \draw[radius=.08, fill=black](4,0)circle;
                              \draw[radius=.08, fill=black](5,0)circle;
                             
                              \draw[radius=.08, fill=black](6,0)circle;
                              \draw[radius=.08, fill=black](3,.75)circle;
                              \draw[radius=.08, fill=black](4,1.5)circle;
                              \node at (1,-.4) {$+$};
\node at (2,-.4) {$+$};
\node at (3,-.4) {$0$};
\node at (4,-.4) {$-$};
\node at (5,-.4) {0};
\node at (6,-.4) {$-$};
                             
                              \begin{scope}[thick,decoration={
                                             markings,
                                             mark=at position 0.5 with {\arrow{>}}}
                              ]
                              \draw[postaction={decorate}, red] (1,0)--(4,1.5);
                              \draw[postaction={decorate}] (3,0)--(3,.75);
                              \draw[postaction={decorate}, red] (2,0)--(3,.75);
                              \draw[postaction={decorate}, red] (4,0)--(3,.75);
                              \draw[postaction={decorate}] (5,0)--(4,1.5);
                              \draw[postaction={decorate}, red] (6,0)--(4,1.5);
                              \end{scope}
                              \end{tikzpicture}   
 \caption{Two $\mathfrak{sl}_3$ bottom webs with the same band diagram.}\label{fig:bandnotunique}                          
\end{figure}

\end{remark}

We conclude this section with one final observation about $\mathfrak{sl}_3$ web shadows that will be useful in upcoming sections. Recall that an arbitrary $\mathfrak{sl}_3$ web can be expressed uniquely as a linear combination of reduced webs. This is accomplished via a sequence of reductions using the rules in Figure \ref{fig:webrelations}. Since all of these relations remove edges, the depths at every point along the boundary in webs after applying a reduction are always less than or equal to those of the original web. Hence, we obtain the following.

\begin{lemma}\label{lem:shadowred}
Let $w$ be an unreduced $\mathfrak{sl}_3$ bottom web with $w = \sum_{i=1}^t c_{\widetilde{w}_i} \widetilde{w}_i$ where $\widetilde{w}_i$ is reduced for all $i$. Then $S(\widetilde{w}_i) \subseteq S(w)$ for all $i$.
\end{lemma}

\section{A partial order on webs coming from tableaux} \label{sec:taborder}

In this section, we construct a partial order on webs coming from tableaux. In the next section, we introduce another partial order on webs coming from band diagrams and compare them. In the last section, we will apply these partial orders to describe the coefficients in the transition matrix between two bases for the same symmetric group representation: the Specht basis (which is indexed by standard Young tableaux) and the web basis. 

Let $m\in \mathbb{N}$ and let $S_{m}$ be the symmetric group on $m$ letters.  Denote by $s_i$ the simple transposition in $S_{m}$ that interchanges $i$ and $i+1$. 

The symmetric group acts on tableaux by permuting entries. Given a tableau $T$ of shape $\lambda \vdash m$ define $s_i\cdot T$ to be the tableau obtained by swapping entries $i$ and $i+1$ in $T$ and leaving all other entries fixed. Note that this action does not in general preserve the property of being standard (or row- or column-strict).

Using this symmetric group action, we may define a partial order on the set of standard Young tableaux of any fixed shape $\lambda$ as follows. 
\begin{definition}
Given standard Young tableaux  $T$ and $T'$ of shape $\lambda$, we say 
\begin{itemize}
\item $T \stackrel{s_i}{\rightarrow} T'$ if there exists $i$ such that $T' = s_i \cdot T$ and $i$ is in a row above $i+1$ in $T'$ and
\item $T\prec T'$ if there exist $T_1, \ldots, T_k$ such that $T\rightarrow T_1 \rightarrow \cdots \rightarrow T_{k}\rightarrow T'$.
\end{itemize}
We sometimes omit the label $s_i$ on the edges $T \stackrel{s_i}{\rightarrow} T'$ in our notation.
\end{definition}

The arrows $T\rightarrow T'$ form the edges of the Hasse diagram for this partial order. An example is shown in Figure \ref{fig:sl2poset}. (Note that we draw the Hasse diagram with edges directed down the page.) In the case of rectangular partitions, the Hasse diagram for this partial order is connected with unique maximal and minimal elements. The minimal element is the column-filled tableau and the maximal one is the row-filled tableau. Moreover, the Hasse diagram is a subgraph of the Bruhat graph, and therefore this poset on tableaux is ranked.  The rank function can be read directly from the tableau, as in the next definition \cite{MR1429587}.
\begin{definition}
Given a tableau $T\in SYT(\lambda)$, 
\begin{itemize}
\item the {\it column word} for $T$ is the sequence of integers obtained by reading down the columns of $T$ from left to right beginning with the leftmost column, and
\item a {\it descent} in the tableau $T$ is an instance where $i<j$ but $j$ appears before $i$ in the column word.
\end{itemize}
 \end{definition}

In other words, a descent in the tableau is a descent in its column word, according to the traditional definition for descents of permutations in one-line notation.

For instance, in Figure \ref{fig:sl2poset}, the column word for the top tableau is $123456$ and the column word for the bottom tableau is $142536$.  Thus the top tableau has no descents and the bottom tableau has 3 descents.   

Given $n\in \mathbb{N}$ we can use the bijections $\psi^{-1}:SYT(n,n)\rightarrow \mathcal{W}^R_{2n}$ and $\psi^{-1}:SYT(n,n,n)\rightarrow \mathcal{W}^R_{3n}$ together with the partial order on tableaux to endow webs with a partial order.  We denote this partial order $\prec_T$. In other words, given reduced webs $w,w'$ with the same boundary such that $\psi(w)=T$ and $\psi(w')=T'$, we say $w\prec_T w'$ if and only if $T\prec T'$. 

\begin{remark}\label{rmk:sl2Hasse}
In the $\mathfrak{sl}_2$ case, the Hasse diagram for $\prec_T$ has a straightforward description directly in terms of webs. Given reduced webs $w, w'\in \mathcal{W}_{2n}^R$ there exists an edge $w\rightarrow w'$ if and only if there are $i,j$, and $k$ such that $w$ and $w'$ are identical on all vertices except $j<i<i+1<k$ and $w$ pairs $(j,i)$ and $(i+1,k)$.  (By necessity $w'$ must then pair $(j,k)$ and $(i,i+1)$.) This partial order on $\mathfrak{sl}_2$ webs is extensively discussed in \cite{MR3920353}. Figure \ref{fig:sl2poset} shows the example of the full poset for all webs when $n=3$.
\end{remark}

\begin{figure}[ht]
\begin{tikzpicture}[baseline=0cm, scale=0.25]
\node at (0,1) {\small{$\young(123,456)$}};
\node at (0,10) {\small{$\young(124,356)$}};
\node at (-8,16) {\small{$\young(125,346)$}};
\node at (8,16) {\small{$\young(134,256)$}};
\node at (0,21.65) {\small{$\young(135,246)$}};
\draw[style=thick, ->] (0,8)--(0,3);
\draw[style=thick, ->] (8,14)--(2.75,10);
\draw[style=thick, ->] (-8,14)--(-2.75,10);
\draw[style=thick, <-] (8,18)--(2.75,22);
\draw[style=thick, <-] (-8,18)--(-2.75,22);
\node at (1,5) {\small{$s_3$}};
\node at (7,12) {\small{$s_2$}};
\node at (-7,12) {\small{$s_4$}};
\node at (7,20) {\small{$s_4$}};
\node at (-7,20) {\small{$s_2$}};
\end{tikzpicture}
\hspace{.25in}
\raisebox{110pt}{$\; \stackrel{\psi^{-1}}{\longrightarrow} \;$}
\hspace{.25in}
\raisebox{-12pt}{\begin{tikzpicture}[baseline=0cm, scale=0.28]
\draw[style=thick] (-2,1) to[out=90, in=180] (0,3) to[out=0, in=90] (2,1);
\draw[style=thick] (-1.3,1) to[out=90, in=180] (0,2.3) to[out=0, in=90] (1.3,1);
\draw[style=thick] (-.6,1) to[out=90, in=180] (0,1.6) to[out=0, in=90] (.6,1);

\draw[style=thick] (-2,8.25) to[out=90, in=180] (0,10.25) to[out=0, in=90] (2,8.25);
\draw[style=thick] (-1.6,8.25) to[out=90, in=180] (-.9,8.95) to[out=0, in=90] (-.2,8.25);
\draw[style=thick] (.2,8.25) to[out=90, in=180] (.9,8.95) to[out=0, in=90] (1.6,8.25);

\draw[style=thick] (-10,15) to[out=90, in=180] (-8.5, 16.5) to[out=0, in=90] (-7, 15);
\draw[style=thick] (-9, 15) to[out=90, in=180] (-8.5, 15.5) to[out=0, in=90] (-8, 15);
\draw[style=thick] (-6.5, 15) to[out=90, in=180] (-6, 15.5) to[out=0, in=90] (-5.5, 15);

\draw[style=thick] (10,15) to[out=90, in=0] (8.5, 16.5) to[out=180, in=90] (7, 15);
\draw[style=thick] (9, 15) to[out=90, in=0] (8.5, 15.5) to[out=180, in=90] (8, 15);
\draw[style=thick] (6.5, 15) to[out=90, in=0] (6, 15.5) to[out=180, in=90] (5.5, 15);

\draw[style=thick] (-.5, 21.75) to[out=90, in=180] (0, 22.5) to[out=0, in=90] (.5, 21.75);
\draw[style=thick] (-2, 21.75) to[out=90, in=180] (-1.5, 22.5) to[out=0, in=90] (-1, 21.75);
\draw[style=thick] (2, 21.75) to[out=90, in=0] (1.5, 22.5) to[out=180, in=90] (1, 21.75);

\draw[style=thick, ->] (0,7.5)--(0,3.5);
\draw[style=thick, ->] (8,14.5)--(2.75,10);
\draw[style=thick, ->] (-8,14.5)--(-2.75,10);
\draw[style=thick, <-] (8,17)--(2.5,21);
\draw[style=thick, <-] (-8,17)--(-2.5,21);
\node at (1,5.5) {\small{$s_3$}};
\node at (7,12) {\small{$s_2$}};
\node at (-7,12) {\small{$s_4$}};
\node at (7,20) {\small{$s_4$}};
\node at (-7,20) {\small{$s_2$}};
\end{tikzpicture}}
\caption{The poset structure on webs and tableaux corresponding to shape $(3,3)\vdash 6$}\label{fig:sl2poset}
\end{figure}
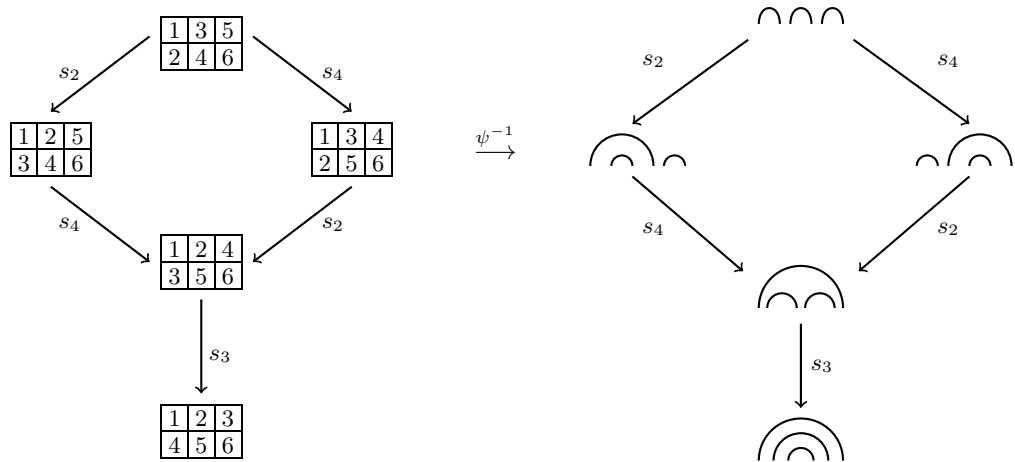

In the $\mathfrak{sl}_3$ case, there are three different combinatorial objects associated to $T\in SYT(n,n,n)$: a web, an $M$-diagram, and a band diagram. The bijection $\psi^{-1}$ and the partial order on tableaux therefore produce partial orders not just on webs but also on the $M$-diagrams and band diagrams associated to webs. In all cases, we denote this partial order by $\prec_T$. This is shown in Figure \ref{fig:sl3poset} in the case when $n=2$.

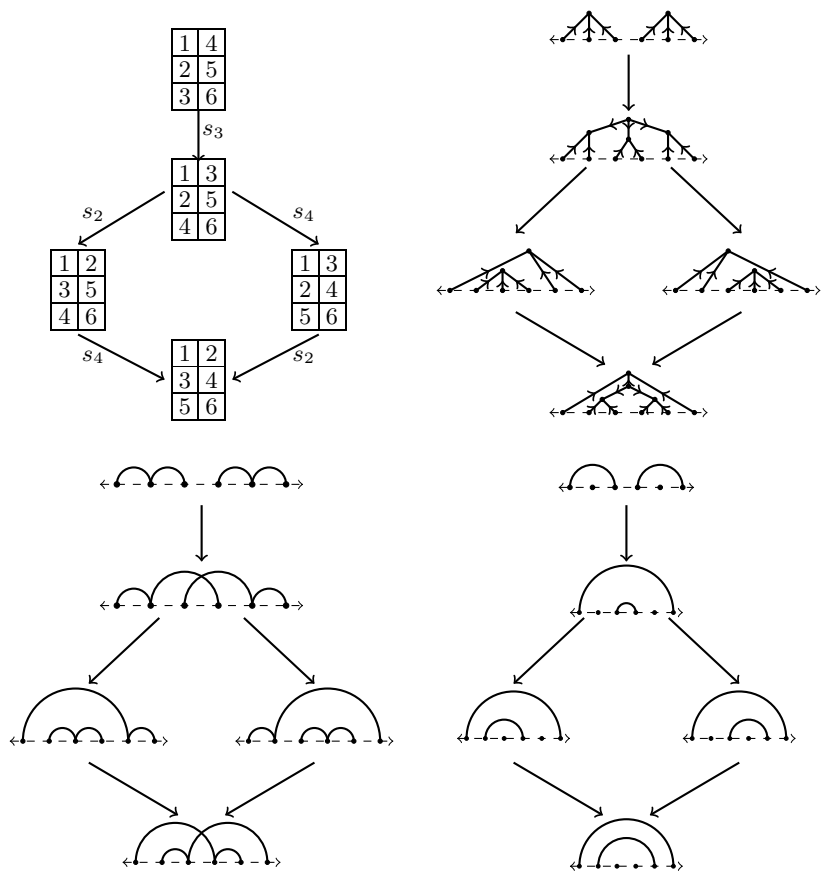
\begin{figure}
\centering
\subfloat{
               \begin{tikzpicture}[baseline=0cm, scale=0.2]
               \node at (0,1) {\small{$\young(12,34,56)$}};
               \node at (8,7) {\small{$\young(13,24,56)$}};
               \node at (-8,7) {\small{$\young(12,35,46)$}};
               \node at (0,13) {\small{$\young(13,25,46)$}};
               \node at (0,21.65) {\small{$\young(14,25,36)$}};
               \draw[style=thick, ->] (0,19)--(0,15.5);
               \draw[style=thick, ->] (8,4)--(2.25,1);
               \draw[style=thick, ->] (-8,4)--(-2.25,1);
               \draw[style=thick, <-] (8,10)--(2.25,13.5);
               \draw[style=thick, <-] (-8,10)--(-2.25,13.5);
               \node at (1,17.5) {\small{$s_3$}};
               \node at (7,12) {\small{$s_4$}};
               \node at (-7,12) {\small{$s_2$}};
               \node at (7,2.5) {\small{$s_2$}};
               \node at (-7,2.5) {\small{$s_4$}};
               \end{tikzpicture}
}\hspace{.5cm}
\subfloat{
              \raisebox{-10pt}{ \begin{tikzpicture}[baseline=0cm, scale=0.25]
               \node at (0,21) {\begin{tikzpicture}[scale=.35]
                              \draw[style=dashed, <->] (.5,0)--(6.5,0);
                              \draw[radius=.08, fill=black](1,0)circle;
                              \draw[radius=.08, fill=black](2,0)circle;
                             
                              \draw[radius=.08, fill=black](3,0)circle;
                              \draw[radius=.08, fill=black](4,0)circle;
                              \draw[radius=.08, fill=black](5,0)circle;
                             
                              \draw[radius=.08, fill=black](6,0)circle;
                              \draw[radius=.08, fill=black](2,1)circle;
                              \draw[radius=.08, fill=black](5,1)circle;
                             
                              \begin{scope}[thick,decoration={
                                             markings,
                                             mark=at position 0.5 with {\arrow{>}}}
                              ]
                              \draw[postaction={decorate}] (1,0)--(2,1);
                              \draw[postaction={decorate}] (3,0)--(2,1);
                              \draw[postaction={decorate}] (2,0)--(2,1);
                              \draw[postaction={decorate}] (4,0)--(5,1);
                              \draw[postaction={decorate}] (5,0)--(5,1);
                              \draw[postaction={decorate}] (6,0)--(5,1);
                              \end{scope}
                              \end{tikzpicture}};
               \node at (0,15) {\begin{tikzpicture}[scale=.35]
                              \draw[style=dashed, <->] (.5,0)--(6.5,0);
                              \draw[radius=.08, fill=black](1,0)circle;
                              \draw[radius=.08, fill=black](2,0)circle;
                             
                              \draw[radius=.08, fill=black](3,0)circle;
                              \draw[radius=.08, fill=black](4,0)circle;
                              \draw[radius=.08, fill=black](5,0)circle;
                             
                              \draw[radius=.08, fill=black](6,0)circle;
                              \draw[radius=.08, fill=black](2,1)circle;
                              \draw[radius=.08, fill=black](5,1)circle;
                             
                              \draw[radius=.08, fill=black](3.5,.75)circle;
                              \draw[radius=.08, fill=black](3.5,1.5)circle;
                             
                              \begin{scope}[thick,decoration={
                                             markings,
                                             mark=at position 0.5 with {\arrow{>}}}
                              ]
                              \draw[postaction={decorate}] (1,0)--(2,1);
                              \draw[postaction={decorate}] (3,0)--(3.5,.75);
                              \draw[postaction={decorate}] (2,0)--(2,1);
                              \draw[postaction={decorate}] (4,0)--(3.5,.75);
                              \draw[postaction={decorate}] (5,0)--(5,1);
                              \draw[postaction={decorate}] (6,0)--(5,1);
                              \draw[postaction={decorate}] (3.5,1.5)--(3.5,.75);
                              \draw[postaction={decorate}] (3.5,1.5)--(2,1);
                              \draw[postaction={decorate}] (3.5,1.5)--(5,1);
                              \end{scope}
                              \end{tikzpicture}};
               \node at (-6,8) {\begin{tikzpicture}[scale=.35]
                              \draw[style=dashed, <->] (.5,0)--(6.5,0);
                              \draw[radius=.08, fill=black](1,0)circle;
                              \draw[radius=.08, fill=black](2,0)circle;
                             
                              \draw[radius=.08, fill=black](3,0)circle;
                              \draw[radius=.08, fill=black](4,0)circle;
                              \draw[radius=.08, fill=black](5,0)circle;
                             
                              \draw[radius=.08, fill=black](6,0)circle;
                              \draw[radius=.08, fill=black](3,.75)circle;
                              \draw[radius=.08, fill=black](4,1.5)circle;
                             
                              \begin{scope}[thick,decoration={
                                             markings,
                                             mark=at position 0.5 with {\arrow{>}}}
                              ]
                              \draw[postaction={decorate}] (1,0)--(4,1.5);
                              \draw[postaction={decorate}] (3,0)--(3,.75);
                              \draw[postaction={decorate}] (2,0)--(3,.75);
                              \draw[postaction={decorate}] (4,0)--(3,.75);
                              \draw[postaction={decorate}] (5,0)--(4,1.5);
                              \draw[postaction={decorate}] (6,0)--(4,1.5);
                              \end{scope}
                              \end{tikzpicture}};
              
               \node at (6,8) {\begin{tikzpicture}[scale=.35]
                              \draw[style=dashed, <->] (.5,0)--(6.5,0);
                              \draw[radius=.08, fill=black](1,0)circle;
                              \draw[radius=.08, fill=black](2,0)circle;
                             
                              \draw[radius=.08, fill=black](3,0)circle;
                              \draw[radius=.08, fill=black](4,0)circle;
                              \draw[radius=.08, fill=black](5,0)circle;
                             
                              \draw[radius=.08, fill=black](6,0)circle;
                              \draw[radius=.08, fill=black](4,.75)circle;
                              \draw[radius=.08, fill=black](3,1.5)circle;
                             
                              \begin{scope}[thick,decoration={
                                             markings,
                                             mark=at position 0.5 with {\arrow{>}}}
                              ]
                              \draw[postaction={decorate}] (1,0)--(3,1.5);
                              \draw[postaction={decorate}] (3,0)--(4,.75);
                              \draw[postaction={decorate}] (2,0)--(3,1.5);
                              \draw[postaction={decorate}] (4,0)--(4,.75);
                              \draw[postaction={decorate}] (5,0)--(4,.75);
                              \draw[postaction={decorate}] (6,0)--(3,1.5);
                              \end{scope}
                              \end{tikzpicture}};
               \node at (0,1.5) {\begin{tikzpicture}[scale=.35]
                              \draw[style=dashed, <->] (.5,0)--(6.5,0);
                              \draw[radius=.08, fill=black](1,0)circle;
                              \draw[radius=.08, fill=black](2,0)circle;
                             
                              \draw[radius=.08, fill=black](3,0)circle;
                              \draw[radius=.08, fill=black](4,0)circle;
                              \draw[radius=.08, fill=black](5,0)circle;
                             
                              \draw[radius=.08, fill=black](6,0)circle;
                              \draw[radius=.08, fill=black](2.5,.5)circle;
                              \draw[radius=.08, fill=black](4.5,.5)circle;

                              \draw[radius=.08, fill=black](3.5,1)circle;
                              \draw[radius=.08, fill=black](3.5,1.5)circle;
                             
                              \begin{scope}[thick,decoration={
                                             markings,
                                             mark=at position 0.5 with {\arrow{>}}}
                              ]
                              \draw[postaction={decorate}] (1,0)--(3.5,1.5);
                              \draw[postaction={decorate}] (3,0)--(2.5,.5);
                              \draw[postaction={decorate}] (2,0)--(2.5,.5);
                              \draw[postaction={decorate}] (4,0)--(4.5,.5);
                              \draw[postaction={decorate}] (5,0)--(4.5,.5);
                              \draw[postaction={decorate}] (6,0)--(3.5,1.5);
                              \draw[postaction={decorate}] (3.5,1)--(2.5,.5);
                              \draw[postaction={decorate}] (3.5,1)--(4.5,.5);
                              \draw[postaction={decorate}] (3.5,1)--(3.5,1.5);
                              \end{scope}
                              \end{tikzpicture}};
              
               \draw[style=thick, ->] (0,19.5)--(0,16.5);
               \draw[style=thick, ->] (6,5.8)--(1.25,3);
               \draw[style=thick, ->] (-6,5.8)--(-1.25,3);
               \draw[style=thick, <-] (6,10)--(2.25,13.5);
               \draw[style=thick, <-] (-6,10)--(-2.25,13.5);
               \end{tikzpicture}}
}
 
\subfloat{
               \begin{tikzpicture}[baseline=0cm, scale=0.25]
               \node at (0,21) {\begin{tikzpicture}[scale=.45]
                              \draw[style=dashed, <->] (.5,0)--(6.5,0);
                              \draw[radius=.08, fill=black](1,0)circle;
                              \draw[radius=.08, fill=black](2,0)circle;
                              \draw[radius=.08, fill=black](3,0)circle;
                              \draw[radius=.08, fill=black](4,0)circle;
                              \draw[radius=.08, fill=black](5,0)circle;
                              \draw[radius=.08, fill=black](6,0)circle;
                              \draw[style=thick] (2,0) arc (0:180: .5cm);
                              \draw[style=thick] (3,0) arc (0:180: .5cm);
                              \draw[style=thick] (5,0) arc (0:180: .5cm);
                              \draw[style=thick] (6,0) arc (0:180: .5cm);
                              \end{tikzpicture}};
               \node at (0,15) {\begin{tikzpicture}[scale=.45]
                              \draw[style=dashed, <->] (.5,0)--(6.5,0);
                              \draw[radius=.08, fill=black](1,0)circle;
                              \draw[radius=.08, fill=black](2,0)circle;
                              \draw[radius=.08, fill=black](3,0)circle;
                              \draw[radius=.08, fill=black](4,0)circle;
                              \draw[radius=.08, fill=black](5,0)circle;
                              \draw[radius=.08, fill=black](6,0)circle;
                              \draw[style=thick] (2,0) arc (0:180: .5cm);
                              \draw[style=thick] (4,0) arc (0:180: 1cm);
                              \draw[style=thick] (5,0) arc (0:180: 1cm);
                              \draw[style=thick] (6,0) arc (0:180: .5cm);
                              \end{tikzpicture}};
               \node at (-6,8.3) {\begin{tikzpicture}[scale=.35]
                              \draw[style=dashed, <->] (.5,0)--(6.5,0);
                              \draw[radius=.08, fill=black](1,0)circle;
                              \draw[radius=.08, fill=black](2,0)circle;
                              \draw[radius=.08, fill=black](3,0)circle;
                              \draw[radius=.08, fill=black](4,0)circle;
                              \draw[radius=.08, fill=black](5,0)circle;
                              \draw[radius=.08, fill=black](6,0)circle;
                              \draw[style=thick] (5,0) arc (0:180: 2cm);
                              \draw[style=thick] (3,0) arc (0:180: .5cm);
                              \draw[style=thick] (4,0) arc (0:180: .5cm);
                              \draw[style=thick] (6,0) arc (0:180: .5cm);
                              \end{tikzpicture}};
              
               \node at (6,8.3) {\begin{tikzpicture}[scale=.35]
                              \draw[style=dashed, <->] (.5,0)--(6.5,0);
                              \draw[radius=.08, fill=black](1,0)circle;
                              \draw[radius=.08, fill=black](2,0)circle;
                              \draw[radius=.08, fill=black](3,0)circle;
                              \draw[radius=.08, fill=black](4,0)circle;
                              \draw[radius=.08, fill=black](5,0)circle;
                              \draw[radius=.08, fill=black](6,0)circle;
                              \draw[style=thick] (2,0) arc (0:180: .5cm);
                              \draw[style=thick] (6,0) arc (0:180: 2cm);
                              \draw[style=thick] (5,0) arc (0:180: .5cm);
                              \draw[style=thick] (4,0) arc (0:180: .5cm);
                              \end{tikzpicture}};
               \node at (0,1.5) {\begin{tikzpicture}[scale=.35]
                              \draw[style=dashed, <->] (.5,0)--(6.5,0);
                              \draw[radius=.08, fill=black](1,0)circle;
                              \draw[radius=.08, fill=black](2,0)circle;
                              \draw[radius=.08, fill=black](3,0)circle;
                              \draw[radius=.08, fill=black](4,0)circle;
                              \draw[radius=.08, fill=black](5,0)circle;
                              \draw[radius=.08, fill=black](6,0)circle;
                              \draw[style=thick] (4,0) arc (0:180: 1.5cm);
                              \draw[style=thick] (6,0) arc (0:180: 1.5cm);
                              \draw[style=thick] (5,0) arc (0:180: .5cm);
                              \draw[style=thick] (3,0) arc (0:180: .5cm);
                              \end{tikzpicture}};
              
               \draw[style=thick, ->] (0,19.5)--(0,16.5);
               \draw[style=thick, ->] (6,5.8)--(1.25,3);
               \draw[style=thick, ->] (-6,5.8)--(-1.25,3);
               \draw[style=thick, <-] (6,10)--(2.25,13.5);
               \draw[style=thick, <-] (-6,10)--(-2.25,13.5);
\end{tikzpicture}
}\hspace{.25cm}
\subfloat{
               \begin{tikzpicture}[baseline=0cm, scale=0.25]
               \node at (0,21) {\begin{tikzpicture}[scale=.3]
                              \draw[style=dashed, <->] (.5,0)--(6.5,0);
                              \draw[radius=.1, fill=black](1,0)circle;
                              \draw[radius=.1, fill=black](2,0)circle;
                              \draw[radius=.1, fill=black](3,0)circle;
                              \draw[radius=.1, fill=black](4,0)circle;
                              \draw[radius=.1, fill=black](5,0)circle;
                              \draw[radius=.1, fill=black](6,0)circle;
                              \draw[style=thick] (3,0) arc (0:180: 1cm);
                              \draw[style=thick] (6,0) arc (0:180: 1cm);
                              \end{tikzpicture}};
               \node at (0,15) {\begin{tikzpicture}[scale=.25]
                              \draw[style=dashed, <->] (.5,0)--(6.5,0);
                              \draw[radius=.1, fill=black](1,0)circle;
                              \draw[radius=.1, fill=black](2,0)circle;
                              \draw[radius=.1, fill=black](3,0)circle;
                              \draw[radius=.1, fill=black](4,0)circle;
                              \draw[radius=.1, fill=black](5,0)circle;
                              \draw[radius=.1, fill=black](6,0)circle;
                              \draw[style=thick] (4,0) arc (0:180: .5cm);
                              \draw[style=thick] (6,0) arc (0:180: 2.5cm);
                              \end{tikzpicture}};
               \node at (-6,8.3) {\begin{tikzpicture}[scale=.25]
                              \draw[style=dashed, <->] (.5,0)--(6.5,0);
                              \draw[radius=.1, fill=black](1,0)circle;
                              \draw[radius=.1, fill=black](2,0)circle;
                              \draw[radius=.1, fill=black](3,0)circle;
                              \draw[radius=.1, fill=black](4,0)circle;
                              \draw[radius=.1, fill=black](5,0)circle;
                              \draw[radius=.1, fill=black](6,0)circle;
                              \draw[style=thick] (4,0) arc (0:180: 1cm);
                              \draw[style=thick] (6,0) arc (0:180: 2.5cm);
                              \end{tikzpicture}};
              
               \node at (6,8.3) {\begin{tikzpicture}[scale=.25]
                              \draw[style=dashed, <->] (.5,0)--(6.5,0);
                              \draw[radius=.1, fill=black](1,0)circle;
                              \draw[radius=.1, fill=black](2,0)circle;
                              \draw[radius=.1, fill=black](3,0)circle;
                              \draw[radius=.1, fill=black](4,0)circle;
                              \draw[radius=.1, fill=black](5,0)circle;
                              \draw[radius=.1, fill=black](6,0)circle;
                              \draw[style=thick] (5,0) arc (0:180: 1cm);
                              \draw[style=thick] (6,0) arc (0:180: 2.5cm);
                              \end{tikzpicture}};
               \node at (0,1.5) {\begin{tikzpicture}[scale=.25]
                              \draw[style=dashed, <->] (.5,0)--(6.5,0);
                              \draw[radius=.1, fill=black](1,0)circle;
                              \draw[radius=.1, fill=black](2,0)circle;
                              \draw[radius=.1, fill=black](3,0)circle;
                              \draw[radius=.1, fill=black](4,0)circle;
                              \draw[radius=.1, fill=black](5,0)circle;
                              \draw[radius=.1, fill=black](6,0)circle;
                              \draw[style=thick] (5,0) arc (0:180: 1.5cm);
                              \draw[style=thick] (6,0) arc (0:180: 2.5cm);
                              \end{tikzpicture}};
              
               \draw[style=thick, ->] (0,19.5)--(0,16.5);
               \draw[style=thick, ->] (6,5.8)--(1.25,3);
               \draw[style=thick, ->] (-6,5.8)--(-1.25,3);
               \draw[style=thick, <-] (6,10)--(2.25,13.5);
               \draw[style=thick, <-] (-6,10)--(-2.25,13.5);
\end{tikzpicture}} \hspace{1cm}
\caption{The poset structure on webs, tableaux, $M$-diagrams, and band diagrams in the $(2,2,2)\vdash 6$ case}\label{fig:sl3poset}
\end{figure}

Since the tableaux poset is ranked, the partial order $\prec_T$ is also a ranked partial order. The rank function for $\prec_T$ can be described explicitly in terms of webs, $M$-diagrams, and band diagrams; we denote each by $r$. The following definition generalizes the notion of nesting number for $\mathfrak{sl}_2$ webs \cite{MR3920353}.

\begin{definition}
Let $D$ be an $\mathfrak{sl}_2$ web, an $M$-diagram for an $\mathfrak{sl}_3$ web, or a band diagram. In each case, $D$ has an associated set of arcs denoted $a\in D$. 
\begin{itemize}
\item For each arc $a\in D$, the {\it nesting number $n(a)$ of $a$} is the number of arcs in $D$ above $a$ that do not intersect $a$.
\item The {\it nesting number $n(D)$ for $D$} is the sum of the nesting numbers of the arcs in $D$. i.e. $n(D) = \sum_{a\in D} n(a)$. 
\end{itemize}
\end{definition}

\begin{definition}
Let $B$ be a band diagram for a half-plane graph with $m$ boundary vertices.
\begin{itemize}
\item For $1\leq i\leq m$, we define the {\it dot depth $d(i)$ of $i$ in $B$} to be $0$ if $i$ is the boundary of an arc in $B$ and the number of arcs of $B$ above $i$ otherwise.
\item The {\it dot depth $d(B)$ of $B$} is the sum $d(B) = \sum_{i=1}^{m} d(i)$. 
\end{itemize}
\end{definition}

We use the term ``dots" because we are counting a statistic of the unpaired vertices along the $x$-axis.

\begin{definition}
Given an $M$-diagram $M$ for an $\mathfrak{sl}_3$ web, we define the {\it crossing number of $M$} to be the number of intersections of arcs in the interior of the diagram. 
\end{definition}

In particular, this definition is written so that a boundary point labeled $0$ in the boundary word for an $M$-diagram does not contribute to the crossing number.

\begin{definition}
Let $w$ be an $\mathfrak{sl}_3$ bottom web with faces denoted $F\in w$.
\begin{itemize}
\item We denote the depth of $F$ in $w$ by $d(F)$.
\item The face $F$ is {\it excluded} if it is the unique highest depth face around an internal vertex of $w$. Denote by $E(w)$ the set of excluded faces.
\end{itemize}
\end{definition}

In other words, the face $F$ is excluded only for interior vertices of the form shown in the right-hand example of Figure \ref{localdepths}. Note that not all internal vertices of $w$ are adjacent to an excluded face: indeed, of all the webs shown in Figure \ref{fig:sl3poset}, only one internal vertex is adjacent to an excluded face (found in the second-highest web in the poset).

In the next two lemmas, these quantities can be used to compute the rank function $r$ explicitly for $\prec_T$ in the varying contexts of webs, band diagrams, and $M$-diagrams. The following lemma states that the rank function for $\mathfrak{sl}_2$ webs is simply the nesting number.  This is proven in \cite[Corollary 4.4]{MR3920353}, so we omit the proof. 

\begin{lemma}\label{lem:sl2rank}
If $w\in \mathcal{W}_{2n}^R$ then $r(w) = n(w)$. In other words, the rank of $w$ is its nesting number.
\end{lemma}

The next lemma characterizes the rank function for band diagrams, $M$-diagrams, and webs for $\mathfrak{sl}_3$ in terms of the nesting number, dot depth, and excluded faces defined above.

\begin{lemma}\label{lem:sl3rank}
Let $w\in \mathcal{W}_{3n}^R$ and let $B$ and $M$ be the band and $M$-diagrams associated to $w$. By definition $r(w) = r(B) = r(M)$. Further, the following are all rank functions for $\prec_T$.
\begin{itemize}
\item On the level of $M$-diagrams, $r(M) = n(M) + c(M)$. In other words, rank is the sum of the nesting number and crossing number of an $M$-diagram.
\item On the level of band diagrams, $r(B) = n(B) + d(B) -n$. In other words, rank is the sum of the nesting number and dot depth of a band diagram minus the number of arcs.
\item Say $w$ has $s(w)$ internal source vertices. Then we have the following rank function on the level of webs. 
$$r(w) = s(w)+\sum_{\substack{F\in w\\ F\notin E(w)}} d(F)$$
In other words, rank is the sum of the depths of the non-excluded faces plus the number of internal source vertices.
\end{itemize}
\end{lemma}

\begin{proof}
The formulas for the rank function on $M$-diagrams and band diagrams follow from straightforward inductive arguments. The formula on webs comes from the formula on $M$-diagrams, together with the observations that:
\begin{itemize}
    \item crossings in $M$-diagrams correspond to trivalent source vertices in the associated web, and 
    \item arcs in the $M$-diagram are in bijection with bounded faces of $w$ that are not excluded. 
\end{itemize}
 To see this, note that all but one face directly below an arc of an $M$-diagram is also directly below another arc. Any face directly below two arcs at a crossing is excluded as can be seen in Figure \ref{fig:anchored}. 
\end{proof}
\begin{remark}
The rank functions on band diagrams and webs is well-defined for unreduced webs though these two functions do not produce the same value in this more general setting. For instance, the web in Figure \ref{fig:unreduced} has $r(B) = 5$ and $r(w) = 16$ . The rank function on band diagrams can be generalized even further to any half-plane graph though it is not clear whether the term $-n$ should continue to count the number of arcs in that case.
\end{remark}

\section{A partial order on webs coming from band diagrams}\label{sec:band}
Recall that if $G_1^+$ and $G_2^+$ are two half-plane graphs then we say $S(G_1^+)\subset S(G_2^+)$ if $S_d(G_1^+) \subset S_d(G_2^+)$ for all $d$. This notion of shadow containment leads to a partial order on webs. This section defines this partial order and proves several results about it, especially comparing it to the tableau partial order.

\begin{definition}
Let $w,w'$ be a pair of reduced bottom webs of the same type, either $\mathcal{W}_{2n}^R$ or $\mathcal{W}_{3n}^R$. We say $w \prec_S w'$ if $S(w) \subset S(w')$. We refer to this as the {\it shadow containment partial order}.
\end{definition}

\begin{example} \label{example: shadow containment for sl2 webs}
For example, consider the two $\mathfrak{sl}_2$ webs $w', w$ where $w'$ consists of arcs $(1,4)$ and $(2,3)$ and $w$ consists of arcs $(1,2)$ and $(3,4)$. Then $S_1(w') = [1,4]$, $S_2(w') = [2,3]$, and $S_d(w') = \emptyset$ for all $d>2$. For $w$, we have $S_1(w) = [1,2]\cup [3,4]$ and $S_d(w) = \emptyset$ for all $d>1$. It follows that $S(w) \subset S(w')$ so $w\prec_S w'$.  
\end{example}

\begin{remark}
We cannot extend the partial order $\prec_S$ to all half-plane graphs or even to unreduced webs because it violates the antisymmetry property of posets. Indeed, these larger sets contain distinct half-plane graphs that have the same shadow. See Remark \ref{rmk:unreduced} above for more detail on this phenomenon for unreduced webs. One could bypass this issue by defining $\prec_S$ on equivalence classes of half-plane  graphs grouped according to their shadows.
\end{remark}

The following lemma is an immediate consequence of the definition of the shadow of a half-plane graph together with Lemma \ref{lem:axisdepth}.  It will be useful when we analyze shadow containment.

\begin{lemma}\label{lem:depthshad}
Given two webs $w$ and $w'$, we have $w\prec_S w'$ if and only if the depth at every point on the boundary of $w$ is no greater than the depth at the corresponding point in $w'$.
\end{lemma}

In the rest of this section, we compare $\prec_T$ to $\prec_S$ on the sets $\mathcal{W}_{2n}^R$ and $\mathcal{W}_{3n}^R$. We will show that these partial orders coincide for reduced $\mathfrak{sl}_2$ webs but differ for $\mathfrak{sl}_3$ webs. In fact, the shadow containment partial order on $\mathfrak{sl}_3$ webs fails to be ranked.

\subsection{Shadows of $\mathfrak{sl}_2$ webs}

In this section, we analyze $\prec_S$ for $\mathfrak{sl}_2$ webs and show it coincides with $\prec_T$.  {\bf All webs in this subsection are assumed to be {\em reduced} $\mathfrak{sl}_2$ bottom webs.} This means they are crossingless matchings with no additional closed components. 

The next lemma makes more rigorous and general an observation stated in Example~\ref{example: shadow containment for sl2 webs}.

\begin{lemma}
Let $(i_1, j_1), \ldots ,(i_t, j_t)$ be the list of all arcs with nesting number $d$ in a web $w$. Then $S_{d+1}(w) = [i_1, j_1] \cup \cdots \cup [i_t, j_t]$.  In other words, the intervals in $S(w)$ are in bijection with the arcs of $w$, and they are partitioned according to nesting number.
\end{lemma}

\begin{proof}
As observed in Subsection \ref{subsec:bandweb}, for a reduced $\mathfrak{sl}_2$ web $w$ with band diagram $B(w)$, we have $B(w)=w$. Intervals in $S(w)$ correspond to endpoints of the arcs in $B(w)$. In this case, then, the intervals in $S(w)$ also correspond to the arcs in $w$. Since nesting number is defined the same way for $\mathfrak{sl}_2$ webs and band diagrams, arcs of nesting number $d$ -- which separate faces of depth $d$ and $d+1$ -- correspond to intervals in $S_{d+1}(w)$.
\end{proof}

Given $w\in \mathcal{W}_{2n}^R$ recall that $r(w)$ is the sum of the nesting numbers of the arcs of $w$.  We will prove a sequence of lemmas whose ultimate goal is to show that $\prec_T$ and $\prec_S$ coincide for reduced $\mathfrak{sl}_2$ bottom webs.  We begin by showing that $w \prec_S w'$ implies the same order for nesting numbers.

\begin{lemma}\label{shadowheight}
Let $w,w'\in \mathcal{W}_{2n}^R$. If $w\prec_S w'$ then $r(w) < r(w')$.
\end{lemma}

\begin{proof}
Let $w,w'\in \mathcal{W}_{2n}^R$ with $S(w) \subset S(w')$. Rather than summing depth over arcs of a web, we will compute the rank function $r$ by summing over endpoints of the web, adding $\frac{d}{2}$ for each endpoint of an arc with nesting number $d$. 

Every arc in an $\mathfrak{sl}_2$ web separates two neighboring faces with depths that differ by one, and the smaller of those depths is the nesting number of the arc.  By Lemma \ref{lem:depthshad}, since $S(w)\subset S(w')$, the depth at any point on the boundary of $w$ is no larger than the depth at the same point in $w'$. It follows that $r(w')\geq r(w)$.

For the rest of this proof, we prove that in fact $r(w')>r(w)$.  
Since the shadow containment is proper, there is at least one integral boundary point around which the shadows of $w$ and $w'$ differ.  Suppose point $i$ is the first boundary point at which the shadows differ. Every boundary point of a $\mathfrak{sl}_2$ web is either the start or end of an arc; the only way for the shadows of $w, w'$ to coincide at $1, 2, \ldots, i-1$ but not at $i$ is for one of $w, w'$ to have an arc start at $i$ while the other has an arc end at $i$.  Only one of these preserves the containment $S(w) \subseteq S(w')$, namely the depth in $w'$ at $i-\frac{1}{2}$ is $d-1$ and at $ i+ \frac{1}{2}$ is $d$, while the depth in $w$ at $i-\frac{1}{2}$ is $d-1$ and at $i+\frac{1}{2}$ is $d-2$.  So $i$ contributes $\frac{d-1}{2}$ to $r(w')$ and $\frac{d-2}{2}$ to $r(w)$.   
 \end{proof}

We now show that the tableau partial order implies shadow containment.

\begin{lemma}\label{cont1}
If $w \prec_T w'$ then $w \prec_S w'$.
\end{lemma}

\begin{proof}
If $w \prec_T w'$ then there is a directed path from $w$ to $w'$ in the Hasse diagram for $\prec_T$. We will show that for each edge $w\rightarrow w'$ in the Hasse diagram for $\prec_T$ we also have $w \prec_S w'$. Once this is established, the lemma follows by induction on path length. 

If $w\rightarrow w'$ then $w$ and $w'$ are identical except for two arcs on four integral boundary points $j<i<i+1<k$ where $w$ has arcs $(j,i)$ and $(i+1, k)$ and $w'$ has arcs $(j,k)$ and $(i,i+1)$. Figure \ref{fig:Hasse-edge} illustrates this relationship using red regions to represent the locations of arcs common to $w$ and $w'$.  

\begin{figure}[h]
\begin{tikzpicture}[scale=.5]
\draw[style=dashed, <->] (0,0)--(7,0);

\draw[radius=.1, fill=black](1.1,0)circle;
\draw[radius=.1, fill=black](3.1,0)circle;
\draw[radius=.1, fill=black](3.9,0)circle;
\draw[radius=.1, fill=black](5.9,0)circle;
\draw[style=thick] (1.1,0) arc (180:0: 1cm); 
\draw[style=thick] (3.9,0) arc (180:0: 1cm); 
\draw[style=thick, color=black, fill=red] (.5,0)--(.9,0)--(.9,2.4)--(6.1,2.4)--(6.1,0)--(6.5,0)--(6.5,2.7)--(.5,2.7)--(.5,0);
\draw[style=thick, color=black, fill=red] (1.4,0)--(2.7,0)--(2.7,.35)--(1.4,.35)--(1.4,0);
\draw[style=thick, color=black, fill=red] (4.3,0)--(5.6,0)--(5.6,.35)--(4.3,.35)--(4.3,0);
\node at (2.05,.58) {\tiny{$d+1$}};
\node at (4.9,.58) {\tiny{$d+1$}};
\node at (3.5,1.2) {\tiny{$d$}};
\node at (1,-.35) {\tiny{$j$}};
\node at (3,-.35) {\tiny{$i$}};
\node at (4,-.35) {\tiny{$i+1$}};
\node at (6,-.35) {\tiny{$k$}};
\node at (3.5,-1.2) {{$w$}};
\node at (7.5,-1.2) {\small{$\longrightarrow$}};

\draw[style=dashed, <->] (8,0)--(15,0);

\draw[radius=.1, fill=black](12.2,0)circle;
\draw[radius=.1, fill=black](13.8,0)circle;
\draw[radius=.1, fill=black](9.2,0)circle;
\draw[radius=.1, fill=black](10.8,0)circle;
\draw[style=thick] (9.2,0) arc (180:0: 2.3cm); 
\draw[style=thick] (10.8,0) arc (180:0: .7cm); 
\draw[style=thick, color=black, fill=red] (8.5,0)--(8.9,0)--(8.9,2.4)--(14.1,2.4)--(14.1,0)--(14.5,0)--(14.5,2.7)--(8.5,2.7)--(8.5,0);
\draw[style=thick, color=black, fill=red] (9.4,0)--(10.6,0)--(10.6,.35)--(9.4,.35)--(9.4,0);
\draw[style=thick, color=black, fill=red] (12.4,0)--(13.6,0)--(13.6,.35)--(12.4,.35)--(12.4,0);
\node at (11.5,.2) {\tiny{$d+2$}};
\node at (11.5,1.2) {\tiny{$d+1$}};
\node at (9.5,2) {\tiny{$d$}};
\node at (13.5,2) {\tiny{$d$}};
\node at (9.25,-.35) {\tiny{$j$}};
\node at (10.75,-.35) {\tiny{$i$}};
\node at (12.25,-.35) {\tiny{$i+1$}};
\node at (13.75,-.35) {\tiny{$k$}};
\node at (11.5,-1.2) {{$w'$}};
\end{tikzpicture}
\caption{Depth change along an edge in the Hasse diagram} \label{fig:Hasse-edge}
\end{figure}
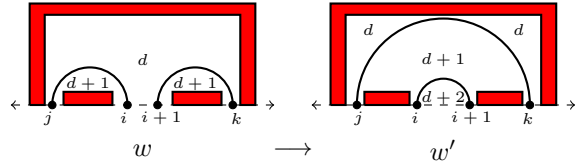
By construction, the depths along the boundaries of $w'$ and $w$ are the same except between boundary points $i$ and $i+1$. If the depth between $i$ and $i+1$ in $w$ is $d$ then it is $d+2$ in $w'$.  It follows that $S_i(w) = S_i(w')$ for $i \neq d+1,d+2$ and $S_i(w)\cup [i,i+1] = S_i(w')$ for $i=d+1, d+2$.  Hence $w\prec_S w'$. 
\end{proof}

We now prove the converse of the previous result: shadow containment implies the tableau partial order.

\begin{lemma}\label{cont2}
If $w \prec_S w'$ then $w \prec_T w'$.
\end{lemma}

\begin{proof}
We will argue via double induction. The claim is trivial for all pairs of webs in $\mathcal{W}_{2n}^R$ when $n= 1$ or $2$ and is not much more complicated when $n=3$. 

Now assume that if $w, w'$ is a pair of webs which each have fewer than $n$ arcs and if $w\prec_S w'$ then $w \prec_T w'$. Let $w$ and $w'$ be webs with $n$ arcs such that $w \prec_S w'$. Consider the arc $(1,j)$ in $w'$ which necessarily has nesting number $0$.  Note that we can write the arcs of nesting number $0$ in $w$ as $(1,t_1), (t_1+1,t_2),(t_2+1,t_3),\ldots$. Suppose $k+1$ of these arcs in $w$ have both endpoints in the interval $[1,j]$. One of the $t_i$ must be $j$ since otherwise there exists $i$ with $t_i +1 \leq j < t_{i+1}$.  In particular, $j+\frac{1}{2} \in S_1(w)$ but $j+\frac{1}{2} \not \in S_1(w')$, contradicting the assumption of shadow containment. 

Suppose $k=0$ so $t_1=j$.  In this case, we can remove the arc $(1,j)$ from both matchings $w, w'$ and maintain shadow containment. What remains is a pair of webs with $n-1$ arcs, and the inductive hypothesis then implies $w\prec_T w'$.

Now suppose the arcs of nesting number $0$ in $w$ with both endpoints in the interval $[1,j]$ are $(1, t_1), (t_1+1, t_2), (t_2+1, t_3), \ldots (t_k+1, t_{k+2}) = (t_k+1,j)$ for some $k\geq 1$. 
Let $w''$ be the web that differs from $w$ only by having nested arcs $(1,t_2)$ and $(t_1, t_1+1)$ as in Figure \ref{fig:localchanges}. By the discussion of $\prec_T$ for $\mathfrak{sl}_2$ in Remark \ref{rmk:sl2Hasse}, there is an edge $w \rightarrow w''$ in the Hasse diagram for $\prec_T$. Lemma \ref{cont1} proved that $w\prec_S w''$. In fact, we know exactly how the shadows of $w$ and $w''$ differ: the shadows are exactly the same except the interval $[t_1, t_1+1]$ has depth $2$ in $w''$ but depth $0$ in $w$.  Therefore $S_i(w'') = S_i(w) \cup [t_1, t_1+1]$ for $i=1,2$ and $S_i(w'') = S_i(w)$ for all $i>2$. Since $S(w)\subset S(w')$ by assumption, it follows that $S(w'')\subset S(w')$ as long as $[t_1,t_1+1]\subset S_i(w')$ for $i=1,2$. We establish this via a parity argument.

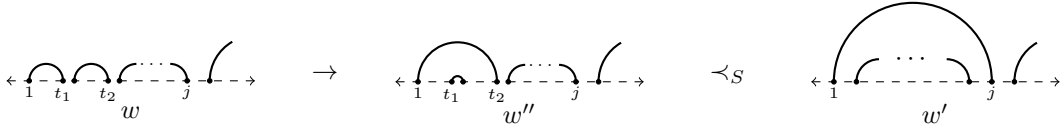
\begin{figure}[h]
\raisebox{2pt}{\begin{tikzpicture}[scale=.3]
\draw[style=dashed, <->] (0,0)--(11,0);
\draw[radius=.1, fill=black](1,0)circle;
\draw[radius=.1, fill=black](2.5,0)circle;
\draw[radius=.1, fill=black](3,0)circle;
\draw[radius=.1, fill=black](4.5,0)circle;
\draw[radius=.1, fill=black](5,0)circle;
\draw[radius=.1, fill=black](8,0)circle;
\draw[radius=.1, fill=black](9,0)circle;
\draw[style=thick] (1,0) arc (180:0: .75cm); 
\draw[style=thick] (3,0) arc (180:0: .75cm);
\draw[style=thick] (5,0) arc (180:90: .75cm); 
\draw[style=thick] (8,0) arc (0:90: .75cm); 
\draw[style=thick] (9,0) arc (180:120: 2cm);
\node at (6.5,.75) {\tiny{$\cdots$}};

\node at (1,-.5) {\tiny{$1$}};
\node at (2.5,-.6) {\tiny{$t_1$}};
\node at (4.5,-.6) {\tiny{$t_2$}};
\node at (8,-.5) {\tiny{$j$}};

\node at (5.5,-1.4) {{$w$}};
\end{tikzpicture}}\hspace{.25in} \raisebox{20pt}{$\rightarrow$}\hspace{.25in}
\begin{tikzpicture}[scale=.3]
\draw[style=dashed, <->] (0,0)--(11,0);
\draw[radius=.1, fill=black](1,0)circle;
\draw[radius=.1, fill=black](2.5,0)circle;
\draw[radius=.1, fill=black](3,0)circle;
\draw[radius=.1, fill=black](4.5,0)circle;
\draw[radius=.1, fill=black](5,0)circle;
\draw[radius=.1, fill=black](8,0)circle;
\draw[radius=.1, fill=black](9,0)circle;
\draw[style=thick] (1,0) arc (180:0: 1.75cm); 
\draw[style=thick] (3,0) arc (0:180: .25cm);
\draw[style=thick] (5,0) arc (180:90: .75cm); 
\draw[style=thick] (8,0) arc (0:90: .75cm); 
\draw[style=thick] (9,0) arc (180:120: 2cm);
\node at (6.5,.75) {\tiny{$\cdots$}};

\node at (1,-.5) {\tiny{$1$}};
\node at (2.5,-.6) {\tiny{$t_1$}};
\node at (4.5,-.6) {\tiny{$t_2$}};
\node at (8,-.5) {\tiny{$j$}};

\node at (5.5,-1.4) {{$w''$}};
\end{tikzpicture}
\hspace{.25in} \raisebox{20pt}{$\prec_S$} \hspace{.25in}
\begin{tikzpicture}[scale=.3]
\draw[style=dashed, <->] (0,0)--(11,0);
\draw[radius=.1, fill=black](1,0)circle;
\draw[radius=.1, fill=black](2,0)circle;
\draw[radius=.1, fill=black](8,0)circle;
\draw[radius=.1, fill=black](9,0)circle;
\draw[radius=.1, fill=black](7,0)circle;
\draw[style=thick] (1,0) arc (180:0: 3.5cm); 
\draw[style=thick] (2,0) arc (180:90: 1cm);
\draw[style=thick] (7,0) arc (0:90: 1cm); 
\draw[style=thick] (9,0) arc (180:120: 2cm);
\node at (4.5,1) {$\cdots$};

\node at (1,-.5) {\tiny{$1$}};

\node at (8,-.5) {\tiny{$j$}};

\node at (5.5,-1.4) {{$w'$}};
\end{tikzpicture}
\caption{Webs $w$, $w''$, and $w'$} \label{fig:localchanges}
\end{figure}

The two endpoints of each arc have different parity since the points below the arc are paired by the matching.  As $(1, t_1), (t_1+1, t_2), (t_2+1, t_3), \ldots (t_k+1, t_{k+2}) = (t_k+1,j)$ are arcs in $w$, it follows that $t_1, t_2, t_3, \ldots, t_{k+2}=j$ are all even.  
Moreover the right endpoint of every arc with even nesting number must be even while the left endpoint of every arc with even nesting number must be odd. Since $1<t_1<j$, $t_1$ lies below the arc $(1,j)$ of nesting number 0 in $w'$. Hence, the depth of $[t_1,t_1+1]$ cannot be 0 in $w'$. Because $t_1$ is even, it also cannot be the right endpoint of an arc of nesting number 1. This means the depth of $[t_1,t_1+1]$ is at least 2 in $w'$, so $[t_1,t_1+1]\subset S_i(w')$ for $i=1,2$.  

Therefore, $S(w'')\subset S(w')$, and by definition $w'' \prec_S w'$. Moreover note that $w''$ has one fewer arc of nesting number $0$ with endpoints between $1$ and $j$ than does $w$. By induction on $k$ we have $w''\prec_T w'$. Since $w\prec_T w''$ and $w''\prec_T w'$ we conclude $w\prec_T w'$ as desired.
\end{proof}

Putting the two previous lemmas together, we have proven that $\prec_S$ and $\prec_T$ are the same for $\mathfrak{sl}_2$ webs.

\begin{theorem}\label{thm:samesl2}
Let $w,w'\in \mathcal{W}_{2n}^R$. Then $w\prec_S w'$ if and only if $w\prec_T w'$. In other words, the partial orders $\prec_T$ and $\prec_S$ coincide on the set $\mathcal{W}_{2n}^R$.
\end{theorem}

\subsection{$\mathfrak{sl}_3$ web shadows}

In this section, we show that the shadow containment partial order is not the same as the tableaux partial order for reduced $\mathfrak{sl}_3$ webs. In fact, the partial order $\prec_S$ is not even ranked in the $\mathfrak{sl}_3$ web case. 

We accomplish this in three steps. First we give a lemma that characterizes all options for the boundary word and depths along the boundary when there is an edge $w \stackrel{s_i}{\rightarrow} w'$ in the Hasse diagram for $\prec_T$.  Then we prove that the covering relation for $\prec_S$  contains the covering relation for $\prec_T$. This implies that if $\prec_S$ were ranked, then it would have the same rank function as $\prec_T$. We conclude with an example of reduced $\mathfrak{sl}_3$ webs $w$ and $w'$ with $w\prec_S w'$ and $r(w)>r(w')$.

\begin{lemma}\label{lem:edge to bdry word}
Let $w,w'\in \mathcal{W}_{3n}^R$ be webs with an edge $w\stackrel{s_i}{\rightarrow} w'$ in the Hasse diagram for $\prec_T$. Then the boundary words for $w$ and $w'$ agree except in positions $i,i+1,$ where exactly one of the following cases applies.  Furthermore the depths along the boundary for $w$ and $w'$ agree except in the interval $(i,i+1)$, where the depth of $w'$ is one greater than the depth of $w$ in Cases 1 and 2, and two greater than the depth of $w$ in Case 3.
\begin{figure}[h]
\begin{tabular}{|c|c|c||c|}
\hline
&   {\bf Boundary word for $w$}   & {\bf Boundary word for $w'$} & {\bf Change in depth in $(i, i+1)$}  \\  \hline
{\bf Case 1}: & $0+$ & $+0$ & $1$  \\ \hline
{\bf Case 2}: & $-0$ & $0-$ & $1$  \\ \hline
{\bf Case 3}: & $-+$ & $+-$ & $2$  \\ \hline
\end{tabular}
\caption{ $i^{th}$ and $(i+1)^{st}$ boundary symbols along edge  $w\stackrel{s_i}{\longrightarrow} w'$}\label{fig:yword}
\end{figure}
\end{lemma}

\begin{proof}
Let $w,w'\in \mathcal{W}_{3n}^R$ such that $w \stackrel{s_i}{\rightarrow} w'$ is an edge in the Hasse diagram for $\prec_T$. It follows from the definition that $i, i+1$ are in different rows of the tableaux for $w, w'$ and that $i$ is in a higher row in the tableau for $w'$.  Moreover, the only places where the boundary words for $w$ and $w'$ differ are in positions $i$ and $i+1$, using the map from boundary words to standard tableaux.

 According to the bijection between reduced webs and standard tableaux described in Subsection \ref{sub:tabweb}, the boundary symbols for $w$ and $w'$ in positions $i$ and $i+1$ must be among the cases listed in the table in Figure \ref{fig:yword}.  The boundary word also encodes depth of the faces along the boundary.  In the first two cases of Figure \ref{fig:yword}, we see that if the depth at $i+\frac{1}{2}$ in $w'$ is $d$ then it is $d-1$ in $w$. In the final case, if the depth at $i+\frac{1}{2}$ in $w'$ is $d$ then it is $d-2$ in $w$. The depths along the boundary at every other point in $[1,i) \cup (i+1,3n]$ are the same in $w, w'$.  This proves the claim.
\end{proof}

The following corollary uses the depth constraints from the previous lemma to show that no web can have a shadow lying strictly between that of two webs $w,w'$ with $w \stackrel{s_i}{\longrightarrow} w'$.

\begin{corollary} \label{cor:TvsS}
Let $w,w'\in \mathcal{W}_{3n}^R$ be webs with an edge $w\stackrel{s_i}{\longrightarrow} w'$ in the Hasse diagram for $\prec_T$. Then $w\prec_S w'$ and there is no web $\overline{w}$ such that $w\prec_S \overline{w} \prec_S w'$. 
\end{corollary}

\begin{proof}
Suppose $\overline{w}$ is a reduced $\mathfrak{sl}_3$ web such that $w\preceq_{S}\overline{w}\preceq_Sw'$.  The depth at each point along the boundary of $\overline{w}$ is bounded below by the corresponding depth in $w$ and above by the depth in $w'$ by Lemma \ref{lem:axisdepth}.  

Now apply Lemma \ref{lem:edge to bdry word}. Depths are integral, so in the first two cases, the depth in $\overline{w}$ between $i$ and $i+1$ must match that of $w$ or $w'$.  In the third case, the depth between boundary vertices $i$ and $i+1$ in $\overline{w}$ could be one greater than in $w$ and one less than in $w'$. However, this would mean the boundary word for $\overline{w}$ would be identical to that of $w$ and $w'$ in all positions except positions $i$ and $i+1$ where it would have symbols $00$. This would mean $\overline{w}$ is a reduced web with an unbalanced boundary word, which is not possible by Proposition \ref{prop:sl3webword}. 

Hence in each of the three cases, if $w\preceq_{S}\overline{w}\preceq_Sw'$ then $w = \overline{w}$ or $\overline{w}=w'$. 
\end{proof}

The next result establishes that $\prec_S$ is a refinement of $\prec_T$ and is a direct consequence of the corollary above.

\begin{corollary}\label{cor:total}
The Hasse diagram for $\prec_T$ is a subgraph of the Hasse diagram for $\prec_S$. Therefore any total order that completes the partial order $\prec_S$ also completes  $\prec_T$.
\end{corollary}

Figure \ref{fig:sl3Counter} exhibits two tableaux $T, \widetilde{T}\in SYT(6,6,6)$ and their corresponding band diagrams on 21 boundary points. The associated webs $\widetilde{w}$ and $w$ are related by $\widetilde{w}\prec_S w$. However $r(\widetilde{w}) > r(w)$ which implies by Lemma \ref{lem:sl3rank} that $\widetilde{w}\not\prec_T w$. This leads us to conclude the following.

\begin{theorem}\label{thm:notsamesl3}
The partial orders $\prec_S$ and $\prec_T$ on reduced $\mathfrak{sl}_3$ webs are different. In fact $\prec_S$ is not ranked.
\end{theorem}

\begin{figure}[h]
\scalebox{.85}{\begin{tabular}{|l|l|}
\hline
& \\ 
$\widetilde{T} = \raisebox{-13pt}{\young(12345\fourteen,6789\ten\fifteen,\eleven\twelve\thirteen\sixteen\seventeen\eighteen)}$&$T=\raisebox{-13pt}{\young(123459,678\thirteen\fourteen\fifteen,\ten\eleven\twelve\sixteen\seventeen\eighteen)}$ \\
& \\ \hline
& \\

 $\widetilde{B} = \raisebox{-10pt}{\begin{tikzpicture}[scale=.25]
\draw[style=thick] (0,0)--(19,0);
\draw[black, fill](1,0)circle(3pt);
\draw[black, fill](2,0)circle(3pt);
\draw[black, fill](3,0)circle(3pt);
\draw[black, fill](4,0)circle(3pt);
\draw[black, fill](5,0)circle(3pt);
\draw[black, fill](6,0)circle(5pt);
\draw[black, fill](7,0)circle(5pt);
\draw[black, fill](8,0)circle(5pt);
\draw[black, fill](9,0)circle(5pt);
\draw[black, fill](10,0)circle(5pt);
\draw[black, fill](11,0)circle(3pt);
\draw[black, fill](12,0)circle(3pt);
\draw[black, fill](13,0)circle(3pt);
\draw[black, fill](14,0)circle(3pt);
\draw[black, fill](15,0)circle(5pt);
\draw[black, fill](16,0)circle(3pt);
\draw[black, fill](17,0)circle(3pt);
\draw[black, fill](18,0)circle(3pt);

\draw[style=thick] (1,0) to[out=90,in=180] (9.5,4.5) to[out=0,in=90] (18,0);
\draw[style=thick] (2,0) to[out=90,in=180] (9.5,4) to[out=0,in=90] (17,0);
\draw[style=thick] (3,0) to[out=90,in=180] (8,3) to[out=0,in=90] (13,0);
\draw[style=thick] (4,0) to[out=90,in=180] (8,2.5) to[out=0,in=90] (12,0);
\draw[style=thick] (5,0) to[out=90,in=180] (8,2) to[out=0,in=90] (11,0);
\draw[style=thick] (14,0) to[out=90,in=180] (15,1) to[out=0,in=90] (16,0);
\end{tikzpicture}}$  & $B = \raisebox{-10pt}{\begin{tikzpicture}[scale=.25]
\draw[style=thick] (0,0)--(19,0);
\draw[black, fill](1,0)circle(3pt);
\draw[black, fill](2,0)circle(3pt);
\draw[black, fill](3,0)circle(3pt);
\draw[black, fill](4,0)circle(3pt);
\draw[black, fill](5,0)circle(3pt);
\draw[black, fill](6,0)circle(5pt);
\draw[black, fill](7,0)circle(5pt);
\draw[black, fill](8,0)circle(5pt);
\draw[black, fill](9,0)circle(3pt);
\draw[black, fill](10,0)circle(3pt);
\draw[black, fill](11,0)circle(3pt);
\draw[black, fill](12,0)circle(3pt);
\draw[black, fill](13,0)circle(5pt);
\draw[black, fill](14,0)circle(5pt);
\draw[black, fill](15,0)circle(5pt);
\draw[black, fill](16,0)circle(3pt);
\draw[black, fill](17,0)circle(3pt);
\draw[black, fill](18,0)circle(3pt);

\draw[style=thick] (1,0) to[out=90,in=180] (9.5,4.5) to[out=0,in=90] (18,0);
\draw[style=thick] (2,0) to[out=90,in=180] (9.5,4) to[out=0,in=90] (17,0);
\draw[style=thick] (3,0) to[out=90,in=180] (9.5,3.5) to[out=0,in=90] (16,0);
\draw[style=thick] (4,0) to[out=90,in=180] (8,2.5) to[out=0,in=90] (12,0);
\draw[style=thick] (5,0) to[out=90,in=180] (8,2) to[out=0,in=90] (11,0);
\draw[style=thick] (9,0) to[out=90,in=180] (9.5,.5) to[out=0,in=90] (10,0);
\end{tikzpicture}}$ \\

& \\ \hline
& \\
$S_1(\widetilde{w}) = [1,21]$ & $S_1(w) = [1,21]$ \\
$S_2(\widetilde{w}) = [2,20]$ & $S_2(w) = [2,20]$ \\
$S_3(\widetilde{w}) = ([3,13]\cup [14,19])$ & $S_3(w) = [3,19]$ \\
$S_4(\widetilde{w}) =([4,12]\cup[15,18])$ & $S_4(w) = [4,12]\cup [15,18]$ \\
$S_5(\widetilde{w}) = [5,11]$ & $S_5(w)=[5,11]$ \\
$S_6(\widetilde{w}) = \emptyset$ & $S_6(w) = [9,10]$ \\
& \\ \hline
t
& \\

$r(\widetilde{w}) = 41$ & $r(w) = 40$ \\
& \\ \hline 
\end{tabular}}
\caption{Data for webs $\widetilde{w}$ and $w$ with $\widetilde{w}\prec_S w$ and $r(\widetilde{w})>r(w)$}\label{fig:sl3Counter}
\end{figure}

\section{The transition matrix from Specht basis to $\mathfrak{sl}_3$ web basis is unitriangular with respect to the shadow partial order} \label{sec:trans}

When $k\geq 2$ the space of $\mathfrak{sl}_k$ webs is equipped with an action of the symmetric group. Previous work by the authors and others focuses primarily on studying properties of the $\mathfrak{sl}_2$ instance of this action along with some preliminary investigation into the $\mathfrak{sl}_3$ case \cite{housley2015robinson, PPR, MR3591372, MR3998730,   MR2801314, MR3920353} . One motivation for studying band diagrams is as a combinatorial tool for characterizing the behavior of $\mathfrak{sl}_3$ webs under the action of the symmetric group. In particular, given $w\in \mathcal{W}_{3n}^R$ and a simple transposition $s_i\in S_{3n}$ we wish to write $s_i \cdot w$ as a linear combination of reduced webs. 

We begin this section with a description of the symmetric group action on $\mathfrak{sl}_2$ and $\mathfrak{sl}_3$ webs.  We then prove that the transition matrix from Specht basis to web basis is unitriangular with respect to the shadow partial order.

Given a simple transposition $s_i$ and an $\mathfrak{sl}_2$ or $\mathfrak{sl}_3$ bottom web $w$, define $s_i \cdot w$ to be the sum of webs obtained by appending a crossing to the $i$ and $(i+1)^{st}$ endpoints of $w$ and then smoothing according to Figure \ref{fig:smoothings}. 
(This skein-theoretic action comes from considering $s_i \cdot w$ as an element of $V^{\mathcal{W}_{2n}}$ or $V^{\mathcal{W}_{3n}}$ and then expressing it in terms of the reduced web basis by applying web reduction rules.)

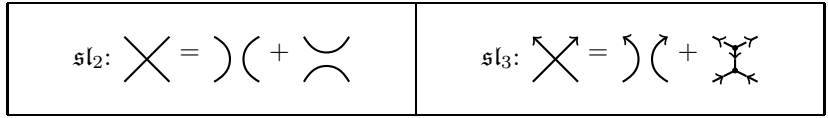
\begin{figure}
\begin{tabular}{|c|c|}
\hline 
\hspace{2in} & \hspace{2in} \\
{\bf $\mathfrak{sl}_2$}:
\raisebox{2pt}{\begin{tikzpicture}[baseline=0cm, scale=0.3]

\draw[style=thick] (-1,-1)--(1,1);
\draw[style=thick] (1,-1)--(-1,1);

\end{tikzpicture}
$= $\,
\begin{tikzpicture}[baseline=0cm, scale=0.3]
 
\draw[style=thick] (-1,-1) to[out=30,in=270] (-.3,0) to[out=90,in=-30](-1,1);
\draw[style=thick] (1,-1) to[out=150,in=270] (.3,0) to[out=90,in=210](1,1);

\end{tikzpicture}
$+$\,
\begin{tikzpicture}[baseline=0cm, scale=0.3]

\draw[style=thick] (-1,-1) to[out=60,in=180] (0,-.3) to[out=0,in=120](1,-1);
\draw[style=thick] (-1,1) to[out=-60,in=180] (0, .3) to[out=0,in=240](1,1);

\end{tikzpicture}}
&
{\bf $\mathfrak{sl}_3$}:
\raisebox{2pt}{\begin{tikzpicture}[baseline=0cm, scale=0.3]

\draw[style=thick,->] (-1,-1)--(1,1);
\draw[style=thick, ->] (1,-1)--(-1,1);

\end{tikzpicture}
$= $\,
\begin{tikzpicture}[baseline=0cm, scale=0.3]

\draw[style=thick, ->] (-1,-1) to[out=30,in=270] (-.3,0) to[out=90,in=-30](-1,1);
\draw[style=thick, ->] (1,-1) to[out=150,in=270] (.3,0) to[out=90,in=210](1,1);

\end{tikzpicture} 
$+$\,
\begin{tikzpicture}[baseline=0cm, scale=0.3]

\draw[style=thick,->] (-1,-1)--(-.5,-.75);
\draw[style=thick] (-.5,-.75)--(0,-.5);
\draw[style=thick,->] (1,-1)--(.5,-.75);
\draw[style=thick] (.5,-.75)--(0,-.5);
\draw[style=thick,-<] (-1,1)--(-.5,.75);
\draw[style=thick] (-.5,.75)--(0,.5);
\draw[style=thick,-<] (1,1)--(.5,.75);
\draw[style=thick, ->] (0,.5)--(0,0);
\draw[style=thick] (0,0)--(0,-.5);
\draw[style=thick] (.5,.75)--(0,.5);
\draw[fill=black] (0,.5)circle(3pt);
\draw[fill=black] (0,-.5)circle(3pt);
\end{tikzpicture}} \\
& \\ \hline
\end{tabular}
\caption{Smoothing rules for symmetric group action on webs}\label{fig:smoothings}
\end{figure}

Figure \ref{fig:exact} shows examples of the symmetric group acting on both $\mathfrak{sl}_2$ and $\mathfrak{sl}_3$ webs.  In the $\mathfrak{sl}_2$ case the web $s_i\cdot w$ has at most two reduced summands. If $i$ and $i+1$ are connected in $w$ then $s_i\cdot w = -w$. If $i$ and $i+1$ are not connected in $w$ then $w$ has arcs $(i,j)$ and $(i+1,k)$ in some configuration, either nested or unnested. In this case $s_i \cdot w = w + w'$ where $w'$ is identical to $w$ off of the vertices $i, i+1, j$, and $k$ and $w'$ has arcs $(j,k)$ and $(i,i+1)$. This action on $\mathfrak{sl}_2$ webs leads to another description of the edges in the Hasse diagram for $\prec_T$. In particular, $w\stackrel{s_i}{\longrightarrow} w'$ is an edge if and only if $i$ is below $i+1$ in $w$ and $s_i\cdot w = w+w'$. 

In the $\mathfrak{sl}_3$ case, the action of the symmetric group is not easily characterized. For instance, given a simple transposition $s_i$ and a reduced $\mathfrak{sl}_3$ web $w$, the upper bound  on the number of reduced summands in the web $s_i\cdot w$ is a {\em nonconstant} function of the number of boundary vertices. 

\begin{figure}[h]
\begin{eqnarray*}
s_2 \cdot \raisebox{-2pt}{\begin{tikzpicture}[ scale=0.25]
\draw[style=dashed, <->] (-3.5,1)--(3.5,1);
\draw[style=thick] (-3,1) to[out=90, in=180] (0,4) to[out=0, in=90] (3,1);
\draw[style=thick] (-2,1) to[out=90, in=180] (0,3) to[out=0, in=90] (2,1);
\draw[style=thick] (-1,1) to[out=90, in=180] (0,2) to[out=0, in=90] (1,1);
\end{tikzpicture}} \hspace{.25in} &=& 
\raisebox{-4pt}{\begin{tikzpicture}[ scale=0.25]
\draw[style=dashed, <->] (-3.5,0)--(3.5,0);
\draw[style=thick] (-3,1) to[out=90, in=180] (0,4) to[out=0, in=90] (3,1);
\draw[style=thick] (-2,1) to[out=90, in=180] (0,3) to[out=0, in=90] (2,1);
\draw[style=thick] (-1,1) to[out=90, in=180] (0,2) to[out=0, in=90] (1,1);
\draw[style=thick] (-3,1)--(-3,0);
\draw[style=thick] (-2,1)--(-1,0);
\draw[style=thick](-1,1)--(-2,0);
\draw[style=thick] (3,1)--(3,0);
\draw[style=thick] (2,1)--(2,0);
\draw[style=thick] (1,1)--(1,0);
\end{tikzpicture}}\\
&=&
\raisebox{-4pt}{\begin{tikzpicture}[ scale=0.25]
\draw[style=dashed, <->] (-3.5,1)--(3.5,1);
\draw[style=thick] (-3,1) to[out=90, in=180] (0,4) to[out=0, in=90] (3,1);
\draw[style=thick] (-2,1) to[out=90, in=180] (0,3) to[out=0, in=90] (2,1);
\draw[style=thick] (-1,1) to[out=90, in=180] (0,2) to[out=0, in=90] (1,1);
\end{tikzpicture}} +
\raisebox{-4pt}{\begin{tikzpicture}[ scale=0.25]
\draw[style=dashed, <->] (-3.5,1)--(3.5,1);
\draw[style=thick] (-3,1) to[out=90, in=180] (0,4) to[out=0, in=90] (3,1);
\draw[style=thick] (-2,1) to[out=90, in=180] (-1.125,1.875) to[out=0, in=90] (-.25,1);
\draw[style=thick] (.25,1) to[out=90, in=180] (1.125,1.875) to[out=0, in=90] (2,1);
\end{tikzpicture}}  \\
s_2 \cdot \begin{tikzpicture}[baseline=0cm, scale=0.4]
\draw[style=dashed, <->] (2.5,0)--(8.5,0);
\draw[radius=.08, fill=black](3,0)circle;
\draw[radius=.08, fill=black](4,0)circle;
\draw[radius=.08, fill=black](5,0)circle;
\draw[radius=.08, fill=black](6,0)circle;
\draw[radius=.08, fill=black](7,0)circle;
\draw[radius=.08, fill=black](8,0)circle;
\draw[style=thick,->](4, 0) -- (4,.5);
\draw[style=thick](4,.5)--(4,1);
\draw[radius=.08, fill=black](4,1)circle;
\draw[style=thick,->](3,0)--(3.5,.5);
\draw[style=thick](3.5,.5)--(4,1);
\draw[style=thick,->](7, 0) -- (7,.5);
\draw[style=thick](7,.5)--(7,1);
\draw[radius=.08, fill=black](7,1)circle;
\draw[style=thick,->](8,0)--(7.5,.5);
\draw[style=thick](7.5,.5)--(7,1);
\draw[radius=.08, fill=black](5.5,1)circle;
\draw[style=thick,-<](5.5,1)--(5.5,1.5);
\draw[style=thick](5.5,1.5)--(5.5,2);
\draw[radius=.08, fill=black](5.5,2)circle;
\draw[style=thick,->](5,0)--(5.25,.5);
\draw[style=thick](5.25,.5)--(5.5,1);
\draw[style=thick,->](6,0)--(5.75,.5);
\draw[style=thick](5.75,.5)--(5.5,1);
\draw[style=thick,-<](4,1)--(4.75,1.5);
\draw[style=thick](4.75,1.5)--(5.5,2);
\draw[style=thick,-<](7,1)--(6.25,1.5);
\draw[style=thick](6.25,1.5)--(5.5,2);
\end{tikzpicture}
&=&
\begin{tikzpicture}[baseline=0cm, scale=0.4]
\draw[style=dashed, <->] (2.5,0)--(8.5,0);
\draw[radius=.08, fill=black](3,0)circle;
\draw[radius=.08, fill=black](4,0)circle;
\draw[radius=.08, fill=black](5,0)circle;
\draw[radius=.08, fill=black](6,0)circle;
\draw[radius=.08, fill=black](7,0)circle;
\draw[radius=.08, fill=black](8,0)circle;
\draw[style=thick,->](4, 0) -- (5.125,.75);
\draw[style=thick](5.125,.75)--(5.5,1);
\draw[radius=.08, fill=black](4,1)circle;
\draw[style=thick,->](3,0)--(3.5,.5);
\draw[style=thick](3.5,.5)--(4,1);
\draw[style=thick,->](7, 0) -- (7,.5);
\draw[style=thick](7,.5)--(7,1);
\draw[radius=.08, fill=black](7,1)circle;
\draw[style=thick,->](8,0)--(7.5,.5);
\draw[style=thick](7.5,.5)--(7,1);
\draw[radius=.08, fill=black](5.5,1)circle;
\draw[style=thick,-<](5.5,1)--(5.5,1.5);
\draw[style=thick](5.5,1.5)--(5.5,2);
\draw[radius=.08, fill=black](5.5,2)circle;
\draw[style=thick,->](5,0)--(4.75,.25);
\draw[style=thick](4.75,.25)--(4,1);
\draw[style=thick,->](6,0)--(5.75,.5);
\draw[style=thick](5.75,.5)--(5.5,1);
\draw[style=thick,-<](4,1)--(4.75,1.5);
\draw[style=thick](4.75,1.5)--(5.5,2);
\draw[style=thick,-<](7,1)--(6.25,1.5);
\draw[style=thick](6.25,1.5)--(5.5,2);
\end{tikzpicture} \\
&=&
\begin{tikzpicture}[baseline=0cm, scale=0.4]
\draw[style=dashed, <->] (2.5,0)--(8.5,0);
\draw[radius=.08, fill=black](3,0)circle;
\draw[radius=.08, fill=black](4,0)circle;
\draw[radius=.08, fill=black](5,0)circle;
\draw[radius=.08, fill=black](6,0)circle;
\draw[radius=.08, fill=black](7,0)circle;
\draw[radius=.08, fill=black](8,0)circle;
\draw[style=thick,->](4, 0) -- (4,.5);
\draw[style=thick](4,.5)--(4,1);
\draw[radius=.08, fill=black](4,1)circle;
\draw[style=thick,->](3,0)--(3.5,.5);
\draw[style=thick](3.5,.5)--(4,1);
\draw[style=thick,->](7, 0) -- (7,.5);
\draw[style=thick](7,.5)--(7,1);
\draw[radius=.08, fill=black](7,1)circle;
\draw[style=thick,->](8,0)--(7.5,.5);
\draw[style=thick](7.5,.5)--(7,1);
\draw[radius=.08, fill=black](5.5,1)circle;
\draw[style=thick,-<](5.5,1)--(5.5,1.5);
\draw[style=thick](5.5,1.5)--(5.5,2);
\draw[radius=.08, fill=black](5.5,2)circle;
\draw[style=thick,->](5,0)--(5.25,.5);
\draw[style=thick](5.25,.5)--(5.5,1);
\draw[style=thick,->](6,0)--(5.75,.5);
\draw[style=thick](5.75,.5)--(5.5,1);
\draw[style=thick,-<](4,1)--(4.75,1.5);
\draw[style=thick](4.75,1.5)--(5.5,2);
\draw[style=thick,-<](7,1)--(6.25,1.5);
\draw[style=thick](6.25,1.5)--(5.5,2);
\end{tikzpicture}
+
\begin{tikzpicture}[baseline=0cm, scale=0.4]
\draw[style=dashed, <->] (2.5,0)--(8.5,0);
\draw[radius=.08, fill=black](3,0)circle;
\draw[radius=.08, fill=black](4,0)circle;
\draw[radius=.08, fill=black](5,0)circle;
\draw[radius=.08, fill=black](6,0)circle;
\draw[radius=.08, fill=black](7,0)circle;
\draw[radius=.08, fill=black](8,0)circle;
\draw[radius=.08, fill=black](4,1)circle;
\draw[radius=.08, fill=black](4.5,.25)circle;
\draw[radius=.08, fill=black](4.5,.5)circle;
\draw[style=thick,->](3,0)--(3.5,.5);
\draw[style=thick](3.5,.5)--(4,1);
\draw[style=thick,->](7, 0) -- (7,.5);
\draw[style=thick](7,.5)--(7,1);
\draw[radius=.08, fill=black](7,1)circle;
\draw[style=thick,->](8,0)--(7.5,.5);
\draw[style=thick](7.5,.5)--(7,1);
\draw[radius=.08, fill=black](5.5,1)circle;
\draw[style=thick,-<](5.5,1)--(5.5,1.5);
\draw[style=thick](5.5,1.5)--(5.5,2);
\draw[radius=.08, fill=black](5.5,2)circle;

\draw[style=thick,->](6,0)--(5.75,.5);
\draw[style=thick](5.75,.5)--(5.5,1);
\draw[style=thick,-<](4,1)--(4.75,1.5);
\draw[style=thick](4.75,1.5)--(5.5,2);
\draw[style=thick,-<](7,1)--(6.25,1.5);
\draw[style=thick](6.25,1.5)--(5.5,2);
 \begin{scope}[thick,decoration={
                                             markings,
                                             mark=at position 0.5 with {\arrow{>}}}]
                              \draw[postaction={decorate}] (4,0)--(4.5,.25);
                              \draw[postaction={decorate}] (5,0)--(4.5,.25);
                              \draw[postaction={decorate}] (4.5,.5)--(4.5,.25);
                              \draw[postaction={decorate}] (4.5,.5)--(4,1);
                               \draw[postaction={decorate}] (4.5,.5)--(5.5,1);
                              \end{scope}
\end{tikzpicture} \\
&=& 
\begin{tikzpicture}[baseline=0cm, scale=0.4]
\draw[style=dashed, <->] (2.5,0)--(8.5,0);
\draw[radius=.08, fill=black](3,0)circle;
\draw[radius=.08, fill=black](4,0)circle;
\draw[radius=.08, fill=black](5,0)circle;
\draw[radius=.08, fill=black](6,0)circle;
\draw[radius=.08, fill=black](7,0)circle;
\draw[radius=.08, fill=black](8,0)circle;
\draw[style=thick,->](4, 0) -- (4,.5);
\draw[style=thick](4,.5)--(4,1);
\draw[radius=.08, fill=black](4,1)circle;
\draw[style=thick,->](3,0)--(3.5,.5);
\draw[style=thick](3.5,.5)--(4,1);
\draw[style=thick,->](7, 0) -- (7,.5);
\draw[style=thick](7,.5)--(7,1);
\draw[radius=.08, fill=black](7,1)circle;
\draw[style=thick,->](8,0)--(7.5,.5);
\draw[style=thick](7.5,.5)--(7,1);
\draw[radius=.08, fill=black](5.5,1)circle;
\draw[style=thick,-<](5.5,1)--(5.5,1.5);
\draw[style=thick](5.5,1.5)--(5.5,2);
\draw[radius=.08, fill=black](5.5,2)circle;
\draw[style=thick,->](5,0)--(5.25,.5);
\draw[style=thick](5.25,.5)--(5.5,1);
\draw[style=thick,->](6,0)--(5.75,.5);
\draw[style=thick](5.75,.5)--(5.5,1);
\draw[style=thick,-<](4,1)--(4.75,1.5);
\draw[style=thick](4.75,1.5)--(5.5,2);
\draw[style=thick,-<](7,1)--(6.25,1.5);
\draw[style=thick](6.25,1.5)--(5.5,2);
\end{tikzpicture}
+
\begin{tikzpicture}[scale=.4]
                              \draw[style=dashed, <->] (.5,0)--(6.5,0);
                              \draw[radius=.08, fill=black](1,0)circle;
                              \draw[radius=.08, fill=black](2,0)circle;
                             
                              \draw[radius=.08, fill=black](3,0)circle;
                              \draw[radius=.08, fill=black](4,0)circle;
                              \draw[radius=.08, fill=black](5,0)circle;
                             
                              \draw[radius=.08, fill=black](6,0)circle;
                              \draw[radius=.08, fill=black](2,1)circle;
                              \draw[radius=.08, fill=black](5,1)circle;
                             
                              \begin{scope}[thick,decoration={
                                             markings,
                                             mark=at position 0.5 with {\arrow{>}}}
                              ]
                              \draw[postaction={decorate}] (1,0)--(2,1);
                              \draw[postaction={decorate}] (3,0)--(2,1);
                              \draw[postaction={decorate}] (2,0)--(2,1);
                              \draw[postaction={decorate}] (4,0)--(5,1);
                              \draw[postaction={decorate}] (5,0)--(5,1);
                              \draw[postaction={decorate}] (6,0)--(5,1);
                              \end{scope}
                              \end{tikzpicture}
                              + \begin{tikzpicture}[scale=.4]
                              \draw[style=dashed, <->] (.5,0)--(6.5,0);
                              \draw[radius=.08, fill=black](1,0)circle;
                              \draw[radius=.08, fill=black](2,0)circle;
                             
                              \draw[radius=.08, fill=black](3,0)circle;
                              \draw[radius=.08, fill=black](4,0)circle;
                              \draw[radius=.08, fill=black](5,0)circle;
                             
                              \draw[radius=.08, fill=black](6,0)circle;
                              \draw[radius=.08, fill=black](3,.75)circle;
                              \draw[radius=.08, fill=black](4,1.5)circle;
                             
                              \begin{scope}[thick,decoration={
                                             markings,
                                             mark=at position 0.5 with {\arrow{>}}}
                              ]
                              \draw[postaction={decorate}] (1,0)--(4,1.5);
                              \draw[postaction={decorate}] (3,0)--(3,.75);
                              \draw[postaction={decorate}] (2,0)--(3,.75);
                              \draw[postaction={decorate}] (4,0)--(3,.75);
                              \draw[postaction={decorate}] (5,0)--(4,1.5);
                              \draw[postaction={decorate}] (6,0)--(4,1.5);
                              \end{scope}
                              \end{tikzpicture}
\end{eqnarray*}
\caption{An example of the symmetric group action on $\mathfrak{sl}_2$ and $\mathfrak{sl}_3$ webs.}\label{fig:exact}
\end{figure}
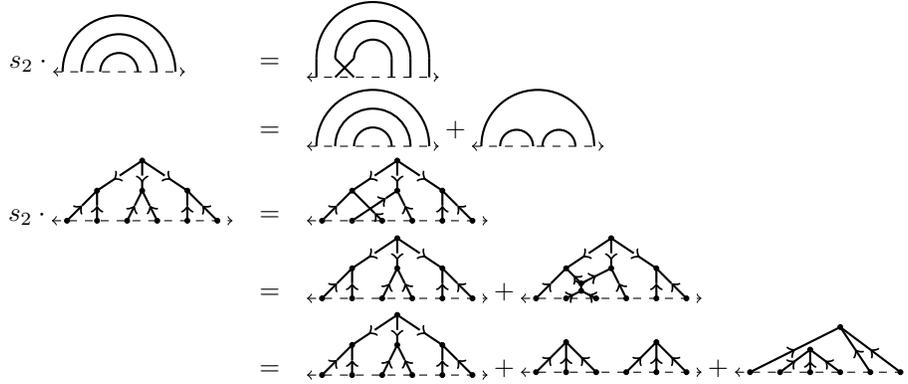

Recall that the complex, finite-dimensional representations of the symmetric group are parameterized by partitions. The following (classical) result characterizes web space as a symmetric group representation \cite{PPR, MR2801314}. In fact, Schur-Weyl duality can be used to obtain a generalized version of this result for $\mathfrak{sl}_k$ webs when $k\geq 2$ \cite{FKK,MR1321638, MR841713, MR0000255}.

\begin{theorem}
The symmetric group representations $V^{\mathcal{W}_{2n}}$ and $V^{\mathcal{W}_{3n}}$ described above are irreducible and correspond to the partitions $(n,n)$ and $(n,n,n)$ respectively.
\end{theorem}

Specht modules are a classic combinatorial construction of irreducible symmetric group representations \cite{MR1545531}. The $(n,n)$ (respectively $(n,n,n)$) Specht module, which we will denote by $V^{\mathcal{T}_{2n}}$ (respectively $V^{\mathcal{T}_{3n}}$), has a basis indexed by standard Young tableaux of shape $(n,n)$ (respectively of shape $(n,n,n)$) \cite{MR0439548}. Since the reduced web bases are also indexed by standard Young tableaux, it is natural to compare the Specht and web bases.   In particular, we wish to describe the unique (up to scaling), nontrivial $S_{2n}$- or $S_{3n}$-equivariant isomorphism $\phi$ between them. 

Given a Young tableau $T$, we write its corresponding Specht vector as $v_T$. We will not give a full description of the action of the symmetric group on the Specht basis. (An interested reader can find such an exposition in \cite{MR1464693}.) However, the following property of the Specht module is important in our discussion below. 

\begin{proposition}\label{prop:tabact}
Let $T\stackrel{s_i}{\longrightarrow} T'$ be an edge in the Hasse diagram for $\prec_T$. Then $s_i\cdot v_T = v_{T'}$. 
\end{proposition}

\begin{remark}
Note that  the combinatorial bijection $\psi^{-1}$ between tableaux and webs fails to extend to a map of representations for either $\mathfrak{sl}_2$ or $\mathfrak{sl}_3$. As a concrete example, consider the initial $\mathfrak{sl}_3$ web to which $s_2$ is applied in Figure \ref{fig:exact}. Call this web $w$. Observe that $\psi^{-1}\left(\raisebox{-8pt}{\tiny{\young(13,25,46)}}\right)=w$ and $s_2\cdot T = T' = \raisebox{-8pt}{\tiny{\young(12,35,46)}}$. Now let $w' = \psi^{-1}(T')$. The calculation in Figure \ref{fig:exact} shows that $s_2\cdot w \neq w'$. In other words $\psi^{-1}$ is not $S_{6}$-equivariant since $s_2 \cdot \psi^{-1}(v_T) \neq \psi^{-1}(s_2\cdot v_T)$.
\end{remark}

The following result of the authors describes the structure of the transition matrix between the Specht and web bases in the $\mathfrak{sl}_2$ \cite{MR3920353}.  

\begin{theorem}\label{thm:sl2triang}
Let $<$ be any total order on $\mathcal{W}_{2n}^R$ that completes $\prec_T$, and let $\phi: V^{\mathcal{T}_{2n}} \rightarrow V_{q=1}^{\mathcal{W}_{2n}}$ be the unique (up to scaling) $S_{2n}$-equivariant isomorphism between the Specht and web modules. Using $<$ to order both the Specht and web bases, the matrix for $\phi$ is unitriangular.
\end{theorem}

We emphasize here that unitriangularity with respect to {\em any} partial order completing $\prec_T$ identifies additional vanishing entries above the diagonal in the matrix for $\phi$. Other recent results show this matrix has all nonnegative entries,  and the only vanishing entries are those identified by Theorem \ref{thm:sl2triang} \cite{ImZhu, MR3998730}. 

The proof of Theorem \ref{thm:sl2triang} follows from an inductive argument and several lemmas characterizing the rank of reduced summands of $s_i\cdot w$. We would like to prove the same unitriangularity result in the $\mathfrak{sl}_3$ web setting. However, the approach for $\mathfrak{sl}_2$ cannot be immediately extended because the action of simple transpositions on $\mathfrak{sl}_3$ webs is much more complicated. Instead of analyzing the rank of reduced summands of $s_i \cdot w$ in this case, we track shadow containment.

 The remainder of this paper is dedicated to proving that the map $\phi$ from the Specht to the web basis is unitriangular in terms of the shadow containment partial order (see Theorem \ref{thm:coefficients}). By Corollary \ref{cor:total}, this implies there is a way to complete the partial order $\prec_T$ to a total order with respect to which $\phi$ is unitriangular (see Theorem \ref{thm:existtriang}). Lemma \ref{lem:basecase} is the base case for Theorem \ref{thm:coefficients}. Its proof follows from an argument completely analogous to the $\mathfrak{sl}_2$ case \cite[Theorem 5.1]{MR3920353}; it is also proven in recent work by Im and Zhu \cite[Theorem 3.9]{ImZhu}. 
\begin{lemma}\label{lem:basecase}
Let $T_0\in SYT(n,n,n)$ be the column-filled tableau and let $w_0$ be the corresponding reduced web on $3n$ vertices. Then $\phi\left(v_{T_0}\right) = aw_0$ for some $a\neq 1$.
\end{lemma}

For simplicity, assume by rescaling if necessary that $a=1$ so $\phi\left(v_{T_0}\right) = w_0$. The next three lemmas characterize shadows of reduced summands of $s_i \cdot w$ for all nine possible pairs of boundary symbols in the $i$ and $(i+1)^{st}$ position of the boundary word for $w$. Proofs for these lemmas follow from straightforward case analyses. 
 We give example calculations for at least one case of each lemma in Figure \ref{fig:twistcases}; the other cases are completely analogous. See \cite{housley2015robinson} for further examples of similar calculations.

\begin{lemma}\label{lem:edgeshadow0}
If $w$ has $+-$, $+0$, or $0-$ in the $i$ and $(i+1)^{st}$ positions of its boundary word, then $s_i \cdot w = -w$.
\end{lemma}

\begin{lemma}\label{lem:edgeshadow1}
If $w$ has $++$, $00$, or $--$ in the $i$ and $(i+1)^{st}$ positions of its boundary word, then $\widetilde{w}\prec_S w$ for all reduced summands $\widetilde{w}\neq w$ of $s_i \cdot w$.
\end{lemma}

\begin{lemma}\label{lem:edgeshadow2}
If $w$ has $-+$, $-0$, or $0+$ in the $i$ and $(i+1)^{st}$ positions of its boundary word, then there is an edge directed out of $w$ in the Hasse diagram for $\prec_T$.  Suppose this edge is $w\stackrel{s_i}{\longrightarrow} w'$.  Then for each reduced summand $\widetilde{w}$ of $s_i \cdot w$ we have $\widetilde{w}\preceq_S w'$.
\end{lemma}

\begin{figure}[h]
\scalebox{.75}{\begin{tabular}{|c|c|}
\hline
& \\
{\bf Boundary word for $w$} & {\bf Local Calculation of $s_i \cdot w$} \\
{\bf at $i$ and $i+1$} & {\bf at vertices $i$ and $i+1$}  \\
&   \\ \hline 
&   \\ 
& $M$-diagram for $w = \raisebox{-25pt}{\begin{tikzpicture}[scale=.4]
                              \draw[style=dashed, <->] (-1,0)--(4.75,0);
                              \draw[radius=.08, fill=black](.75,0)circle;
                              \draw[radius=.08, fill=black](2.75,0)circle;
                               \node at (.75,-.5){$+$};
                             \node at (2.75,-.5){$-$};
                                                           \begin{scope}[thick,decoration={
                                             markings,
                                             mark=at position 0.75 with {\arrow{>}}}
                              ]
                              \draw[postaction={decorate}] (.75,0)to[out=90,in=180](4,3);
                              \draw[postaction={decorate}] (2.75,0)to[out=90,in=0](-.5,3);
                              \end{scope}
                              \end{tikzpicture}}$  \\ 
\begin{huge}$+-$\end{huge} 
                              & $w = 
                              \raisebox{-15pt}{\begin{tikzpicture}[scale=.4]
                              \draw[style=dashed, <->] (-1,0)--(4.75,0);
                              \draw[radius=.08, fill=black](.75,0)circle;
                              \draw[radius=.08, fill=black](2.75,0)circle;
                              \draw[radius=.08, fill=black](1.75,1.5)circle;
                              \draw[radius=.08, fill=black](1.75,2)circle;
    \node at (.75,-.5){$+$};
                             \node at (2.75,-.5){$-$};

                              \begin{scope}[thick,decoration={
                                             markings,
                                             mark=at position 0.5 with {\arrow{>}}}
                              ]
                              \draw[postaction={decorate}] (.75,0) to[out=90,in=225](1.75,1.5);
                              \draw[postaction={decorate}] (2.75,0)to[out=90,in=-45](1.75,1.5);
                               \draw[postaction={decorate}] (1.75,2)--(1.75,1.5);
                                \draw[postaction={decorate}] (1.75,2)--(.75,3);
                                 \draw[postaction={decorate}] (1.75,2)--(2.75,3);
                              \end{scope}
                              \end{tikzpicture}}$ \\ 
                              & $s_i \cdot w = w+\widehat{w} =  \raisebox{-15pt}{\begin{tikzpicture}[scale=.4]
                              \draw[style=dashed, <->] (-1,0)--(4.75,0);
                              \draw[radius=.08, fill=black](.75,0)circle;
                              \draw[radius=.08, fill=black](2.75,0)circle;
                              \draw[radius=.08, fill=black](1.75,1.5)circle;
                              \draw[radius=.08, fill=black](1.75,2)circle;

                              \begin{scope}[thick,decoration={
                                             markings,
                                             mark=at position 0.5 with {\arrow{>}}}
                              ]
                              \draw[postaction={decorate}] (.75,0) to[out=90,in=225](1.75,1.5);
                              \draw[postaction={decorate}] (2.75,0)to[out=90,in=-45](1.75,1.5);
                               \draw[postaction={decorate}] (1.75,2)--(1.75,1.5);
                                \draw[postaction={decorate}] (1.75,2)--(.75,3);
                                 \draw[postaction={decorate}] (1.75,2)--(2.75,3);
                              \end{scope}
                              \end{tikzpicture}}
                              +
                               \raisebox{-15pt}{\begin{tikzpicture}[scale=.4]
                              \draw[style=dashed, <->] (-1,0)--(4.75,0);
                              \draw[radius=.08, fill=black](.75,0)circle;
                              \draw[radius=.08, fill=black](2.75,0)circle;
                              \draw[radius=.08, fill=black](1.75,.5)circle;
                              \draw[radius=.08, fill=black](1.75,1)circle;
                              \draw[radius=.08, fill=black](1.75,2)circle;
                              \draw[radius=.08, fill=black](1.75,2.5)circle;

                              \begin{scope}[thick,decoration={
                                             markings,
                                             mark=at position 0.5 with {\arrow{>}}}
                              ]
                              \draw[postaction={decorate}] (.75,0) to[out=90,in=190](1.75,.5);
                              \draw[postaction={decorate}] (2.75,0)to[out=90,in=-10](1.75,.5);
                               \draw[postaction={decorate}] (1.75,1)--(1.75,.5);
                                \draw[postaction={decorate}] (1.75,1) to[out=170,in=270] (1.25,1.5) to[out=90,in=190](1.75,2);
                                 \draw[postaction={decorate}] (1.75,1) to[out=10,in=270] (2.25,1.5) to[out=90,in=-10](1.75,2);
                                 \draw[postaction={decorate}] (1.75,2.5)--(1.75,2);
                                 \draw[postaction={decorate}] (1.75,2.5)--(.75,3);
                                 \draw[postaction={decorate}] (1.75,2.5)--(2.75,3);
                              \end{scope}
                              \end{tikzpicture}}
                              = w-2w$  \\ 
                              & \\
                              & {\bf Conclusion}: In this case $s_i \cdot w = -w$ as desired \\ 
                              & \\ \hline
&   \\ 
& $w = \raisebox{-25pt}{\begin{tikzpicture}[scale=.4]
                              \draw[style=dashed, <->] (-1,0)--(4.75,0);
                              \draw[radius=.08, fill=black](.75,0)circle;
                              \draw[radius=.08, fill=black](2.75,0)circle;
                               \node at (.75,-.5){$+$};
                             \node at (2.75,-.5){$+$};
                             \node at (0,1) {\tiny{$d$}};
                             \node at (1.75, 1) {\tiny{$d+1$}};
                              \node at (3.75, 1) {\tiny{$d+2$}};
                             
                              \begin{scope}[thick,decoration={
                                             markings,
                                             mark=at position 0.5 with {\arrow{>}}}
                              ]
                              \draw[postaction={decorate}] (.75,0)--(.75,3);
                              \draw[postaction={decorate}] (2.75,0)--(2.75,3);
                              \end{scope}
                              \end{tikzpicture}}$ \\
\begin{huge}$++$\end{huge} & $s_i \cdot w = w+\widehat{w} =  \raisebox{-15pt}{\begin{tikzpicture}[scale=.4]
                              \draw[style=dashed, <->] (-1,0)--(4.75,0);
                              \draw[radius=.08, fill=black](.75,0)circle;
                              \draw[radius=.08, fill=black](2.75,0)circle;
                             \node at (0,1) {\tiny{$d$}};
                             \node at (1.75, 1) {\tiny{$d+1$}};
                              \node at (3.75, 1) {\tiny{$d+2$}};
                             
                              \begin{scope}[thick,decoration={
                                             markings,
                                             mark=at position 0.5 with {\arrow{>}}}
                              ]
                              \draw[postaction={decorate}] (.75,0)--(.75,3);
                              \draw[postaction={decorate}] (2.75,0)--(2.75,3);
                              \end{scope}
                              \end{tikzpicture}}
                              +
                              \raisebox{-15pt}{\begin{tikzpicture}[scale=.4]
                              \draw[style=dashed, <->] (-1,0)--(4.75,0);
                              \draw[radius=.08, fill=black](.75,0)circle;
                              \draw[radius=.08, fill=black](2.75,0)circle;
                              \draw[radius=.08, fill=black](1.75,1.5)circle;
                              \draw[radius=.08, fill=black](1.75,2)circle;

                             \node at (0,1) {\tiny{$d$}};
                             \node at (1.75, .3) {\tiny{$d+1$}};
                              \node at (1.75, 3) {\tiny{$d+1$}};
                              \node at (3.75, 1) {\tiny{$d+1$}};
                             
                              \begin{scope}[thick,decoration={
                                             markings,
                                             mark=at position 0.5 with {\arrow{>}}}
                              ]
                              \draw[postaction={decorate}] (.75,0) to[out=90,in=225](1.75,1.5);
                              \draw[postaction={decorate}] (2.75,0)to[out=90,in=-45](1.75,1.5);
                               \draw[postaction={decorate}] (1.75,2)--(1.75,1.5);
                                \draw[postaction={decorate}] (1.75,2)--(.75,3);
                                 \draw[postaction={decorate}] (1.75,2)--(2.75,3);
                              \end{scope}
                              \end{tikzpicture}}
                              $  \\
                              & \\ 
&  {\bf Conclusion}: If $\widetilde{w}$ is a reduced summand of $\widehat{w}$ then $S(\widetilde{w}) \subseteq S(\widehat{w})$ by Lemma \ref{lem:shadowred}. \\
& Depth along the boundary is equivalent to $\prec_S$ by Lemma \ref{lem:depthshad} so we have $\widetilde{w} \prec_S w$ \\ 
& \\ \hline
&   \\ 
& $M$-diagram for $w = \raisebox{-25pt}{\begin{tikzpicture}[scale=.4]
                              \draw[style=dashed, <->] (-1,0)--(4.75,0);
                              \draw[radius=.08, fill=black](.75,0)circle;
                              \draw[radius=.08, fill=black](2.75,0)circle;
                               \node at (.75,-.5){$0$};
                             \node at (2.75,-.5){$0$};
                                                           \begin{scope}[thick,decoration={
                                             markings,
                                             mark=at position 0.75 with {\arrow{<}}}
                              ]
                              \draw[postaction={decorate}] (.75,0)to[out=90,in=180](4,3);
                              \draw[postaction={decorate}] (.75,0)to[out=90,in=0](-.5,2);
                              \draw[postaction={decorate}] (2.75,0)to[out=90,in=0](-.5,3);
                              \draw[postaction={decorate}] (2.75,0)to[out=90,in=180](4,2);
                              \end{scope}
                              \end{tikzpicture}}$  \\ 
\begin{huge}$00$\end{huge} 
                              & $w = 
                              \raisebox{-15pt}{\begin{tikzpicture}[scale=.4]
                              \draw[style=dashed, <->] (-1,0)--(4.75,0);
                              \draw[radius=.08, fill=black](.75,0)circle;
                              \draw[radius=.08, fill=black](2.75,0)circle;
                               \draw[radius=.08, fill=black](.75,.75)circle;
                              \draw[radius=.08, fill=black](2.75,.75)circle;
                              \draw[radius=.08, fill=black](1.75,2.5)circle;
                              \draw[radius=.08, fill=black](1.75,1.75)circle;
    \node at (.75,-.5){$0$};
                             \node at (2.75,-.5){$0$};
                             \node at (0,1) {\tiny{$d$}};
                             \node at (1.75, .75) {\tiny{$d$}};
                              \node at (1.75, 3.5) {\tiny{$d-2$}};
                              \node at (3.75, 1) {\tiny{$d$}};
                              \node at (.75,2) {\tiny{$d-1$}};
                               \node at (2.75,2) {\tiny{$d-1$}};
                             
                              \begin{scope}[thick,decoration={
                                             markings,
                                             mark=at position 0.5 with {\arrow{>}}}
                              ]
                              \draw[postaction={decorate}] (.75,0)--(.75,.75);
                               \draw[postaction={decorate}] (2.75,0)--(2.75,.75);
                              \draw[postaction={decorate}] (1.75,1.75)--(.75,.75);
                               \draw[postaction={decorate}] (1.75,1.75)--(2.75,.75);
                                \draw[postaction={decorate}] (-.5,2)--(.75,.75);
                                 \draw[postaction={decorate}] (4,2)--(2.75,.75);
                                 \draw[postaction={decorate}] (1.75,1.75)--(1.75,2.5);
                                  \draw[postaction={decorate}] (-.5,3)--(1.75,2.5);
                                 \draw[postaction={decorate}] (4,3)--(1.75,2.5);
                              \end{scope}
                              \end{tikzpicture}}$ \\ 
                              & $s_i \cdot w = w+\widehat{w} =  \raisebox{-15pt}{\begin{tikzpicture}[scale=.4]
                              \draw[style=dashed, <->] (-1,0)--(4.75,0);
                              \draw[radius=.08, fill=black](.75,0)circle;
                              \draw[radius=.08, fill=black](2.75,0)circle;
                               \draw[radius=.08, fill=black](.75,.75)circle;
                              \draw[radius=.08, fill=black](2.75,.75)circle;
                              \draw[radius=.08, fill=black](1.75,2.5)circle;
                              \draw[radius=.08, fill=black](1.75,1.75)circle;

                             \node at (0,1) {\tiny{$d$}};
                             \node at (1.75, .75) {\tiny{$d$}};
                              \node at (1.75, 3.5) {\tiny{$d-2$}};
                              \node at (3.75, 1) {\tiny{$d$}};
                              \node at (.75,2) {\tiny{$d-1$}};
                               \node at (2.75,2) {\tiny{$d-1$}};
                             
                              \begin{scope}[thick,decoration={
                                             markings,
                                             mark=at position 0.5 with {\arrow{>}}}
                              ]
                              \draw[postaction={decorate}] (.75,0)--(.75,.75);
                               \draw[postaction={decorate}] (2.75,0)--(2.75,.75);
                              \draw[postaction={decorate}] (1.75,1.75)--(.75,.75);
                               \draw[postaction={decorate}] (1.75,1.75)--(2.75,.75);
                                \draw[postaction={decorate}] (-.5,2)--(.75,.75);
                                 \draw[postaction={decorate}] (4,2)--(2.75,.75);
                                 \draw[postaction={decorate}] (1.75,1.75)--(1.75,2.5);
                                  \draw[postaction={decorate}] (-.5,3)--(1.75,2.5);
                                 \draw[postaction={decorate}] (4,3)--(1.75,2.5);
                              \end{scope}
                              \end{tikzpicture}}
                              +
                               \raisebox{-17pt}{\begin{tikzpicture}[scale=.4]
                              \draw[style=dashed, <->] (-1,0)--(4.75,0);
                              \draw[radius=.08, fill=black](.5,0)circle;
                              \draw[radius=.08, fill=black](3,0)circle;
                               \draw[radius=.08, fill=black](.75,1.25)circle;
                              \draw[radius=.08, fill=black](2.75,1.25)circle;
                              \draw[radius=.08, fill=black](1.75,2.5)circle;
                              \draw[radius=.08, fill=black](1.75,1.75)circle;
                              \draw[radius=.08, fill=black](1.75,.95)circle;
                               \draw[radius=.08, fill=black](1.75,.75)circle;

                             \node at (0,1) {\tiny{$d$}};
                             \node at (1.75, 1.35) {\tiny{$d$}};
                              \node at (1.75, 3.5) {\tiny{$d-2$}};
                               \node at (1.75, .25) {\tiny{$d+1$}};
                              \node at (3.75, 1) {\tiny{$d$}};
                              \node at (.75,2.25) {\tiny{$d-1$}};
                               \node at (2.75,2.25) {\tiny{$d-1$}};
                             
                              \begin{scope}[thick,decoration={
                                             markings,
                                             mark=at position 0.5 with {\arrow{>}}}
                              ]
                              \draw[postaction={decorate}] (1.75,.95)--(.75,1.25);
                               \draw[postaction={decorate}] (1.75,.95)--(1.75,.75);
                                \draw[postaction={decorate}] (.5,0) to[out=90,in=200] (1.75,.75);
                                 \draw[postaction={decorate}] (3,0)to[out=90,in=-20](1.75,.75);
                              \draw[postaction={decorate}] (1.75,.95)--(2.75,1.25);
                              \draw[postaction={decorate}] (1.75,1.75)--(.75,1.25);
                               \draw[postaction={decorate}] (1.75,1.75)--(2.75,1.25);
                                \draw[postaction={decorate}] (-.5,2)--(.75,1.25);
                                 \draw[postaction={decorate}] (4,2)--(2.75,1.25);
                                 \draw[postaction={decorate}] (1.75,1.75)--(1.75,2.5);
                                  \draw[postaction={decorate}] (-.5,3)--(1.75,2.5);
                                 \draw[postaction={decorate}] (4,3)--(1.75,2.5);
                              \end{scope}
                              \end{tikzpicture}} 
                              = \raisebox{-15pt}{\begin{tikzpicture}[scale=.4]
                              \draw[style=dashed, <->] (-1,0)--(4.75,0);
                              \draw[radius=.08, fill=black](.75,0)circle;
                              \draw[radius=.08, fill=black](2.75,0)circle;
                               \draw[radius=.08, fill=black](.75,.75)circle;
                              \draw[radius=.08, fill=black](2.75,.75)circle;
                              \draw[radius=.08, fill=black](1.75,2.5)circle;
                              \draw[radius=.08, fill=black](1.75,1.75)circle;

                             \node at (0,1) {\tiny{$d$}};
                             \node at (1.75, .75) {\tiny{$d$}};
                              \node at (1.75, 3.5) {\tiny{$d-2$}};
                              \node at (3.75, 1) {\tiny{$d$}};
                              \node at (.75,2) {\tiny{$d-1$}};
                               \node at (2.75,2) {\tiny{$d-1$}};
                             
                              \begin{scope}[thick,decoration={
                                             markings,
                                             mark=at position 0.5 with {\arrow{>}}}
                              ]
                              \draw[postaction={decorate}] (.75,0)--(.75,.75);
                               \draw[postaction={decorate}] (2.75,0)--(2.75,.75);
                              \draw[postaction={decorate}] (1.75,1.75)--(.75,.75);
                               \draw[postaction={decorate}] (1.75,1.75)--(2.75,.75);
                                \draw[postaction={decorate}] (-.5,2)--(.75,.75);
                                 \draw[postaction={decorate}] (4,2)--(2.75,.75);
                                 \draw[postaction={decorate}] (1.75,1.75)--(1.75,2.5);
                                  \draw[postaction={decorate}] (-.5,3)--(1.75,2.5);
                                 \draw[postaction={decorate}] (4,3)--(1.75,2.5);
                              \end{scope}
                              \end{tikzpicture}} +
                               \raisebox{-17pt}{\begin{tikzpicture}[scale=.4]
                              \draw[style=dashed, <->] (-1,0)--(4.75,0);
                              \draw[radius=.08, fill=black](.5,0)circle;
                              \draw[radius=.08, fill=black](3,0)circle;
                       
                              \draw[radius=.08, fill=black](1.75,2.5)circle;
                             
                               \draw[radius=.08, fill=black](1.75,.75)circle;

                             \node at (1.75, 1.35) {\tiny{$d-1$}};
                              \node at (1.75, 3.5) {\tiny{$d-2$}};
                               \node at (1.75, .35) {\tiny{$d$}};
                              \node at (.75,2.25) {\tiny{$d-1$}};
                               \node at (3.75, 1) {\tiny{$d$}};

                              \begin{scope}[thick,decoration={
                                             markings,
                                             mark=at position 0.5 with {\arrow{>}}}
                              ]

                               \draw[postaction={decorate}] (4,2)--(1.75,.75);
                                \draw[postaction={decorate}] (.5,0) to[out=90,in=200] (1.75,.75);
                                 \draw[postaction={decorate}] (3,0)to[out=90,in=-20](1.75,.75);
                    
                                 \draw[postaction={decorate}] (-.5,1.75)to[out=0, in=270] (1.75,2.5);
                                  \draw[postaction={decorate}] (-.5,3)--(1.75,2.5);
                                 \draw[postaction={decorate}] (4,3)--(1.75,2.5);
                              \end{scope}
                              \end{tikzpicture}}+
                              
                               \raisebox{-17pt}{\begin{tikzpicture}[scale=.4]
                              \draw[style=dashed, <->] (-1,0)--(4.75,0);
                              \draw[radius=.08, fill=black](.5,0)circle;
                              \draw[radius=.08, fill=black](3,0)circle;
                       
                              \draw[radius=.08, fill=black](1.75,2.5)circle;
                             
                               \draw[radius=.08, fill=black](1.75,.75)circle;

                             \node at (1.75, 1.45) {\tiny{$d-1$}};
                              \node at (1.75, 3.5) {\tiny{$d-2$}};
                               \node at (1.75, .35) {\tiny{$d$}};
                              \node at (2.75,2.25) {\tiny{$d-1$}};
                               \node at (0, 1) {\tiny{$d$}};

                              \begin{scope}[thick,decoration={
                                             markings,
                                             mark=at position 0.5 with {\arrow{>}}}
                              ]

                               \draw[postaction={decorate}] (-.5,2)--(1.75,.75);
                                \draw[postaction={decorate}] (.5,0) to[out=90,in=200] (1.75,.75);
                                 \draw[postaction={decorate}] (3,0)to[out=90,in=-20](1.75,.75);
                    
                                 \draw[postaction={decorate}] (4,2)to[out=180, in=270] (1.75,2.5);
                                  \draw[postaction={decorate}] (-.5,3)--(1.75,2.5);
                                 \draw[postaction={decorate}] (4,3)--(1.75,2.5);
                              \end{scope}
                              \end{tikzpicture}}
                              $  \\ 
                              & \\
                              & {\bf Conclusion}: The web $\widehat{w}$ is never reduced. All reduced summands $\widetilde{w}$ of $\widehat{w}$ satisfy $\widetilde{w} \prec_S w$ by Lemmas \ref{lem:shadowred} and \ref{lem:depthshad} \\ 
                              & \\ \hline
&   \\ 
&   In this case $w \stackrel{s_i}{\longrightarrow} w'$ is an edge in the Hasse diagram for $\prec_T$.\\
& \\ 
& $w = \raisebox{-25pt}{\begin{tikzpicture}[scale=.4]
                              \draw[style=dashed, <->] (-1,0)--(4.75,0);
                              \draw[radius=.08, fill=black](.75,0)circle;
                              \draw[radius=.08, fill=black](2.75,0)circle;
                               \node at (.75,-.5){$-$};
                             \node at (2.75,-.5){$0$};
                             \node at (0,1) {\tiny{$d$}};
                             \node at (1.75, 1) {\tiny{$d-1$}};
                              \node at (3.75, 1) {\tiny{$d-1$}};
                             
                              \begin{scope}[thick,decoration={
                                             markings,
                                             mark=at position 0.5 with {\arrow{>}}}
                              ]
                              \draw[postaction={decorate}] (.75,0)--(.75,3);
                              \draw[postaction={decorate}] (2.75,0)--(2.75,3);
                              \end{scope}
                              \end{tikzpicture}}$ \hspace{.5in} $w' = \raisebox{-25pt}{\begin{tikzpicture}[scale=.4]
                              \draw[style=dashed, <->] (-1,0)--(4.75,0);
                              \draw[radius=.08, fill=black](.75,0)circle;
                              \draw[radius=.08, fill=black](2.75,0)circle;
                               \node at (.75,-.5){$0$};
                             \node at (2.75,-.5){$-$};
                             \node at (0,1) {\tiny{$d$}};
                             \node at (1.75, 1) {\tiny{$d$}};
                              \node at (3.75, 1) {\tiny{$d-1$}};
                             
                              \begin{scope}[thick,decoration={
                                             markings,
                                             mark=at position 0.5 with {\arrow{>}}}
                              ]
                              \draw[postaction={decorate}] (.75,0)--(.75,3);
                              \draw[postaction={decorate}] (2.75,0)--(2.75,3);
                              \end{scope}
                              \end{tikzpicture}}$ \\
                            
\begin{huge}$-0$\end{huge} & $s_i \cdot w = w+\widehat{w} =  \raisebox{-15pt}{\begin{tikzpicture}[scale=.4]
                              \draw[style=dashed, <->] (-1,0)--(4.75,0);
                              \draw[radius=.08, fill=black](.75,0)circle;
                              \draw[radius=.08, fill=black](2.75,0)circle;
                             \node at (0,1) {\tiny{$d$}};
                             \node at (1.75, 1) {\tiny{$d-1$}};
                              \node at (3.75, 1) {\tiny{$d-1$}};
                             
                              \begin{scope}[thick,decoration={
                                             markings,
                                             mark=at position 0.5 with {\arrow{>}}}
                              ]
                              \draw[postaction={decorate}] (.75,0)--(.75,3);
                              \draw[postaction={decorate}] (2.75,0)--(2.75,3);
                              \end{scope}
                              \end{tikzpicture}}
                              +
                              \raisebox{-15pt}{\begin{tikzpicture}[scale=.4]
                              \draw[style=dashed, <->] (-1,0)--(4.75,0);
                              \draw[radius=.08, fill=black](.75,0)circle;
                              \draw[radius=.08, fill=black](2.75,0)circle;
                              \draw[radius=.08, fill=black](1.75,1.5)circle;
                              \draw[radius=.08, fill=black](1.75,2)circle;

                             \node at (0,1) {\tiny{$d$}};
                             \node at (1.75, .3) {\tiny{$d$}};
                              \node at (1.75, 3) {\tiny{$d-1$}};
                              \node at (3.75, 1) {\tiny{$d-1$}};
                             
                              \begin{scope}[thick,decoration={
                                             markings,
                                             mark=at position 0.5 with {\arrow{>}}}
                              ]
                              \draw[postaction={decorate}] (.75,0) to[out=90,in=225](1.75,1.5);
                              \draw[postaction={decorate}] (2.75,0)to[out=90,in=-45](1.75,1.5);
                               \draw[postaction={decorate}] (1.75,2)--(1.75,1.5);
                                \draw[postaction={decorate}] (1.75,2)--(.75,3);
                                 \draw[postaction={decorate}] (1.75,2)--(2.75,3);
                              \end{scope}
                              \end{tikzpicture}}
                              $  \\
                              & \\ 
&  {\bf Conclusion}: All reduced summands $\widetilde{w}$ of $\widehat{w}$ satisfy $\widetilde{w} \prec_S w'$ by Lemmas \ref{lem:shadowred} and \ref{lem:depthshad} \\ 
& \\ 
& \\ \hline
\end{tabular}}
\caption{Case analysis examples for Lemmas \ref{lem:edgeshadow0}, \ref{lem:edgeshadow1}, and \ref{lem:edgeshadow2}}\label{fig:twistcases}
\end{figure}

To show that $\phi$ is unitriangular with respect to the shadow partial ordering, we need something stronger than Lemma \ref{lem:edgeshadow2}: namely, we need to know that $s_i \cdot w$ is the sum of $w+w'$ plus reduced summands that are less than $w$ in shadow order, not just less than $w'$. We suspect this is always true and can prove it in various cases. However, there is an obstacle to a general proof: changing depth in one part of a web diagram can propagate to other parts of the diagram and create cascading sequences of unreduced webs.

Fortunately, in order to prove Theorem \ref{thm:coefficients} it suffices that this claim hold just for {\em some} of the edges coming into $w'$. The following lemma proves this claim for the smallest-indexed edge coming into each web $w'$ other than the minimal web $w_0$ (which has no edges coming into it).  In this case, we can control the changes to depth created by braiding according to $s_i$.  This is the substance of the analysis in the following proof. 

\begin{lemma}\label{lem:edgeshadow3}
Let $w'$ be a reduced web other than $w_0$ and consider the smallest $i$ such that there exists an edge $w\stackrel{s_i}{\longrightarrow} w'$ in the Hasse diagram for $\prec_T$. Then $$s_i \cdot w = w+w'+ \sum_{\{\widetilde{w}:\, \widetilde{w}\prec_S w\}} \widetilde{c}\widetilde{w}$$ where $\widetilde{c}\in \mathbb{Z}$.
\end{lemma}

\begin{proof}
Let $w$ and $w'$ satisfy the hypotheses of the lemma, so $i,i+1$ are the first pair in the tableau for $w'$ in which $i$ is above and in a separate column from $i+1$.  (This is a necessary condition for there to be an edge labeled $s_i$ directed into $w'$ in the Hasse diagram for $\prec_T$.)  Since the tableau for $w'$ is standard, it is equivalent to the statement that $i$ is in a column further right than $i+1$.  This in turn is equivalent to $i+1$ being the first number inserted into a row that was previously skipped, in the following sense.  The boundary word for $w'$ up to and including the $(i+1)^{st}$ symbol must have one of these three forms:
\begin{itemize}
\item Case 1: $(+0-)^p(+0)^q(+)^{r}+0$ where $p\geq 0$, $q\geq 0$, and $r\geq 1$: the first row grows at least two entries longer than the second row, and then $i+1$ is inserted into the second row
\item Case 2: $(+0-)^p(+0)^q(+)^r+-$ where $p\geq 0$, $q\geq 1$, and $r\geq 0$: the first row grows at least two entries longer than the third row, and then $i+1$ is inserted into the third row (possibly to the left of the first available spot in the second row)
\item Case 3: $(+0-)^p(+0)^q+0-$ where $p\geq 0$ and $q\geq 1$: the first two rows are the same length, this length is at least two entries longer than the third row, and then $i+1$ is inserted into the third row
\end{itemize}
The boundary word for $w$ is the same as that for $w'$ except that the $i$ and $(i+1)^{st}$ symbols are exchanged. 

To complete the proof, we sketch the relevant parts of the $M$-diagram for $w$, the web $w$, and $s_i \cdot w$  in each case. We indicate the $i$ and $(i+1)^{st}$ endpoints in both $M$-diagrams and webs by large vertices.  The region marked with a star is a face that a priori might become a square after applying $s_i$ and thus merits attention.

First, consider a web $w$ with boundary word as indicated in Case 1. The web and its $M$-diagram have the following local structure.
$$
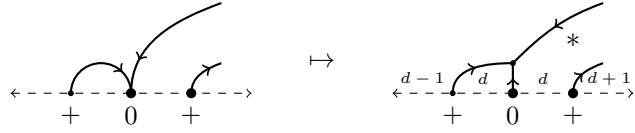
\begin{tikzpicture}[scale=.4]
                              \draw[style=dashed, <->] (-2,0)--(6,0);
                              \draw[radius=.08, fill=black](0,0)circle;
                              \draw[radius=.15, fill=black](2,0)circle;
                               \draw[radius=.15, fill=black](4,0)circle;
                               \node at (0,-.75){$+$};
                             \node at (2,-.75){$0$};
                             \node at (4,-.75) {$+$};
                                                           \begin{scope}[thick,decoration={
                                             markings,
                                             mark=at position 0.75 with {\arrow{>}}}
                              ]
                              \draw[postaction={decorate}] (0,0)to[out=90,in=180](1,1) to[out=0,in=90] (2,0);
                              \draw[postaction={decorate}] (4,0)to[out=90,in=200](5,1);
                              \draw[postaction={decorate}]  (5,3) to[out=200,in=90] (2,0);
                        
                              \end{scope}
                              \end{tikzpicture} \hspace{.25in} \textup{ \raisebox{25pt}{$\mapsto$} }  \hspace{.25in}
                              \begin{tikzpicture}[scale=.4]
                              \draw[style=dashed, <->] (-2,0)--(6,0);
                              \draw[radius=.08, fill=black](0,0)circle;
                              \draw[radius=.15, fill=black](2,0)circle;
                               \draw[radius=.15, fill=black](4,0)circle;
                               \draw[radius=.08, fill=black](2,1)circle;
                               \node at (0,-.75){$+$};
                             \node at (2,-.75){$0$};
                             \node at (4,-.75) {$+$};
                                                    \node at (1,.5) {\tiny{$d$}};
                             \node at (3, .5) {\tiny{$d$}};
                             \node at (-1, .5) {\tiny{$d-1$}};
                              \node at (5.2, .5) {\tiny{$d+1$}};
                                \node at (4, 1.75) {\large{$*$}};
                                                           \begin{scope}[thick,decoration={
                                             markings,
                                             mark=at position 0.5 with {\arrow{>}}}
                              ]
                              \draw[postaction={decorate}] (0,0)to[out=90,in=180](2,1);
                              \draw[postaction={decorate}] (2,0)--(2,1);
                              \draw[postaction={decorate}] (4,0)to[out=90,in=200](5,1);
                              \draw[postaction={decorate}]  (5,3) to[out=200,in=45] (2,1);
                       
                              \end{scope}
                              \end{tikzpicture}$$
                              It is possible that $M$-diagram arcs incident to the $i$ and $(i+1)^{st}$ boundary vertices exiting to the right may cross.  If they do not cross, the starred region in $w$ necessarily has more than 4 edges (not including the boundary) or touches the boundary elsewhere.
                   
                              If the arcs cross, then the starred region has exactly 4 edges not including the boundary.

Consider $s_i\cdot w = w+\widehat{w}$ when the arcs do not cross. 
 $$s_i \cdot w = w+\widehat{w} = \raisebox{-10pt}
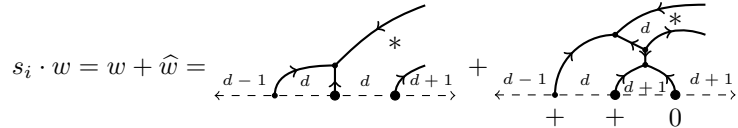
{\begin{tikzpicture}[scale=.4]
                              \draw[style=dashed, <->] (-2,0)--(6,0);
                              \draw[radius=.08, fill=black](0,0)circle;
                              \draw[radius=.15, fill=black](2,0)circle;
                               \draw[radius=.15, fill=black](4,0)circle;
                               \draw[radius=.08, fill=black](2,1)circle;

                                                    \node at (1,.5) {\tiny{$d$}};
                                                    \node at (-1, .5) {\tiny{$d-1$}};
                             \node at (3, .5) {\tiny{$d$}};
                              \node at (5.2, .5) {\tiny{$d+1$}};
                                \node at (4, 1.75) {\large{$*$}};
                                                           \begin{scope}[thick,decoration={
                                             markings,
                                             mark=at position 0.5 with {\arrow{>}}}
                              ]
                              \draw[postaction={decorate}] (0,0)to[out=90,in=180](2,1);
                              \draw[postaction={decorate}] (2,0)--(2,1);
                              \draw[postaction={decorate}] (4,0)to[out=90,in=200](5,1);
                              \draw[postaction={decorate}]  (5,3) to[out=200,in=45] (2,1);
                        
                              \end{scope}
                              \end{tikzpicture}}+\raisebox{-23pt}{\begin{tikzpicture}[scale=.4]
                              \draw[style=dashed, <->] (-2,0)--(6,0);
                              \draw[radius=.08, fill=black](0,0)circle;
                              \draw[radius=.15, fill=black](2,0)circle;
                               \draw[radius=.15, fill=black](4,0)circle;
                               \draw[radius=.08, fill=black](2,2)circle;
                                \draw[radius=.08, fill=black](3,1)circle;
                                 \draw[radius=.08, fill=black](3,1.5)circle;
                               \node at (0,-.75){$+$};
                             \node at (2,-.75){$+$};
                             \node at (4,-.75) {$0$};
                                                    \node at (1,.5) {\tiny{$d$}};
                                                    \node at (-1, .5) {\tiny{$d-1$}};
                             \node at (3, 2.25) {\tiny{$d$}};
                              \node at (3, .25) {\tiny{$d+1$}};
                              \node at (5.2, .5) {\tiny{$d+1$}};
                                \node at (4, 2.5) {\large{$*$}};
                                                           \begin{scope}[thick,decoration={
                                             markings,
                                             mark=at position 0.5 with {\arrow{>}}}
                              ]
                              \draw[postaction={decorate}] (0,0)to[out=90,in=180](2,2);
                              \draw[postaction={decorate}] (2,0) to[out=90,in=200] (3,1);
                              \draw[postaction={decorate}] (3,1.5)--(2,2);
                               \draw[postaction={decorate}] (3,1.5)--(3,1);
                                \draw[postaction={decorate}] (4,0) to[out=90,in=-20](3,1);
                              \draw[postaction={decorate}] (3,1.5)to[out=60,in=160](5,2);
                              \draw[postaction={decorate}]  (5,3) to[out=180,in=45] (2,2);
                        
                              \end{scope}
                              \end{tikzpicture}}$$
 In this subcase $\widehat{w}$ is reduced (since the starred region is not a square or bigon) and has a boundary word identical to $w'$. We conclude $\widehat{w}=w'$ and $s_i \cdot w = w+w'$. 
 
 Next, consider the structure of the web $\widehat{w}$ if the arcs incident to vertices $i$ and $i+1$ do cross in the $M$-diagram for $w$. Then $\widehat{w}$ can be reduced as follows
  $$\widehat{w} = 
\raisebox{-10pt}
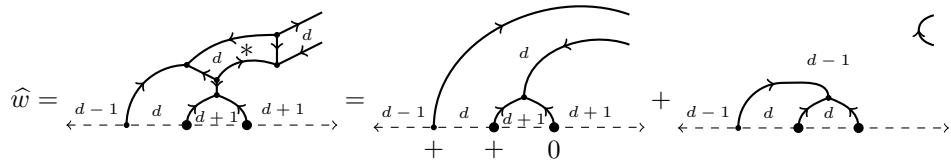
{\begin{tikzpicture}[scale=.4]
                              \draw[style=dashed, <->] (-2,0)--(7,0);
                              \draw[radius=.08, fill=black](0,0)circle;
                              \draw[radius=.15, fill=black](2,0)circle;
                               \draw[radius=.15, fill=black](4,0)circle;
                               \draw[radius=.08, fill=black](2,2)circle;
                                \draw[radius=.08, fill=black](3,1)circle;
                                 \draw[radius=.08, fill=black](3,1.5)circle;
                                 \draw[radius=.08, fill=black](5,2)circle;
                                 \draw[radius=.08, fill=black](5,3)circle;
                                                    \node at (1,.5) {\tiny{$d$}};
                                                    \node at (-1, .5) {\tiny{$d-1$}};
                             \node at (3, 2.25) {\tiny{$d$}};
                              \node at (3, .25) {\tiny{$d+1$}};
                              \node at (5.2, .5) {\tiny{$d+1$}};
                               \node at (6, 3) {\tiny{$d$}};
                                \node at (4, 2.5) {\large{$*$}};
                                                           \begin{scope}[thick,decoration={
                                             markings,
                                             mark=at position 0.5 with {\arrow{>}}}
                              ]
                              \draw[postaction={decorate}] (0,0)to[out=90,in=180](2,2);
                              \draw[postaction={decorate}] (2,0) to[out=90,in=200] (3,1);
                              \draw[postaction={decorate}] (3,1.5)--(2,2);
                               \draw[postaction={decorate}] (3,1.5)--(3,1);
                                \draw[postaction={decorate}] (5,3)--(5,2);
                                \draw[postaction={decorate}] (5,3)--(6.5,3.75);
                                \draw[postaction={decorate}] (6.5,2.75)--(5,2);
                                \draw[postaction={decorate}] (4,0) to[out=90,in=-20](3,1);
                              \draw[postaction={decorate}] (3,1.5)to[out=60,in=160](5,2);
                              \draw[postaction={decorate}]  (5,3) to[out=180,in=45] (2,2);

                              \end{scope}
                              \end{tikzpicture}}=
                              \raisebox{-23pt}{\begin{tikzpicture}[scale=.4]
                              \draw[style=dashed, <->] (-2,0)--(7,0);
                              \draw[radius=.08, fill=black](0,0)circle;
                              \draw[radius=.15, fill=black](2,0)circle;
                               \draw[radius=.15, fill=black](4,0)circle;
                                \draw[radius=.08, fill=black](3,1)circle;
                               \node at (0,-.75){$+$};
                             \node at (2,-.75){$+$};
                             \node at (4,-.75) {$0$};
                                                    \node at (1,.5) {\tiny{$d$}};                                         
                                                     \node at (-1, .5) {\tiny{$d-1$}};
                             \node at (3, 2.5) {\tiny{$d$}};
                              \node at (3, .25) {\tiny{$d+1$}};
                              \node at (5.2, .5) {\tiny{$d+1$}};

                                                           \begin{scope}[thick,decoration={
                                             markings,
                                             mark=at position 0.5 with {\arrow{>}}}
                              ]
                              \draw[postaction={decorate}] (0,0)to[out=90,in=160](6.5,3.75);
                              \draw[postaction={decorate}] (2,0) to[out=90,in=200] (3,1);
                                \draw[postaction={decorate}] (4,0) to[out=90,in=-20](3,1);
                              \draw[postaction={decorate}] (6.5,2.75)to[out=160,in=90](3,1);
                              \end{scope}
                              \end{tikzpicture}}
                              +
                              \raisebox{-10pt}{\begin{tikzpicture}[scale=.4]
                              \draw[style=dashed, <->] (-2,0)--(7,0);
                             \draw[radius=.08, fill=black](0,0)circle;
                              \draw[radius=.15, fill=black](2,0)circle;
                               \draw[radius=.15, fill=black](4,0)circle;
                                \draw[radius=.08, fill=black](3,1)circle;
                                                    \node at (1,.5) {\tiny{$d$}};
                                                    \node at (-1, .5) {\tiny{$d-1$}};
                             \node at (3, 2.25) {\tiny{$d-1$}};
                              \node at (3, .5) {\tiny{$d$}};

                                                           \begin{scope}[thick,decoration={
                                             markings,
                                             mark=at position 0.5 with {\arrow{>}}}
                              ]
                              \draw[postaction={decorate}] (0,0)to[out=90,in=180](1.5,1.5) to[out=0,in=90] (3,1);
                              \draw[postaction={decorate}] (2,0) to[out=90,in=200] (3,1);
                
                                \draw[postaction={decorate}] (4,0) to[out=90,in=-20](3,1);
         \draw[postaction={decorate}] (6.5,2.75) to[out=165,in=270] (6,3.25) to[out=90,in=195] (6.5,3.75);
                       
                              \end{scope}
                              \end{tikzpicture}}
                              $$
In this subcase, the first summand in $\widehat{w}$ is reduced because the only faces that were modified lie on the boundary; moreover, its boundary matches $w'$. By Remark \ref{rem:sl3undet}, the boundary word uniquely determines a reduced web, so we conclude the summand is $w'$.   The second summand is a possibly-unreduced web; the depths in those regions are indicated, and the depths along the boundary are at most those in $w$. Therefore by Lemma \ref{lem:shadowred} we obtain $s_i \cdot w = w+w'+ \sum_{\{\widetilde{w}:\widetilde{w}\prec_S w\}} \widetilde{c} \widetilde{w}$ where $\widetilde{c}\in \mathbb{Z}$. 

Now, consider a web $w$ with boundary word of the form given in Case 2. The following is a local picture of both its $M$-diagram and the web itself.
 $$
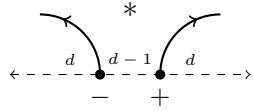
\begin{tikzpicture}[scale=.4]
                              \draw[style=dashed, <->] (-1,0)--(7,0);
  
                              \draw[radius=.15, fill=black](2,0)circle;
                               \draw[radius=.15, fill=black](4,0)circle;
                               \node at (1,.5){\tiny{$d$}};
                                \node at (3,.5){\tiny{$d-1$}};
                                 \node at (3,2){\Large{*}};
                                  \node at (5,.5){\tiny{$d$}};
                             \node at (2,-.75){$-$};
                             \node at (4,-.75) {$+$};
                                                           \begin{scope}[thick,decoration={
                                             markings,
                                             mark=at position 0.75 with {\arrow{>}}}
                              ]
                              \draw[postaction={decorate}] (2,0)to[out=90,in=0](0,2);
                              \draw[postaction={decorate}] (4,0)to[out=90,in=180](6,2);

                              \end{scope}
                              \end{tikzpicture}$$
                              By construction, no arc of the $M$-diagram for $w$ can intersect both of the arcs shown, since arcs between numbers on the first two rows of the tableau only intersect arcs between numbers on the second two rows of the tableau. Thus the starred region in this case will have more than 4 edges (not counting the boundary) or will touch the boundary elsewhere. Arguing as in the first subcase of Case 1, we conclude that $s_i \cdot w = w+\widehat{w}$ where $\widehat{w}$ is reduced and $\widehat{w}=w'$.
                              
Finally, consider a web $w$ with boundary word as indicated in Case 3. Local pictures of the $M$-diagram for $w$ and of $w$ are again given below. Note that the starred region in $w$ touches the boundary, as deduced from the $M$-diagram (which shows vertices $i-3,i-2,\ldots,i+1$).  This will be important in what follows. 
$$
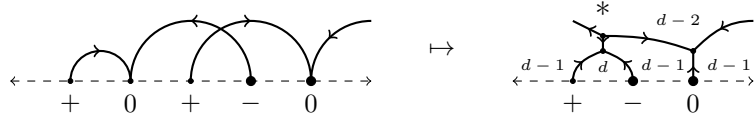
\begin{tikzpicture}[scale=.4]
                              \draw[style=dashed, <->] (-6,0)--(6,0);
                              \draw[radius=.08, fill=black](0,0)circle;
                              \draw[radius=.15, fill=black](2,0)circle;
                               \draw[radius=.08, fill=black](-4,0)circle;
                                \draw[radius=.08, fill=black](-2,0)circle;
                               \draw[radius=.15, fill=black](4,0)circle;
                               \node at (0,-.75){$+$};
                             \node at (2,-.75){$-$};
                             \node at (4,-.75) {$0$};
                               \node at (-2,-.75){$0$};
                             \node at (-4,-.75) {$+$};
                                                           \begin{scope}[thick,decoration={
                                             markings,
                                             mark=at position 0.5 with {\arrow{>}}}
                              ]
                              \draw[postaction={decorate}] (0,0)to[out=90,in=180](2,2) to[out=0,in=90] (4,0);
                                \draw[postaction={decorate}] (-4,0)to[out=90,in=180](-3,1) to[out=0,in=90] (-2,0);
                              \draw[postaction={decorate}] (6,2)to[out=180,in=90](4,0);
                              \draw[postaction={decorate}]  (2,0) to[out=90,in=0] (0,2) to[out=180,in=90] (-2,0);
                       
                              \end{scope}
                              \end{tikzpicture} \hspace{.25in} \textup{ \raisebox{25pt}{$\mapsto$} }  \hspace{.25in}
                              \begin{tikzpicture}[scale=.4]
                              \draw[style=dashed, <->] (-2,0)--(6,0);
                              \draw[radius=.08, fill=black](0,0)circle;
                              \draw[radius=.08, fill=black](1,1)circle;
                              \draw[radius=.08, fill=black](1,1.5)circle;
                              \draw[radius=.15, fill=black](2,0)circle;
                               \draw[radius=.15, fill=black](4,0)circle;
                               \draw[radius=.08, fill=black](4,1)circle;
                               \node at (0,-.75){$+$};
                             \node at (2,-.75){$-$};
                             \node at (4,-.75) {$0$};
                                                    \node at (1,.5) {\tiny{$d$}};
                                                    \node at (1,2.25) {\Large{*}};
                             \node at (3, .5) {\tiny{$d-1$}};
                             \node at (-1, .5) {\tiny{$d-1$}};
                              \node at (5.2, .5) {\tiny{$d-1$}};
                                \node at (3.5, 2) {\tiny{$d-2$}};
                                                           \begin{scope}[thick,decoration={
                                             markings,
                                             mark=at position 0.5 with {\arrow{>}}}
                              ]
                              \draw[postaction={decorate}] (0,0)to[out=90,in=200](1,1);
                              \draw[postaction={decorate}] (2,0)to[out=90,in=-20](1,1);
                               \draw[postaction={decorate}] (1,1.5)--(1,1);
                                 \draw[postaction={decorate}] (1,1.5)--(0,2);
                                   \draw[postaction={decorate}] (1,1.5) to[out=0,in=160] (4,1);
                              \draw[postaction={decorate}] (4,0)--(4,1);
                               \draw[postaction={decorate}] (6,2)to[out=180,in=45](4,1);
                        
                              \end{scope}
                              \end{tikzpicture}$$
                             Acting on $w$ by $s_i$ gives $s_i\cdot w = w+\widehat{w}$ where $\widehat{w}$ reduces as shown below.
                             $$ \widehat{w} = \raisebox{-10pt}
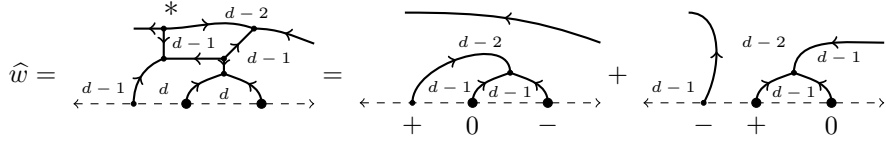
{  \begin{tikzpicture}[scale=.4]
                              \draw[style=dashed, <->] (-2,0)--(6,0);
                              \draw[radius=.08, fill=black](0,0)circle;
                              \draw[radius=.08, fill=black](1,2.5)circle;
                              \draw[radius=.08, fill=black](1,1.5)circle;
                              \draw[radius=.15, fill=black](1.75,0)circle;
                               \draw[radius=.15, fill=black](4.25,0)circle;
                               \draw[radius=.08, fill=black](4,2.5)circle;
                               \draw[radius=.08, fill=black](3,1)circle;
                               \draw[radius=.08, fill=black](3,1.5)circle;
                    
                                                    \node at (1,.5) {\tiny{$d$}};
                                                      \node at (1.25,3) {\Large{*}};
                             \node at (3, .35) {\tiny{$d$}};
                             \node at (-1, .5) {\tiny{$d-1$}};
                              \node at (2, 2) {\tiny{$d-1$}};
                              \node at (4.5, 1.5) {\tiny{$d-1$}};
                                \node at (3.75, 3) {\tiny{$d-2$}};
                                                           \begin{scope}[thick,decoration={
                                             markings,
                                             mark=at position 0.5 with {\arrow{>}}}
                              ]
                              \draw[postaction={decorate}] (0,0)to[out=90,in=200](1,1.5);
                           \draw[postaction={decorate}] (1.75,0) to[out=90,in=200] (3,1);
                               \draw[postaction={decorate}] (1,2.5)--(1,1.5);
                                 \draw[postaction={decorate}] (1,2.5)--(0,2.5);
                                   \draw[postaction={decorate}] (1,2.5) to[out=0,in=160] (4,2.5);
                              \draw[postaction={decorate}] (4.25,0) to[out=90,in=-20] (3,1);
           \draw[postaction={decorate}] (3,1.5)--(3,1);
            \draw[postaction={decorate}] (3,1.5)--(4,2.5);
             \draw[postaction={decorate}] (3,1.5)--(1,1.5);
                               \draw[postaction={decorate}] (6,2)to[out=160,in=0](4,2.5);
                        
                              \end{scope}
                              \end{tikzpicture}}=
                               \raisebox{-23pt}{  \begin{tikzpicture}[scale=.4]
                              \draw[style=dashed, <->] (-2,0)--(6,0);
                              \draw[radius=.08, fill=black](-.25,0)circle;
                              \draw[radius=.15, fill=black](1.75,0)circle;
                               \draw[radius=.15, fill=black](4.25,0)circle;
                               \draw[radius=.08, fill=black](3,1)circle;

                               \node at (-.25,-.75){$+$};
                             \node at (1.75,-.75){$0$};
                             \node at (4.25,-.75) {$-$};
                                                    \node at (3,.35) {\tiny{$d-1$}};
                             \node at (1, .5) {\tiny{$d-1$}};
                              \node at (2, 2) {\tiny{$d-2$}};
                      
                                                           \begin{scope}[thick,decoration={
                                             markings,
                                             mark=at position 0.5 with {\arrow{>}}}
                              ]
                              \draw[postaction={decorate}] (6,2)to[out=160,in=0](-.5,3);
                           \draw[postaction={decorate}] (1.75,0) to[out=90,in=200] (3,1);

                              \draw[postaction={decorate}] (4.25,0) to[out=90,in=-20] (3,1);

                               \draw[postaction={decorate}] (-.25,0)to[out=90,in=90](3,1);
                        
                              \end{scope}
                              \end{tikzpicture}}
                              +
                              \raisebox{-23pt}{  \begin{tikzpicture}[scale=.4]
                              \draw[style=dashed, <->] (-2,0)--(6,0);
                              \draw[radius=.08, fill=black](0,0)circle;
                              \draw[radius=.15, fill=black](1.75,0)circle;
                               \draw[radius=.15, fill=black](4.25,0)circle;
                               \draw[radius=.08, fill=black](3,1)circle;

                               \node at (0,-.75){$-$};
                             \node at (1.75,-.75){$+$};
                             \node at (4.25,-.75) {$0$};
                                                    \node at (3,.35) {\tiny{$d-1$}};
                             \node at (-1, .5) {\tiny{$d-1$}};
                              \node at (2, 2) {\tiny{$d-2$}};
                              \node at (4.5, 1.5) {\tiny{$d-1$}};
                                                           \begin{scope}[thick,decoration={
                                             markings,
                                             mark=at position 0.5 with {\arrow{>}}}
                              ]
                              \draw[postaction={decorate}] (0,0)to[out=60,in=0](-.5,3);
                           \draw[postaction={decorate}] (1.75,0) to[out=90,in=200] (3,1);

                              \draw[postaction={decorate}] (4.25,0) to[out=90,in=-20] (3,1);

                               \draw[postaction={decorate}] (6,2)to[out=180,in=90](3,1);
                        
                              \end{scope}
                              \end{tikzpicture}}$$
                              Since all regions adjacent to the square in $\widehat{w}$ also meet the boundary, we see that $\widehat{w}$ splits into the sum of  two reduced webs, one of which is $w'$. Hence $s_i \cdot w = w+w'+\widetilde{w}$ where $\widetilde{w} \prec_S w$ by analyzing depths along the boundary.
\end{proof}

Lemmas \ref{lem:edgeshadow2}, \ref{lem:edgeshadow3}, and \ref{lem:shadowprop} already begin to give a sense of the relationships between the partial orders $\prec_S$ and $\prec_T$.  We have shown that for certain edges $w \stackrel{s_i}{\longrightarrow} w'$ in the Hasse diagram for $\prec_T$ the web $s_i \cdot w$ decomposes into a sum of $w+w'$ plus reduced summands in the shadow of $w'$.  In order to induct, we need to bound the shadow obtained by applying another (possibly different) simple transposition $s_j$ to each of these terms.  This is what the following lemma does.  It is analogous to the ``diamond lemma" in algebraic combinatorics \cite[Proposition 2.27]{MR2133266}, which proves that in the Hasse diagram for Bruhat order, any three of four edges in a diamond shape imply the fourth edge. (The diamond shape is exactly like that shown in Figure \ref{fig:diamond} for the Hasse diagram for $\prec_S$.  See also \cite{housley2015robinson} for a similar diamond lemma for $\mathfrak{sl}_2$ webs.)  This lemma is also the core of the inductive step in Theorem \ref{thm:coefficients}. 

\begin{lemma}\label{lem:shadowprop}
Let $w, w'$, and $\overline{w}$ be reduced webs such that $\overline{w}\stackrel{s_i}{\longrightarrow} w'$ is an edge in the Hasse diagram for $\prec_T$.  If $w\prec_S \overline{w}$ and $\widetilde{w}$ is a reduced summand of $s_i \cdot w$ then also $\widetilde{w}\prec_S w'$.
\end{lemma}

\begin{proof}
The edges $\overline{w}\stackrel{s_i}{\longrightarrow} w'$ are the covering relations in the poset $\prec_T$ so $\overline{w}\prec_T w'$.  Corollary \ref{cor:TvsS} then implies that $\overline{w}\prec_S w'$. Since both $w\prec_S \overline{w}$ and $\overline{w}\prec_S w'$ we know $w\prec_S w'$. Lemmas \ref{lem:edgeshadow0} and \ref{lem:edgeshadow1} showed that if the boundary symbols for $w$ in the $i$ and $(i+1)^{t}$ positions are $+0, +-, 0-, ++, 00$, or $--$, then all reduced summands $\widetilde{w}$ of $s_i\cdot w$ satisfy $\widetilde{w}\preceq_S w$. Since $w\prec_S w'$ we conclude $\widetilde{w}\prec_S w'$ in all of these cases as desired.

For the remainder of the proof, assume the $i$ and $(i+1)^{st}$ boundary symbols for $w$ are $-+$, $0-$, or $0+$, namely that $w$ has an edge labeled $s_i$ coming out of it. Let $w''$ denote the reduced web with $w \stackrel{s_i}{\longrightarrow} w''$ in the Hasse diagram for $\prec_T$.  In this case Lemma \ref{lem:edgeshadow2} applies and shows that all reduced summands $\widetilde{w}$ of $s_i\cdot w$ satisfy $\widetilde{w}\preceq_S w''$.  It only remains to show $w''\prec_S w'$ in Figure \ref{fig:diamond}.
\begin{figure}[h]
\begin{center}
\begin{tikzpicture}[scale=.5]
\node at (0,4) {$w$};
\node at (2,4) {\Large{$\prec_S$}};
\node at (4,4) {$ \overline{w}$}; 
\node at (0,0) {$w''$};
\node at (2,0) {??};
\node at (4,0) {$w'$};
\node at (-.5,2) {$s_i$};
\draw[->] (0,3.5)--(0,.5);
\node at (4.5,2) {$s_i$};
\draw[->] (4,3.5)--(4,.5);
\end{tikzpicture}
\end{center}
\caption{Diamond fragment in the partial order $\prec_S$} \label{fig:diamond}
\end{figure}
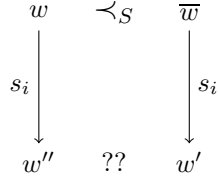

Lemma \ref{lem:edge to bdry word} proved that the boundary words for $w$ and $w''$ differ only at the $i$ and $(i+1)^{st}$ symbols, and the depths along the boundary differ only in the interval $(i,i+1)$.  Figure \ref{fig:yword} showed the depth between $i$ and $i+1$  increases by 1 from $w$ to $w''$ if the boundary word for $w$ is $0+$ or $-0$ in the $i$ and $(i+1)^{st}$ positions and increases by 2 if the boundary word is $-+$. The boundary depths in $\overline{w}$ and $w'$ are similarly constrained since there is also an edge $\overline{w}\stackrel{s_i}{\longrightarrow} w'$.

Suppose the depth at $i+\frac{1}{2}$ in $w$ is $d$ and in $\overline{w}$ is $e$.  We know $w \prec_S \overline{w}$ so $e \geq d$.  If $e > d$ then the depth at $i+\frac{1}{2}$ in $w'$ is at least $e+1$ and the depth at $i+\frac{1}{2}$ is at most $d+2 \leq e+1$ so $w'' \prec_S w'$ by Lemma \ref{lem:depthshad}.  If $e=d$ and the depth at $i+\frac{1}{2}$ in $w''$ is $d+1$ then again by Lemma \ref{lem:depthshad} we have $w'' \prec_S w'$.

Finally suppose $e=d$ and the depth at $i+\frac{1}{2}$ in $w''$ is $d+2$.  This means we are in Case 3 of Figure \ref{fig:yword} and so the boundary word for $w$ at the $i$ and $(i+1)^{st}$ positions is $-+$. It follows that the depths between $i-1$ and $i$ as well as between $i+1$ and $i+2$ are $d+1$ in $w$ and consequently $w''$. Since $w\prec_S \overline{w}$ we know that the depths in $\overline{w}$ between $i-1$ and $i$ as well as between $i+1$ and $i+2$ are at least $d+1$. We assumed  the depth between $i$ and $i+1$ in $\overline{w}$ is $d$ so the boundary word for $\overline{w}$ at the $i$ and $(i+1)^{st}$ positions must also be $-+$. This means the depth between $i$ and $i+1$ in $w'$ is $d+2$. It follows that $w''\prec_S w'$ as desired, proving the claim.
\end{proof}

Finally, we put these claims together to obtain the following theorem proving that the transition matrix from the Specht basis to $\mathfrak{sl}_3$ web basis is unitriangular with respect to the shadow partial order.

\begin{theorem}\label{thm:coefficients}
Let $T'\in SYT(n,n,n)$ and $w'$ be the reduced $\mathfrak{sl}_3$ web corresponding to $T'$. Then $$\phi(v_{T'}) = w' + \sum_{\widetilde{w}\in \mathcal{W}_{3n}^R} c_{\widetilde{w}}^{w'} \widetilde{w}$$ where $c_{\widetilde{w}}^{w'}= 0$ whenever $\widetilde{w} \not\prec_S w'$.
\end{theorem}

\begin{proof}
We prove this theorem by induction on $r(T')$. If $r(T')=0$ then $T'=T_0$. Lemma \ref{lem:basecase} shows $\phi(v_{T_0}) = w_0$ and the base case is proven. 

Assume the theorem holds for all tableaux of rank $k$. Consider $T'\in SYT(n,n,n)$ such that $r(T')=k+1$. Choose the smallest value $i$ such that there exists a reduced web $w$ and an edge $w\stackrel{s_i}{\longrightarrow} w'$ in the Hasse diagram for $\prec_T$ and let $T$ be the tableau corresponding to $w$. 

Proposition \ref{prop:tabact} stated that $s_i\cdot v_{T} = v_{T'}$. Since $\phi$ is equivariant with respect to the symmetric group, it follows that $\phi(v_{T'}) = \phi(s_i \cdot v_T) = s_i \cdot \phi(v_T)$. Since $r(T) = k$ the inductive assumption tells us that  $$\phi(v_{T'}) = s_i \cdot \phi(v_T)=s_i\cdot \left(w+\sum_{\widetilde{w}\in \mathcal{W}_{3n}^R} c_{\widetilde{w}}^{w} \widetilde{w}\right) = s_i \cdot w + \sum_{\widetilde{w}\in \mathcal{W}_{3n}^R} c_{\widetilde{w}}^{w} (s_i\cdot \widetilde{w})$$ where $c_{\widetilde{w}}^{w}= 0$ whenever $\widetilde{w} \not\prec_S w$.

First consider $s_i \cdot w$. By Lemma \ref{lem:edgeshadow3} we know $s_i\cdot w = w+w'+ \sum_{\{\widetilde{w}:\widetilde{w}\prec_S w\}} \widetilde{c} \widetilde{w}$ where $\widetilde{c}\in \mathbb{Z}$. Hence $w'$ occurs with a coefficient of 1 in $s_i\cdot w$ and moreover the shadow of every other reduced web in $s_i \cdot w$ is strictly contained in the shadow of $w'$ as desired.

Now let $\widetilde{w}$ be a reduced web such that $c_{\widetilde{w}}^w\neq 0$. Then $\widetilde{w}\prec_S w$. By Lemma \ref{lem:shadowprop} the shadow of every reduced summand of $s_i \cdot \widetilde{w}$ is strictly contained in the shadow of $w'$. Hence $\phi(v_{T'})$ is a linear combination of reduced webs in which $w'$ occurs with coefficient 1 and any other reduced web with a nonzero coefficient has shadow strictly contained in the shadow of $w'$.  By induction, this proves the claim.
\end{proof}

 Recall Corollary \ref{cor:total} showed that any total order completing $\prec_S$ also completes $\prec_T$. Hence we get the following, which answers the last remaining open question from our paper on $\mathfrak{sl}_2$ webs \cite{MR3920353}.

\begin{theorem}\label{thm:existtriang}
There exists a total order completion of $\prec_T$ with respect to which the matrix for $\phi$ is unitriangular.
\end{theorem}

We make three additional conjectures about these coefficients, all aimed at characterizing vanishing entries in the matrix for $\phi$. The first conjecture is that if $\widetilde{w} \not \prec_T w'$ then the coefficient $c_{\widetilde{w}}^{w'} = 0$.  At first glance, this might seem the same as Theorem \ref{thm:existtriang}, but in fact it upgrades the theorem from an existential statement to a universal one, in the sense that it means the matrix for $\phi$ is unitriangular with respect to {\em any} total order that completes the partial order $\prec_T$.

\begin{conjecture}\label{conj:taborder}
If $c_{\widetilde{w}}^{w'}\neq 0$ then $\widetilde{w}\prec_T w'$.
\end{conjecture}

The second conjecture gives a necessary but not sufficient condition for Conjecture \ref{conj:taborder}, by characterizing the relative ranks of the nonzero coefficients of $\phi$.

\begin{conjecture}\label{conj:coefrank}
If $c_{\widetilde{w}}^{w'}\neq 0$ then $r(\widetilde{w})< r(w')$.
\end{conjecture}

If it were true that $\widetilde{w}\prec_S w'$ implies $r(\widetilde{w}) < r(w')$ then Conjecture \ref{conj:coefrank} would follow. Unfortunately, the example in Figure \ref{fig:sl3Counter} shows this is not always the case. However, we only need $\widetilde{w} \prec_S s'$ to imply $r(\widetilde{w}) < r(w')$ in certain contexts.  In particular, in order for Conjecture \ref{conj:coefrank} to be true, it would be sufficient to show the shadow containment statements in Lemmas \ref{lem:edgeshadow1}, \ref{lem:edgeshadow2}, \ref{lem:edgeshadow3}, and \ref{lem:shadowprop} can be replaced with rank inequalities.  Thus we make the following conjecture.

\begin{conjecture} \label{conj:shadrank}
Let $\widetilde{w}$ be a reduced summand of $s_i \cdot w$. Then $r(\widetilde{w})\leq r(w)+1$. Furthermore if $\widetilde{w}\preceq_S w$ then $r(\widetilde{w}) \leq r(w)$.
\end{conjecture}

It is reasonable to wonder if the webs $\widetilde{w}$ and $w$ described in Figure \ref{fig:sl3Counter} provide a counterexample to Conjecture \ref{conj:shadrank}. The answer is no, per the following lemma.  Recall that we denote the number of crossings in an $M$-diagram by $c(M)$.

\begin{lemma}\label{lem:crossings}
Let $\widetilde{w}$ be a reduced summand of $s_i \cdot w$ and let $\widetilde{M}$ and $M$ be the $M$-diagrams for $\widetilde{w}$ and $w$ respectively. Then $c(\widetilde{M}) \leq c(w)+1$. Moreover $c(\widetilde{M}) = c(w)+1$ if and only if both $s_i \cdot w = w +\widetilde{w}$ and the boundary words for $\widetilde{w}$ and $w$ differ in one of the ways described in the following table. 

\begin{figure}[h]
\begin{tabular}{|c|c|c|c|}
\hline
&   {\bf Boundary word for $w$}   & {\bf Boundary word for $\widetilde{w}$} & {\bf Differ elsewhere?}  \\  \hline
{\bf Case 1}: & $0+$ & $+0$ & No  \\ \hline
{\bf Case 2}: & $-0$ & $0-$ & No  \\ \hline
{\bf Case 3}: & $-+$ & $+-$ & No  \\ \hline
{\bf Case 4}: & $++$ & $+0$ & Yes  \\ \hline
{\bf Case 5}: & $--$ & $0-$ & Yes  \\ \hline
\end{tabular}
\caption{ $i^{th}$ and $(i+1)^{st}$ boundary symbols when $s_i \cdot w = w+\widetilde{w}$} \label{fig:crossingcases}
\end{figure}
\end{lemma} 

\begin{proof}
Suppose $\widetilde{w}$, $w$, $\widetilde{M}$, and $M$ satisfy the hypotheses of the lemma. To obtain a web from an $M$-diagram, each crossing is replaced with a pair of internal vertices (one source and one sink), and the middle of each $M$ is replaced with a boundary edge and an internal sink vertex. This is illustrated in Figure \ref{fig:MMod}. Hence the number of source vertices in a reduced web is the number of crossings in its $M$-diagram.

Given $i$ we compute $s_i \cdot w$ by appending a flattened crossing to the bottom of $w$ at the $i$ and $(i+1)^{st}$ vertices and resolving as a sum of webs as shown in Figure \ref{fig:smoothings}. Say $s_i\cdot w = w + \widehat{w}$. The web $\widehat{w}$ is a possibly-reducible web with one more source vertex than $w$. If $\widetilde{w} = w$ then $c(\widetilde{M}) = c(M)$ and we are done. Assume for the remainder of the argument that $\widetilde{w}$ is a reduced summand of $\widehat{w}$.

If $\widehat{w}$ is not reduced, it either has a square or a bigon. If $\widehat{w}$ has a bigon then boundary vertices $i$ and $i+1$ were connected to the same internal vertex in $w$. Applying the bigon relation from  Figure \ref{fig:webrelations}, it follows that $s_i \cdot w = w + \widehat{w} = w-2w=-w$. If $\widehat{w}$ has a square then expressing $\widehat{w}$ as a sum of reduced webs requires one or more applications of the square relation on $\mathfrak{sl}_3$ webs shown in Figure \ref{fig:webrelations}. Each application of this relation removes two source vertices, so a reduced summand of $\widehat{w}$ in these cases can have at most $c(M)$ source vertices, and we conclude $c(\widetilde{M})\leq c(M)$.

If $\widehat{w}$ is reduced then $\widehat{w} = \widetilde{w}$. A straightforward case analysis of the depths at $i$ and $i+1$ under the action of $s_i$ shows that the five cases in Figure \ref{fig:crossingcases} are the only possibilities for the relationship between the boundary words of $w$ and $\widetilde{w}$. In Cases 1--3, the changes in depths are completely local and $\widetilde{M}$ has exactly one more crossing than $M$, namely the new crossing introduced by the action of $s_i$. 

Now consider Cases 4 and 5. If the symbols at $i$ and $i+1$ are both $+$ in $w$ then they will be $+0$ in $\widehat{w}$ by analyzing depths along the boundary after acting by $s_i$. Similarly if the symbols at $i$ and $i+1$ are both $-$ in $w$ then they will be $0-$ in $\widehat{w}$.  Thus the boundary word for $\widetilde{w}$ can only be balanced and Yamanouchi if it differs from the boundary word for $w$ in at least one other position. 
\end{proof}

We use this general result to show that the potential counterexample in Figure \ref{fig:sl3Counter} does not violate our conjecture, because $\widetilde{w}$ is not a reduced summand of $s_i \cdot w$ for any $s_i$.

\begin{corollary}
Consider the webs $\widetilde{w}$ and $w$ with data shown in Figure \ref{fig:sl3Counter}. There is no $i$ such that $\widetilde{w}$ is a reduced summand of $s_i \cdot w$. 
\end{corollary}

\begin{proof}
Examining the $M$-diagrams for webs $\widetilde{w}$ and $w$ in Figure \ref{fig:sl3Counter}, we see they have 10 and 9 crossings respectively. Hence their boundary words must differ according to Figure \ref{fig:crossingcases}. Comparing boundary words, we see that the only possibility is that $\widetilde{w}$ is a reduced summand of $s_{10}\cdot w$. However, $s_{10}\cdot w = w+\widehat{w}$ where $\widehat{w}$ has a square face and is therefore reducible. This implies all reduced summands of $s_{10}\cdot w$ have $M$-diagrams with 9 or fewer crossings.
\end{proof}

While we can verify the example shown in Figure \ref{fig:sl3Counter} does not yield a counterexample to Conjecture \ref{conj:shadrank}, we cannot at this point generalize this argument. Using the aid of a computer program, we have crossing, rank, and shadow data for all webs in $\mathcal{W}_{3n}^R$ with $n\leq 7$. Using this data, we have generated a list of potential counterexamples that is given in Figure \ref{fig:counterdata}. 

\begin{figure}[h]
\begin{tabular}{|c|c|c|c|}
\hline
&& & \\
& Count of & Number of pairs $\widetilde{w},w\in SYT(n,n,n)$ such that & Number of these pairs with \\
$n$ &  $SYT(n,n,n)$ &   $\widetilde{w}\prec_S w$ and $r(\widetilde{w})>r(w)$ &  $c(\widetilde{M})\leq c(M)+1$ \\
& & & \\ \hline
1 & 1 & 0 & 0 \\ \hline
2 & 5 & 0 & 0 \\ \hline
3 & 42 & 0 & 0 \\ \hline
4 & 462 & 0 & 0 \\ \hline
5 & 6006  & 0 & 0 \\ \hline
6 & 87516  & 660  & 446  \\ \hline
7 & 1385670 & 62147  & 40865   \\ \hline
\end{tabular}
\caption{Number of potential counterexamples to Conjecture \ref{conj:shadrank}}\label{fig:counterdata}
\end{figure}

\bibliographystyle{plain}
\bibliography{references}

\end{document}